\documentclass[11pt,a4paper,twoside]{article}

\usepackage[french,english]{babel}
\usepackage{amsfonts}
\usepackage{amsmath,amssymb,amscd}
\usepackage{mathrsfs}  
\usepackage{amsthm}
\usepackage{a4wide}   
\usepackage[T1]{fontenc} 
\usepackage{newlfont} 
\usepackage{color}
\usepackage{stackrel}

\newcommand{\cache}[1]{}

\usepackage[
backref=page,bookmarksopen=true]{hyperref} 

\usepackage[all]{xy}  

\usepackage[applemac]{inputenc}       

\def\choixcompteur{subsection}
\newtheorem{theo}[\choixcompteur]{Theorem}

\newtheorem{prop}[\choixcompteur]{Proposition}
\newtheorem{lemm}[\choixcompteur]{Lemma}

\theoremstyle{definition}

\newtheorem{defi}[\choixcompteur]{Definition}

\newtheorem{rema}[\choixcompteur]{Remark}
\newtheorem{remas}[\choixcompteur]{Remarks}

\newtheorem*{exem*}{Example}
\newtheorem*{exems*}{Examples}
\newtheorem*{exam*}{Example}
\newtheorem*{exams*}{Examples}
\newtheorem*{rema*}{Remark}
\newtheorem*{remas*}{Remarks}
\newtheorem*{NB}{N.B}

\theoremstyle{definition}
\newtheorem*{defi*}{Definition}
\newtheorem*{defiprop*}{Definition-Proposition}

\theoremstyle{plain}
\newtheorem*{prop*}{Proposition}
\newtheorem*{lemm*}{Lemma}
\newtheorem*{coro*}{Corollary}
\newtheorem*{theo*}{Theorem}

 \def\cdr@enoncedef{%
 \newenvironment{enonce*}[2][plain]%
 {\let\cdrenonce\relax \theoremstyle{##1}%
 \newtheorem*{cdrenonce}{##2}%
 \begin{cdrenonce}}%
 {\end{cdrenonce}}   }%

 \AtBeginDocument{%
    \cdr@enoncedef}

\def\cf{{\it cf.\/}\ }
\def\ie{{\it i.e.\/}\ }
\def\eg{{\it e.g.\/}\ }

\def\lc{{\it l.c.\/}\ }

\def\resp{{\it resp.,\/}\ }

\def\truc{\unskip\kern 3pt\penalty 500
\hbox{\vrule\vbox to 5pt{\hrule width 4pt\vfill\hrule}\vrule}\kern
3pt}

\def\vect{\overrightarrow}

\def\parni{\par\noindent}

\def\ed{ editor}

\def\N{{\mathbb N}}    
\def\Z{{\mathbb Z}}

\def\R{{\mathbb R}}
\def\C{{\mathbb C}}

\def\A{{\mathbb A}}
\def\M{{\mathbb M}}

\def\G{{\mathbb G}}

\newcommand{\g}[1]{\mathfrak{#1}} 

\def\qa{\alpha}     
\def\qb{\beta}
\def\qd{\delta}
\def\qe{\varepsilon}
\def\qf{\varphi}
\def\qg{\gamma}
\def\qi{\iota}
 \def\qk{\kappa}
 \def\ql{\lambda}
\def\qm{\mu}
\def\qn{\nu}

\def\qp{\pi}
\def\qr{\rho}
\def\qs {\sigma}

 \def\qth{\theta}

\def\QD{\Delta}
\def\QF{\Phi}
\def\QG{\Gamma}
\def\QL{\Lambda}
\def\QO{\Omega}

\def\sha{{\mathcal A}}   

\def\shc{{\mathcal C}}

\def\shh{{\mathcal H}}

\def\shk{{\mathcal K}}

\def\shm{{\mathcal M}}

\def\sho{{\mathcal O}}

\def\shq{{\mathcal Q}}

\def\shs{{\mathcal S}}
\def\sht{{\mathcal T}}

\def\SHC{{\mathscr C}}

\def\SHE{{\mathscr E}}
\def\SHF{{\mathscr F}}

\def\SHI{{\mathscr I}}
\def\SHJ{{\mathscr J}}

\def\SHT{{\mathscr T}}

\def\HT{\text{ht\ }}

\begin{document}

\title{Macdonald's formula for Kac-Moody groups \goodbreak over local fields}
\author{Nicole Bardy-Panse, St\'ephane Gaussent and Guy Rousseau
 \footnote{keywords: Masure, building, Hecke algebras, Bernstein-Lusztig relations, Kac--Moody group, ultrametric  local field, Satake isomorphism, Macdonald's formula}
 \footnote{Classification: 20G44, 20E42, 20C08(primary), 17B67, 20G05, 20G25, 22E50, 20E42, 22E65, 51E24, 33D80}
}

\date{{September 12, 2017}}

%




\maketitle



\begin{abstract} {For an almost split Kac-Moody group $G$ over a local non-archimedean field, the last two authors constructed a spherical Hecke algebra $^s\shh$ (over $\C$, say) and its Satake isomorphism $\shs$ with the commutative algebra $\C[[Y]]^{W^v}$ of Weyl invariant elements in some formal series algebra  $\C[[Y]]$.
In this article we prove a Macdonald's formula, \ie an explicit formula for the image $\shs(c_\ql)$ of a basis element of $^s\shh$. 
The proof involves geometric arguments in the masure associated to $G$ and algebraic tools, including the Cherednik's representation of the Bernstein-Lusztig-Hecke algebra (introduced in a previous article) and the Cherednik's identity between some symmetrizers.
}
\end{abstract}


\setcounter{tocdepth}{1}    
\tableofcontents

\section*{Introduction}
\label{seIntro}

The Macdonald's formula, which is discussed here, is the one giving the image of the Satake isomorphism between the spherical Hecke algebra and the ring of invariant symmetric functions in the context of Kac-Moody groups. If the group is semisimple, the formula was proven by Macdonald \cite{Ma71}. In the case of a split {untwisted} affine Kac-Moody group, it is a result of Braverman, Kazhdan and Patnaik \cite{BrKP15}. 

Let $\mathcal K$ be a local field, denote its ring of integers by $\mathcal O$, fix a uniformizer {$\varpi$} and suppose that the residue field is of cardinality $q<+\infty$. For this introduction, let $G = G(\mathcal K)$ be a (split) minimal Kac-Moody group over $\mathcal K${, as defined by Tits \cite{T87}}.
Let $T\subset G$ be a maximal torus, denote the $\mathbb Z$-lattice of coweights $Hom (\mathcal K^*, T)$ by $Y$, {the (dual) $\mathbb Z$-lattice of weights $Hom (T,\mathcal K^*)$ by $X$,} the $\mathbb Z$-lattice of coroots by $Q^\vee$, and the real roots of $(G,T)$ by $\Phi$. Finally, let $W^v$ be the Weyl group of $(G,T)$.

To these data, the third author, in {\cite{R12}}, associates a masure $\SHI = \SHI(G,\mathcal K)$, which is a generalization of the Bruhat-Tits building. The masure $\SHI$ is covered by apartments, all isomorphic to the standard one $\mathbb A = Y\otimes_{\mathbb Z}\mathbb R$. The group $G$ acts on $\SHI$ such that $Stab_G(0) = G(\mathcal O)$. The preorder associated to the positive Tits cone in $\mathbb A$ extends to a $G$-invariant preorder $\geq$ on $\SHI$. Then $G^+ = \{g\in G\mid g\cdot 0 \geq 0\}$ is a subsemigroup of $G$ and the last two authors show in \cite{GR14} that 
$$
G^+ = \bigsqcup_{\lambda\in Y^{++}} G(\mathcal O) {\varpi^{-\lambda}} G(\mathcal O),
$$ 
where $Y^{++}$ is the set of dominant coweights (for a choice of a Borel subgroup $B$ in $G$ {containing $T$}) and {$\varpi^\lambda$ is the image of $\varpi$} by $\lambda$. 
Note that $G^+ = G$ if, and only if, $G$ is finite dimensional. Let $c_\lambda$ be the characteristic function of the double coset $G(\mathcal O)  {\varpi^{-\lambda}}  G(\mathcal O)$. 
The spherical Hecke algebra is the set
$$
^s\mathcal H = \Big\{\varphi = \sum_{\lambda\in Y^{++}} a_\lambda c_\lambda \mid supp(\varphi) \subset \cup_{i=1}^n \mu_i-Q_+^\vee, \ \mu_i \in Y^{++}{, a_\ql\in R}\Big\}
$$ 
where {$R$ is any ring containing $\Z[q^{\pm 1}]$ and }the algebra structure is given by the convolution product, see Section \ref{s2} for more details. 
Consider now the Looijenga's coweights algebra ${R}[[Y]]$, with the same kind of support condition as above. The second and third authors show in \cite{GR14}, using an extended tree inside the masure $\SHI$, the following theorem.

\medskip
\textbf{Theorem.}
The spherical Hecke algebra $^s\mathcal H$ is isomorphic, via the Satake isomorphism $\mathcal S$, to the commutative algebra ${R}[[Y]]^{W^v}$ of Weyl invariant elements in ${R}[[Y]]$.

\medskip
Now, to compute the image of $c_\lambda$ by the Satake isomorphism, we use essentially an idea of Casselmann \cite{Ca80} (see also \cite{HaKP10}): decompose  $\mathcal S(c_\lambda)$ as a sum (indexed by the Weyl group {$W^v$}) of more simple elements $J_w(\lambda)$ and then compute these  $J_w(\lambda)$. 
Originally this decomposition was obtained using intertwinning operators; this same idea was still in use in \cite{BrKP15}, for the affine case. 
Here we get this decomposition thanks our interpretation of $\mathcal S(c_\lambda)$ in terms of paths in the masure.
{ Namely $\mathcal S(c_\lambda)$ is a sum of some terms indexed by the ``Hecke paths $\qp$ of type $\ql$'' starting from $0$ in the standard apartment $\A$.
We define $J_w(\lambda)$ as the same sum indexed by the paths $\qp$ satisfying moreover $\qp'_+(0)=w.\ql$.
Then we are able to prove a recursive formula for the $J_w(\lambda)$, using the paths, retracting onto these Hecke paths, that are in some extended trees contained in the masure $\SHI$.
}

{

Independently, in a more algebraic way, we introduce the ring $\Z[\qs^{\pm 1}]$ where $\qs$ is an indeterminate and consider the Bernstein-Lusztig-Hecke algebra $^{BL}\mathcal H$,  which was defined in \cite{BPGR16} as an abstract (and bigger)  version of the affine Iwahori-Hecke algebra of the Kac-Moody group. 
More precisely, $^{BL}\mathcal H$ admits (as module) a $\Z[\qs^{\pm 1}]-$basis denoted $(Z^\ql T_w)_{\ql\in Y, w\in W^v}$, which satisfies some multiplication relations, including one of Bernstein-Lusztig type.
We define an algebra embedding of $^{BL}\mathcal H$ into $\Z[\qs^{\pm1}](Y)[W^v]$, where $\Z[\qs^{\pm1}](Y)$ is an algebra consisting of some infinite formal sums $\sum_{\ql\in Y}a_\ql e^\ql $ with $a_\ql\in \Z[\qs^{\pm 1}]$. 
So that we can see the elements of the Bernstein-Lusztig-Hecke algebra as written $ \sum_{w\in F} f_w[w]$ with $ F$ a finite subset of $ W^v$ and  $f_w$ in  $\Z[\qs^{\pm 1}] (Y)$. 

Further, the algebra $\Z[\qs^{\pm1}](Y)[W^v]$ acts on $\Z[\qs^{\pm1}](Y)$. In fact, with a specialization at $\qs^2=q^{-1}$, we have $J_w(\ql)=q^{\qr(\ql)}.(T_w(e^\ql))_{(\qs^2=q^{-1})} $, with $\rho\in X$  a weight taking value $1$ on each simple coroot and $T_w = Z^0T_w$. This equality is true because the right-hand side satisfies the same recursive formula as the $J_w(\lambda)$.

In order to consider the analogue of the sum $\mathcal S(c_\ql)$ in the ``same'' algebraic way.  We introduce some algebraic symmetrizers and the link between them (called Cherednik's identity)  will be a key to obtain the Macdonald's formula. In all this part of the article, we have to pay attention since we are dealing with infinite sums or products. So there are some completions to be considered, in order that these expressions make sense.

In some completion  of $\Z[\qs^{\pm 1}] (Y)$, set
 $$\Delta = \prod_{\alpha\in\Phi_+}\frac{1 - \qs^{2}e^{-\alpha^\vee}}{1- e^{-\alpha^\vee}}, \text{ and, for  }w\in W^v, \quad  ^w\Delta = \prod_{\alpha\in\Phi_+}\frac{1 - \qs^{2}e^{-w\alpha^\vee}}{1- e^{-w\alpha^\vee}}.
 $$
Then, the $T-$symmetrizer $P_\qs=\sum_{w\in W^v}T_w$ and the $\Delta-$symmetrizer $\SHJ_\qs = \sum_{w\in W^v}^w\Delta[w]$ can be seen as formal expressions $\sum_{w\in W^v} a_w [w]$  with $a_w$ in the above completion.
 In order to justify that the expression of $P_\qs$ makes sense, we use a geometrical argument on galleries, whereas the explanation given for $\SHJ_\qs$ is more classical.

The generalization of  theorem 2.18 in \cite{CheM13} is then achieved and we get the  Cherednik's identity:

 \textbf{Cherednik's identity.}
 $$ P_\qs=\sum_{w\in W^v} \qs_wH_w=\mathfrak{m}_\qs\!\SHJ_\qs$$ 
with  $\mathfrak{m}_\qs$ ``well-defined'',  invertible and $ W^v-$invariant in $\Z[\qs^{\pm 1}] [[Y]]$.

\par This was also proved by Braverman, Kazhdan and Patnaik \cite{BrKP15} (\resp Patnaik and Pusk\'as \cite {PaP17}) in the case of an untwisted, affine (\resp split, symmetrizable) Kac-Moody group.

Finally, we consider the Poincaré series $W'(\qs^2) = \sum_{w\in W'} \qs^{2\ell(w)}$  for $W'$ any subgroup of $W^v$.
  The element $P_\qs(e^\ql)$  is well defined in some greater completion and (up to the factor $W^v_\lambda(\qs^2)$ with $W^v_\lambda = Stab_{W^v}(\lambda)$) plays in this abstract context the role of $S(c_\ql))$.
   We can also consider $\SHJ_\qs (e^\ql)$ and we get $\mathfrak{m}_\qs=\frac{W^v(\qs^2)}{\sum_{w\in W^v}\,{^w\Delta}}\in \Z[\qs^{\pm 1}] [[Y]]$.
   We have also $\frac{\sum_{w\in W^v}\,{^w\Delta}.e^{w\lambda}}{W^v_\lambda(\qs^2)}\in  \Z[\qs^{\pm 1}] [[Y]]$, so we can specialize at $\qs^2=q^{-1}$ and we get the Macdonald's formula:
   
 \medskip
 \textbf{Macdonald's formula.}
Let $\lambda\in Y^{++}$. Then 
$$ 
S(c_\lambda)= q^{\rho(\lambda)}\bigg(\frac{W^v(\qs^2)}{\sum_{w\in W^v}\,{^w\Delta}}\bigg)_{(\qs^2=q^{-1})}  \bigg(\frac{\sum_{w\in W^v}\,{^w\Delta}.e^{w\lambda}}{W^v_\lambda(\qs^2)}\bigg)_{(\qs^2=q^{-1})}.
$$ 

}



\medskip
The right hand side of the equality is the Hall-Littlewood polynomial $P_\lambda(t)$ defined by Viswanath \cite{Vi08} in this context, where its $t$ corresponds to our $q^{-1}$. In particular, $P_\lambda(0)$ is the character of the irreducible representation $V(\lambda)$ of highest weight $\lambda$ of the Langlands dual Kac-Moody group.

Finally, the factor {$\mathfrak m = (\mathfrak m_\qs)_{\qs^2=q^{-1}}$} is equal to $1$ in the finite dimensional setting. In the affine setting, $\mathfrak m$ has an expression as an infinite product involving the minimal positive {imaginary} coroot and the exponents of the underlying semisimple group. In the general Kac-Moody case, no such formula is known.

Actually, our {Macdonald's formula} is still valid in the more general framework of an abstract masure as defined in \cite{R11}. In particular, we can deal with the case of an almost split Kac-Moody group over a local field and we may have unequal parameters in the definition of the Hecke algebras.

\medskip
\textbf{Content of the paper.}
In Section 1, we introduce the setting of the article, recalling several notation and definitions, including the one of the masure. Section 2 introduces the kind of algebras we are dealing with and the Satake isomorphism is recalled. The computation of the recurrence relation on the components $J_w(\lambda)$ is performed in Section 3. Then, we turn to the algebraic part of the paper: in Section 4, we introduce the Bernstein-Lusztig-Hecke algebra and study some of its representations. The symmetrizers are defined and studied in Sections 5 and 6. In particular, we explain how we cope with the questions of their well-definedness. Finally, Section 7 contains the main theorem of the paper: the Macdonald's formula for the Satake isomorphism.

\section{General framework}\label{s1}

\subsection{Vectorial data}\label{1.1}  We consider a quadruple $(V,W^v,(\qa_i)_{i\in I}, (\qa^\vee_i)_{i\in I})$ where $V$ is a finite dimensional real vector space, $W^v$ a subgroup of $GL(V)$ (the vectorial Weyl group), $I$ a finite set, $(\qa^\vee_i)_{i\in I}$ a family in $V$ and $(\qa_i)_{i\in I}$ a  family in the dual $V^*$.
{We suppose these families free, \ie the sets  $\{\qa_i\mid i\in I\}$ $\{\qa^\vee_i\mid i\in I\}$ linearly independent. We ask moreover} these data to satisfy the conditions of \cite[1.1]{R11}.
  In particular, the formula $r_i(v)=v-\qa_i(v)\qa_i^\vee$ defines a linear involution in $V$ which is an element in $W^v$ and $(W^v,\{r_i\mid i\in I\})$ is a Coxeter system. {We consider also the dual action of $W^v$ on $V^*$.}

  \par To be more concrete, we consider the Kac-Moody case of [\lc; 1.2]: the matrix $\M=(\qa_j(\qa_i^\vee))_{i,j\in I}$ is a generalized Cartan matrix.
  Then $W^v$ is the Weyl group of the corresponding Kac-Moody Lie algebra $\g g_\M$ and the associated real root system is
  
$$
\QF=\{w(\qa_i)\mid w\in W^v,i\in I\}\subset Q=\bigoplus_{i\in I}\,\Z.\qa_i.
$$ 
We set 
 $\QF^\pm{}=\QF\cap Q^\pm{}$ where $Q^\pm{}=\pm{}(\bigoplus_{i\in I}\,(\Z_{\geq 0}).\qa_i)$.
 Also $Q^\vee:=\bigoplus_{i\in I}\,\Z.\qa_i^\vee$
and $Q^\vee_\pm{}=\pm{}(\bigoplus_{i\in I}\,(\Z_{\geq 0}).\qa_i^\vee)$.
   We have  $\QF=\QF^+\cup\QF^-$
   and, for $\qa=w(\qa_i)\in\QF$, $r_\qa=w.r_i.w^{-1}$ and $r_\qa(v)=v-\qa(v)\qa^\vee$, where the coroot $\qa^\vee=w(\qa_i^\vee)$ depends only on $\qa$.

\par The set  $\QF_{}$ is an (abstract) reduced real root system in the sense of \cite{MP89}, \cite{MP95} or \cite{Ba96}.
We shall occasionally use {imaginary roots: $\QF_{im}=\QF^+_{im}\sqcup\QF^-_{im}$} with 
 $-\QF^-_{im}=\QF^+_{im}\subset Q^+$,  $W^v-$stable.
The set  $\QF_{all}=\QF_{}\sqcup\QF_{im}$ of all roots {has to be} an (abstract) root system in the sense of \cite{Ba96}.
{An example for $\QF_{all}$ is the full set of roots of $\g g_\M$.}

  \par The {\it fundamental positive chamber} is $C^v_f=\{v\in V\mid\qa_i(v)>0,\forall i\in I\}$.
   Its closure $\overline{C^v_f}$ is the disjoint union of the vectorial faces $F^v(J)=\{v\in V\mid\qa_i(v)=0,\forall i\in J,\qa_i(v)>0,\forall i\in I\setminus J\}$ for $J\subset I$. We set $V_0 = F^v(I){=V^{W^v}}$.
    The positive (resp. negative) vectorial faces are the sets $w.F^v(J)$ (resp. $-w.F^v(J)$) for $w\in W^v$ and $J\subset I$.
    The support of such a face is the vector space it generates.
    The set $J$ or the face $w.F^v(J)$ or an element of this face is called {\it spherical} if the group $W^v(J)$ generated by $\{r_i\mid i\in J\}$ is finite.
    An element of a vectorial chamber $\pm w.C^v_f$ is called {\it regular}.

    \par The {\it Tits cone}  $\sht$ (resp. its interior $\sht^\circ$) is the (disjoint) union of the positive (resp. and spherical) vectorial faces. It is a $W^v-$stable convex cone in $V$.
    Actually $W^v$ permutes the vectorial walls $M^v(\qa)=\ker(\qa)$ (for $\qa\in\QF$); it acts simply transitively on the positive (resp. negative) vectorial chambers.

  \par
  We say that $\A^v=(V,W^v)$ is a {\it vectorial apartment}.

\subsection{The model apartment}\label{1.2} 
As in \cite[1.4]{R11} the model apartment $\A$ is $V$ considered as an affine space and endowed with a family $\shm$ of walls. 
 These walls  are affine hyperplanes directed by $\ker(\qa)$ for $\qa\in\QF$.

 \par{\bf 1)} We ask this apartment to be {\bf semi-discrete} and the origin $0$ to be {\bf special}.
  This means that these walls are the hyperplanes defined as follows:
$$M(\qa,k)=\{v\in V\mid\qa(v)+k=0\}\qquad\text{for }\qa\in\QF_{}\text{ and } k\in\QL_\qa,$$ with $\QL_\qa=\QL_{-\qa}=d_\qa.\Z$ a non trivial discrete subgroup of $\R$.
\ Using Lemma 1.3 in \cite{GR14}, (\ie replacing $\QF$ by another system $\QF_1$), we may (and shall) assume that $\QL_\qa=\Z$ for all $\qa\in\QF$.

  \par For $\qa=w(\qa_i)\in\QF$, $k\in \Z$ and $M=M(\qa,k)$, the reflection $r_{\qa,k}=r_M$ with respect to $M$ is the affine involution of $\A$ with fixed points the wall $M$ and associated linear involution $r_\qa$.
   The affine Weyl group $W^a$ is the group generated by the reflections $r_M$ for $M\in \shm$; we assume that $W^a$ stabilizes $\shm$.
We know that $W^a=W^v\ltimes Q^\vee$ and we write $W^a_\R=W^v\ltimes V$; here $ Q^\vee$ and $V$ have to be understood as groups of translations.
One should notice that $Q^\vee$ (which is relative to the above new root system $\QF_1${, denoted $\QF$ below}) is actually completely determined by $\A$: it is the group $Y_{W^a}$ of all translations in the affine Weyl group $W^a$.
 
 \par To define enclosures, we shall also need some imaginary walls $M(\qa,k)$ for $\qa\in\QF_{im}$ and $k\in\QL_\qa$, where $\QL_\qa=-\QL_{-\qa}$ is a subset of $\R$.
 We ask $W^a$ to stabilize this set of imaginary walls.
 In many cases, \eg for split Kac-Moody groups over local fields with normalized valuation, we may suppose  $\QF$ (resp. $\QF_{im}$) is the set of real (resp. imaginary) roots of $\g g_\M$ and $\QL_\qa=\Z$, for all $\qa\in\QF_{all}$.

   \par{\bf 2)} An automorphism of $\A$ is an affine bijection $\qf:\A\to\A$ stabilizing the set of pairs $(M,\qa^\vee)$ of a wall $M$ and the coroot $\qa^\vee$ associated with $\qa\in\QF_{}$ such that $M=M(\qa,k)$, $k\in \Z$. 
   We write $\vect\qf:V\to V$ the linear application associated to $\qf$.
   The group $Aut(\A)$ of these automorphisms contains $W^a$ and normalizes it.
We consider also the group $Aut^W_\R(\A)=\{\qf\in Aut(\A)\mid\vect{\qf}\in W^v\}=Aut(\A)\cap W^a_\R$ { of vectorially Weyl automorphisms.
One has $Aut^W_\R(\A)=W^v\ltimes P^\vee$, where $ P^\vee=\{ v\in V\mid \qa(v)\in \Z, \forall\qa\in\QF \}$. }

   \par{\bf 3)} For $\qa\in{\QF_{all}}$ and $k\in\R$, $D(\qa,k)=\{v\in V\mid\qa(v)+k\geq 0\}$ is an half-space, it is called an {\it half-apartment} if {$k\in \Z$ and $\qa\in\QF$}. We write  $D(\alpha,\infty) = \mathbb A$.


The Tits cone $\mathcal T$ 
and its interior $\mathcal T^o$ are convex and $W^v-$stable cones, therefore, we can define two $W^v-$invariant preorder relations  on $\mathbb A$: 
$$
x\leq y\;\Leftrightarrow\; y-x\in\mathcal T
; \quad x\stackrel{o}{<} y\;\Leftrightarrow\; y-x\in\mathcal T^o.
$$
 If $W^v$ has no  fixed point in $V\setminus\{0\}$ {(\ie $V_0=\{0\}$)} and no finite factor, then they are orders; but, in general, they are not.


\subsection{Faces, sectors ...} 
\label{1.2b}

 The faces in $\mathbb A$ are associated to the above systems of walls
and half-apartments
. As in \cite{BrT72}, they
are no longer subsets of $\mathbb A$, but filters of subsets of $\mathbb A$. For the definition of that notion and its properties, we refer to \cite{BrT72} or \cite{GR08}.

If $F$ is a subset of $\mathbb A$ containing an element $x$ in its closure,
the germ of $F$ in $x$ is the filter $\mathrm{germ}_x(F)$ consisting of all subsets of $\mathbb A$ which contain intersections of $F$ and neighbourhoods of $x$. 
 {We say that $x$ is the {\it origin} of this germ.}
In particular, if $x\neq y\in \mathbb A$, we denote the germ in $x$ of the segment $[x,y]$ (resp. of the interval $]x,y]=[x,y]\setminus\{x\}$) by $[x,y)$ (resp. $]x,y)$).
{ \par For $y≠x$, the segment germ $[x,y)$ is called of sign $\pm$ if $y-x\in\pm\sht$.
The segment $[x,y]$ or the segment germ $[x,y)$} is called {\it preordered} if $x\leq y$ or $y\leq x$ and {\it generic} if $x\stackrel{o}{<} y$ or $y\stackrel{o}{<} x$.

Given $F$ a filter of subsets of $\mathbb A$, its {\it enclosure} $cl_{\mathbb A}(F)$ (resp. {\it closure} $\overline F$) is the filter made of  the subsets of $\mathbb A$ containing an element of $F$ of the shape $\cap_{\alpha\in\QF_{all}}D(\alpha,k_\alpha)$, where $k_\alpha\in {\QL_\qa}\cup\{\infty\}$ (resp. containing the closure $\overline S$ of some $S\in F$).

\medskip

A {\it local face} $F$ in the apartment $\mathbb A$ is associated
 to a point $x\in \mathbb A$, its vertex {or origin}, and a  vectorial face $F^v{=:\vect F}$ in $V$, its direction. It is defined as $F=germ_x(x+F^v)$ and we denote it by $F=F^\ell(x,F^v)$.
 Its closure is $\overline{F^\ell}(x,F^v)=germ_x(x+\overline{F^v})$.
 {Its sign is the sign of $F^v$.}

There is an order on the local faces: the assertions ``$F$ is a face of $F'$ '',
``$F'$ covers $F$ '' and ``$F\leq F'$ '' are by definition equivalent to
$F\subset\overline{F'}$.
 The dimension of a local face $F$ is the smallest dimension of an affine space generated by some $S\in F$.
  The (unique) such affine space $E$ of minimal dimension is the support of $F$; if $F=F^\ell(x,F^v)$, $supp(F)=x+supp(F^v)$.
 A local face $F=F^\ell(x,F^v)$ is spherical if the direction of its support meets the open Tits cone (\ie if $F^v$ is spherical), then its  pointwise stabilizer $W_F$ in $W^a$ is finite.

\medskip
  For any local face $F^\ell=F^\ell(x,F^v)$, there is a unique face $F$ (as defined in \cite{R11}) containing $F^\ell$.
 Then $\overline{F^\ell}\subset\overline F=cl_\A(F^\ell)=cl_\A(F)$ is also the enclosure of any interval-germ $]x,y)=germ_x(]x,y])$ included in $F^\ell$.
 
 \par We shall actually here speak only of local faces, and sometimes forget the word local.

\medskip
 A {\it local chamber}  is a maximal local face, \ie a local face $F^\ell(x,\pm w.C^v_f)$ for $x\in\A$ and $w\in W^v$.
 The {\it fundamental local chamber} {of sign $\pm$} is $C_0^\pm=germ_0(\pm C^v_f)$.

A {\it (local) panel} is a spherical local face maximal among local faces which are not chambers, or, equivalently, a spherical face of dimension $n-1$. Its support is a wall.

\medskip
 A {\it sector} in $\mathbb A$ is a $V-$translate $\mathfrak s=x+C^v$ of a vectorial chamber
$C^v=\pm w.C^v_f$, $w \in W^v$. The point $x$ is its {\it base point} and $C^v{=:\vect{\g s}}$ its  {\it direction}.  Two sectors have the same direction if, and only if, they are conjugate
by $V-$translation,
 and if, and only if, their intersection contains another sector.

 The {\it sector-germ} of a sector $\mathfrak s=x+C^v$ in $\mathbb A$ is the filter $\mathfrak S$ of
subsets of~$\mathbb A$ consisting of the sets containing a $V-$translate of $\mathfrak s$, it is well
determined by the direction $C^v{=\vect{\g s}=:\vect{\g S}}$. So, the set of
translation classes of sectors in $\mathbb A$, the set of vectorial chambers in $V$ and
 the set of sector-germs in $\mathbb A$ are in canonical bijection.
  We denote the sector-germ associated to the  fundamental vectorial chamber $\pm C^v_f$ by $\g S_{\pm\infty}$.

 A {\it sector-face} in $\mathbb A$ is a $V-$translate $\mathfrak f=x+F^v$ of a vectorial face
$F^v=\pm w.F^v(J)$. The sector-face-germ of $\mathfrak f$ is the filter $\mathfrak F$ of
subsets containing a translate $\mathfrak f'$ of $\mathfrak f$ by an element of $F^v$ ({\it i.e.} $\mathfrak
f'\subset \mathfrak f$). If $F^v$ is spherical, then $\mathfrak f$ and $\mathfrak F$ are also called
spherical. The sign of $\mathfrak f$ and $\mathfrak F$ is the sign of $F^v$.

\cache{ 
\medskip
A {\it chimney} in $\mathbb A$ is associated to a face $F=F(x, F_0^v)$, called its basis, and to a vectorial face $F^v$, its direction, it is the filter
$$
\mathfrak r(F,F^v) = cl_{\mathbb A}(F+F^v).
$$ A chimney $\mathfrak r = \mathfrak r(F,F^v)$ is {\it splayed} if $F^v$ is spherical, it is {\it solid} if its support (as a filter, i.e. the smallest affine subspace containing $\mathfrak r$) has a finite  pointwise stabilizer in $W^v$. A splayed chimney is therefore solid. The enclosure of a sector-face $\mathfrak f=x+F^v$ is a chimney.
} 


 \subsection{The masure}\label{1.3}

 In this section, we recall some properties of a masure {as defined} in \cite{R11}.

\medskip
\parni{\bf 1)} An apartment of type $\mathbb A$ is a set $A$ endowed with a set $Isom^W\!(\mathbb A,A)$ of bijections (called Weyl-isomorphisms) such that, if $f_0\in Isom^W\!(\mathbb A,A)$, then $f\in Isom^W\!(\mathbb A,A)$ if, and only if, there exists $w\in W^a$ satisfying $f = f_0\circ w$.
An isomorphism (resp. a Weyl-isomorphism, a vectorially-Weyl isomorphism) between two apartments $\varphi :A\to A'$ is a bijection such that, for any $f\in Isom^W\!(\mathbb A,A)$, $f'\in Isom^W\!(\mathbb A,A')$, {we have} $f'^{-1}\circ\qf\circ f\in Aut(\A)$ (resp. $\in W^a$, $\in Aut^W_\R(\A)$); the {set} of these isomorphisms is written $Isom(A,A')$ (resp. $Isom^W(A,A')$, $Isom^W_\R(A,A')$).
 As the filters in $\A$ defined in \ref{1.2b} above (\eg local faces, sectors, walls,..) are permuted by $Aut(\A)$, they are well defined in any apartment of type $\A$ and exchanged by any isomorphism.

\medskip
\par An {\it ordered affine hovel} (or {\it masure} for short) of type $\mathbb A$ is a set $\SHI$ endowed with a covering $\mathcal A$ of subsets  called apartments, each endowed with some structure of an apartment of type $\A$.
 We do not recall here the precise definition, but indicate some of the main properties:


\par {\bf a)} If $F$ is a point, a preordered segment,  a local face or a spherical sector face in an apartment $A$ and if $A'$ is another apartment containing $F$, then $A\cap A'$ contains the
enclosure $cl_A(F)$ of $F$ and there exists a Weyl-isomorphism from $A$ onto $A'$ fixing $cl_A(F)$.
\par A filter or subset in $\SHI$ is called a preordered segment, a preordered segment germ, a local face, a spherical sector face or a spherical sector face germ if it is included in some apartment $A$ and is called like that in $A$.

\par {\bf b)} If $\mathfrak F$ is  the germ of a spherical sector face and if $F$ is a face or a germ of a spherical sector face, then there exists an apartment that contains $\mathfrak F$ and $F$.

\par {\bf c)}  If two apartments $A,A'$ contain $\mathfrak F$ and $F$ as in {\bf b)}, then their intersection contains $cl_A(\mathfrak F\cup F)$ and there exists a Weyl-isomorphism from $A$ onto $A'$ fixing $cl_A(\mathfrak F\cup F)$;

\par {\bf d)} We consider the relations $≤$ and $\stackrel{o}{<}$ on $\SHI$ defined as follows:
$$x≤y \text{ (resp. } x\stackrel{o}{<}y\ ) \iff \exists A\in\sha \text{ such that }x,y\in A\text{ and } x≤_A y  \text{ (resp. } x\stackrel{o}{<}_Ay\ ) $$
then $≤$ and $\stackrel{o}{<}$ are well defined preorder relations, in particular transitive.

\par {\bf e)}  We ask here $\SHI$ to be thick of {\bf finite thickness}: the number of local chambers  containing a given (local) panel has to be finite $\geq 3$.
     This number is the same for any panel in a given wall $M$ \cite[2.9]{R11}; we denote it by $1+q_M$.

\par {\bf f)}  An automorphism (resp. a Weyl-automorphism, a vectorially-Weyl automorphism) of $\SHI$ is a bijection $\qf:\SHI\to\SHI$ such that $A\in\sha\iff \qf(A)\in\sha$ and then $\qf\vert_A:A\to\qf(A)$ is an isomorphism (resp. a Weyl-isomorphism, a vectorially-Weyl isomorphism).
{We write $Aut(\SHI)$ (resp. $Aut^W(\SHI)$, $Aut^W_\R(\SHI)$) the group of these automorphisms.}
\medskip
\parni{\bf 2)} For $x\in\SHI$, the set $\sht^+_x\SHI$ (resp. $\sht^-_x\SHI$) of {positive (resp. negative)} segment germs $[x,y)$ may be considered as a building, the positive (resp. negative) tangent building. The corresponding faces are the {positive (resp. negative)}  local faces {of} vertex $x$. The associated Weyl group is $W^v$.
 If the $W-$distance (calculated in $\sht^\pm_x\SHI$) of two local chambers is $d^W(C_x,C'_x)=w\in W^v$, to any reduced decomposition $w=r_{i_1}\cdots r_{i_n}$ corresponds a unique minimal gallery from $C_x$ to $C'_x$ of type $(i_1,\cdots,i_n)$. We shall say, {by abuse of notation}, that this gallery is of type $w$.

 \par The buildings $\sht^+_x\SHI$ and $\sht^-_x\SHI$ are actually twinned. The codistance $d^{*W}(C_x,D_x)$ of two opposite sign chambers $C_x$ and $D_x$ is the $W-$distance $d^W(C_x, op D_x)$, where $op D_x$ denotes the opposite chamber to $D_x$ in an apartment containing $C_x$ and $D_x$.

\begin{enonce*}[plain]{3) Lemma}
  \cite[2.9]{R11} Let $D$ be an half-apartment in $\SHI$ and $M=\partial D$ its wall (\ie its boundary).
 One considers a panel $F$ in $M$ and a local chamber $C$ in $\SHI$ covering $F$.
 Then there is an apartment containing $D$ and $C$.
 \end{enonce*}

\parni{\bf 4)}  We assume that {there is a group $G$ acting strongly transitively on $\SHI$ (by automorphisms)},
\ie all isomorphisms involved in {1.a or 1.c above} are induced by elements of $G$, \cf \cite[4.10]{R13} and \cite{CiMR17}.
  We choose in $\SHI$ a fundamental apartment which we identify with $\A$.
   As $G$ is strongly transitive, the apartments of $\SHI$ are the sets $g.\A$ for $g\in G$. 
  { If $N$ is the stabilizer of $\A$ in $G$, there is an homomorphism $\qn:N\to Aut(\A)$ and the  group $W=\qn(N)$ of affine automorphisms of $\A$}
   permutes the walls, local faces, sectors, sector-faces... and contains the affine Weyl group $W^a=W^v\ltimes Q^\vee$ \cite[4.13.1]{R13}.

\par   We denote the  stabilizer of $0\in\A$ in $G$ by $K$ and the pointwise stabilizer (or fixer) of {$C_0^+$ (resp. $C_0^-$) by $K_I=K^+_I$ (resp.$K_I^-$); this group $K_I^\pm$ is called the {\it fundamental Iwahori subgroup} of sign $\pm$.
More generally we write $G_\QO$ the (pointwise) fixer in $G$ of a subset or filter $\QO$ in $\SHI$.}

\medskip
\parni{\bf 5)}  We ask  $W=\qn(N)$ to be {\bf vectorially-Weyl} for its action on the vectorial faces. {With this hypothesis we know that the preorders on $\SHI$ are $G-$invariant and the elements of $G$ are vectorially Weyl.}
  
  \par    As $W$ contains $W^a$ and stabilizes $\shm$, we have $W=W^v\ltimes Y$, where $W^v$ fixes the origin $0$ of $\A$ and $Y$ is a group of translations such that:
  \quad    {$ Q^\vee\subset Y\subset P^\vee$}.
 An element $\mathbf{w}\in W$ will often be written $\mathbf{w}=\ql.w$ {or $\mathbf{w}=t_\ql.w$}, with $\ql\in Y$ and $w\in W^v$.


  \par We ask $Y$ to be {\bf discrete} in $V$. This is clearly satisfied if $\QF$ generates $V^*$ \ie $(\qa_i)_{i\in I}$ is a basis of $V^*$.
  {We write $Z=\qn^{-1}(Y)$ and $Z_0=\ker\qn$.}

\medskip
\parni{\bf 6)} Note that there is only a finite number of constants $q_M$ as in the definition of thickness. 
 First for $\qa\in\QF$, one has $\qa(\qa^\vee)=2$, hence the translation by $\qa^\vee\in  Q^\vee \subset Y$ permutes the walls $M(\qa,k)$ (for $k\in \Z$) with two orbits. 
 So, $ Q^\vee\subset W^a$ has at most two orbits in the set of the constants $q_{M(\qa,k)}$: one containing $q_\qa=q_{M(\qa,0)}$ and the other containing $q_\qa'=q_{M(\qa,\pm1)}$.
 Moreover we have $w.M(\qa,k)=M(w(\qa),k),\forall w\in W^v$.
 So the only possible parameters are the $q_i:=q_{\qa_i}$ and $q'_i:=q'_{\qa_i}$.
 We denote this set of parameters by $\shq=\{q_i,q'_i\mid i\in I\}$.

\par If $\qa_i(\qa_j^\vee)$ is odd for some $i,j\in I$, the translation by $\qa_j^\vee$ exchanges the two walls $M(\qa_i,0)$ and $M(\qa_i,\qa_i(\qa_j^\vee))$; so $q_i=q'_i$.
More generally, we see that $q_i=q'_i$ when {$\qa_i(Y)=\Z$, \ie $\qa_i(Y)$} contains an odd integer.
If $\qa_i(\qa_j^\vee)=\qa_j(\qa_i^\vee)=-1$, one knows that the element $r_ir_jr_i$ of $W^v(\{i,j\})$ exchanges $\qa_i$ and $-\qa_j$, so  $q_i=q'_i=q_j=q'_j$.

\par As in \cite[4.8]{BPGR16}, for $\qa\in\QF$ and $n\in\N=\Z_{≥0}$, we write $q_\qa^{*n}=q_\qa.q'_\qa.q_\qa.q'_\qa.\cdots$ and $q_\qa'^{*n}=q_\qa'.q_\qa.q'_\qa.q_\qa.\cdots$ with $n$ terms in each product.

\medskip
  \par\noindent{\bf 7) Remark.}
  All isomorphisms in \cite{R11} are Weyl-isomorphisms, and, when $G$ is strongly transitive, all isomorphisms constructed in \lc are induced by an element of $G$.


  \subsection{Type $0$ vertices}\label{1.4}

  The elements of $Y$, through the identification $Y=N.0$, are called {\it vertices of type $0$} in $\A$; they are special vertices. We note $Y^+=Y\cap\sht$ and $Y^{++}=Y\cap \overline{C^v_f}$.
   The type $0$ vertices in $\SHI$ are the points on the orbit $\SHI_0$ of $0$ by $G$. This set $\SHI_0$ is often called the affine Grassmannian as it is equal to $G/K$, where $K =$ Stab$_G(\{0\})$. But in general, $G$ is not equal to $KYK=KNK$ \cite[6.10]{GR08} \ie $\SHI_0\not=K.Y$.

   \par {Using the  preorder $≤$ on $\SHI$}, we set $\SHI^+=\{x\in\SHI\mid0\leq x\}$, $\SHI^+_0=\SHI_0\cap\SHI^+$ and $G^+=\{g\in G\mid0\leq g.0\}$; so $\SHI^+_0=G^+.0=G^+/K$.
   As $\leq $ is a $G-$invariant preorder, $G^+$ is a semigroup.

 \par  If $x\in\SHI^+_0$ there is an apartment $A$ containing $0$ and $x$ (by definition of $\leq$) and all apartments containing $0$ are conjugated to $\A$ by $K$ (\cf \ref{1.3}{.4}); so $x\in K.Y^+$ as $\SHI^+_0\cap\A=Y^+$.
    But $\qn(N\cap K)=W^v$ and $Y^+=W^v.Y^{++}$, with uniqueness of the element in $Y^{++}$. 
     So $\SHI^+_0=K.Y^{++}$, more precisely $\SHI^+_0=G^+/K$ is the union of the $KyK/K$ for $y\in Y^{++}$.
 This union is disjoint, for the above construction does not depend on the choice of $A$ (\cf \ref{1.11}.a).

    \par Hence, we have proved that the map $Y^{++}\to K\backslash G^+/K$ is one-to-one and onto.

\subsection{Vectorial distance and {$ Q^\vee-$}order}\label{1.5} 

    For $x$ in the Tits cone $\sht$, we denote by $x^{++}$ the unique element in $\overline{C^v_f}$ conjugated by $W^v$ to $x$.

    \par  Let $\SHI\times_\leq \SHI=\{(x,y)\in\SHI\times\SHI\mid x\leq y\}$ be the set of increasing pairs in $\SHI$.
    Such a pair $(x,y)$ is always in a same apartment $g.\A$; so $(g^{-1}).y-(g^{-1}).x\in\sht$ and we define the {\it vectorial distance} $d^v(x,y)\in  \overline{C^v_f}$ by $d^v(x,y)=((g^{-1}).y-(g^{-1}).x)^{++}$.
    It does not depend on the choices we made (by \ref{1.11}.a below).
    

    \par For $(x,y)\in \SHI_0\times_\leq \SHI_0=\{(x,y)\in\SHI_0\times\SHI_0\mid x\leq y\}$, the vectorial distance $d^v(x,y)$ takes values in $Y^{++}$.
     Actually, as $\SHI_0=G.0$, $K$ is the  stabilizer of $0$ and $\SHI^+_0=K.Y^{++}$ (with uniqueness of the element in $Y^{++}$), the map $d^v$ induces a bijection between the set $(\SHI_0\times_\leq \SHI_0)/G$ of $G-$orbits in $\SHI_0\times_\leq \SHI_0$ and $Y^{++}$.

     \par Further, $d^v$ gives the inverse of the  map $Y^{++}\to K\backslash G^+/K$, as any $g\in G^+$ is in $K.d^v(0,g.0).K$.

     \par For $x,y\in\A$, we say that $x\leq _{Q^\vee}y$ (resp. $x\leq _{Q^\vee_\R}\,y$) when $y-x\in { Q^\vee_+}$ (resp. $y-x\in Q^\vee_{\R+}=\sum_{i\in I}\,\R_{\geq 0}.\qa_i^\vee$).
     We get thus {an order on $\A$}.

\subsection{Paths}
\label{1.9}

We consider piecewise linear continuous paths
$\pi:[0,1]\rightarrow \mathbb A$ such that each (existing) tangent vector $\pi'(t)$
belongs to an orbit $W^v.\lambda$ for some $\lambda\in {\overline{C^v_f}}$. Such a path is called a {\it $\lambda-$path}; it is
increasing with respect to the preorder relation $\leq$ on $\mathbb A$.
 For any $t\neq 0$ (resp. $t \neq1$), we let
$\pi'_-(t)$ (resp. $\pi'_+(t)$) denote the derivative of $\pi$ at $t$ from the left
(resp. from the right). 

\par Hecke paths of shape $\ql$ (with respect to the sector germ $\g S_{-\infty}=germ_\infty(-C^v_f)$) are  $\ql-$paths satisfying some further precise conditions, see \cite[3.27]{KM08} or \cite[1.8]{GR14}. For us their interest will appear just below in \ref{1.10}.







  \subsection{Retractions onto $Y^+$}\label{1.10} 

For all $x\in \SHI^+$ there is an apartment containing $x$ and {$C_0^\pm=germ_0(\pm C^v_f)$} \cite[5.1]{R11} and this apartment is conjugated to $\A$ by an element of $K$ fixing {$C_0^\pm$ (\cf \ref{1.3}.1.a )}.
    So, by the usual arguments and [\lc, 5.5], see below \ref{1.12}.a), we can define the retraction $\qr_{{C_0^\pm,\A}}$ of $\SHI^+$ into $\A$ with center {$C_0^\pm$}; its image is $\qr_{{C_0^\pm,\A}}(\SHI^+)=\sht=\SHI^+\cap\A$ and $\qr_{{C_0^\pm,\A}}(\SHI^+_0)=Y^+$.

    \par Using {\ref{1.3}.1 b) and c)}, \cf  \cite[4.4]{GR08}, we may also define the retraction $\qr_{-\infty}$ of $\SHI$ onto $\A$ with center the sector-germ $\g S_{-\infty}$.

    \par More generally {by \cite[1.8]{BPGR16}}, we may define the retraction $\qr$ of $\SHI$ (resp. of the subset $\SHI_{\geq z}=\{y\in\SHI\mid y\geq z\}$, for a fixed $z$) onto an apartment $A$ with center any sector germ (resp. any local chamber with vertex $z$) {contained in $A$}.
 For any such retraction $\qr$, the image of any segment $[x,y]$ with $(x,y)\in\SHI\times_\leq \SHI$ and $d^v(x,y)=\ql\in\overline{C^v_f}$ (resp. and moreover $x,y\in\SHI_{\geq  z}$) is a $\ql-$path. 
  In particular, $\qr(x)\leq \qr(y)$.

     \par Actually, the image by $\qr_{-\infty}$ of any segment $[x,y]$ with $(x,y)\in\SHI\times_\leq \SHI$ and $d^v(x,y)=\ql\in Y^{++}$ is a Hecke path of shape $\ql$ with respect to $\g S_{-\infty}$.  We have the following:



 \begin{lemm*}  a) For $\ql\in Y^{++}$ and $w\in W^v$, $w.\ql\in\ql- Q^\vee_+$, \ie $w.\ql\leq _{{ Q}^\vee}\,\ql$.

 \par b) Let $\qp$ be a Hecke path of shape $\ql\in Y^{++}$ with respect to $\g S_{-\infty}$, from $y_0\in Y$ to $y_1\in Y$.
  Then, for $0\leq t<t'<1$,
$$
\begin{array}{l}\ql=\qp'_+(t)^{++}=\qp'_-(t')^{++};\\
\qp'_+(t)\leq _{{ Q}^\vee}\qp'_-(t')\leq _{{ Q}^\vee}\qp'_+(t')\leq _{{ Q}^\vee}\qp'_-(1);\\
\qp'_+(0)\leq _{{ Q}^\vee}\,\ql;\\ 
\qp'_+(0)\leq _{{ Q}^\vee}\,(y_1-y_0)\leq _{{ Q}^\vee}\,\qp'_-(1)\leq _{Q^\vee}\,\ql;\\
y_1-y_0 \leq _{{ Q}^\vee}\,\ql.
\end{array}
$$

  \par{ Moreover $y_1-y_0$ is in the convex hull $conv(W^v.\ql)$ of all $w.\ql$ for $w\in W^v$, more precisely in the convex hull 
   $conv(W^v.\ql,\geq\qp'_+(0))$ of all $w'.\ql$ for $w'\in W^v$, $w'\leq w$, where $w$ is the element with minimal length such that $\qp'_+(0)=w.\ql$.}



\par c) If  $x\leq z\leq y$ in $\SHI_0$, then $d^v(x,y)\leq _{{ Q}^\vee}d^v(x,z)+d^v(z,y)$.

 \end{lemm*}

\begin{NB} The last paragraph of b) above, applied to the image of a segment $[x,y]$ with $d^v(x,y)=\ql$, proves a Kostant convexity result for $\qr_{-\infty}$ (generalizing a part of \cite{Hi10}): for $x\in Y$, $\qr_{-\infty}(\{ y\in\SHI \mid d^v(x,y)=\ql\}) \subset x+conv(W^v.\ql)$.
\end{NB}

\subsection{Chambers of type $0$}\label{1.11}

Let $\mathscr C_0^{\pm}$ be the set of all {\it chambers of type $0$}, \ie all local chambers with vertices of type $0$ and {of sign $\pm$}.
{A local chamber} of vertex $x\in\SHI_0$ 
 will often be written  $C_x$  {and its direction} $C_x^v{=\vect {C_x}}$.
We consider $\mathscr C_0^{\pm}\times_\leq \mathscr C_0^{\pm}=\{(C_x,C_y)\in \mathscr C_0^{\pm}\times \mathscr C_0^{\pm}\mid x\leq y\}$.

\begin{prop*} \cite[5.4 and 5.1]{R11} Let $x,y\in\SHI$ with $x\leq y$.
We consider two local faces $F_x,F_y$ with respective vertices $x,y$.

\par a) $\{x,y\}$ is included in an apartment and two such apartments $A,A'$ are isomorphic by a Weyl-isomorphism in $G$, fixing $cl_A(\{x,y\})=cl_{A'}(\{x,y\})\supset[x,y]$. {This property is called the preordered convexity of intersections of apartments.}

\par b) There is an apartment containing $F_x$ and $F_y$, {unless $F_x$ and $F_y$ are respectively positive and negative.
In this case we have to assume moreover $x\stackrel{o}{<} y$ or $x=y$ to get the same result.}
\end{prop*}

\parni{\bf Consequence.}  We define $W^+ = W^v\ltimes Y^+$ which is a subsemigroup of $W$.

\par If $C_x^{\pm}\in\SHC_0^{\pm}$ {and $0≤x$}, we know by b) above, that there is an apartment $A$ containing $C_0^{\pm}$ and $C_x$.
But all apartments containing $C_0^{\pm}$ are conjugated to $\A$ by $K_I^{\pm}$ (by \ref{1.3}.1.a), so there is $k\in K_I^{\pm}$ with $k^{-1}.C_x\subset\A$.
 Now the vertex $k^{-1}.x$ of $k^{-1}.C_x$ satisfies $k^{-1}.x\geq0$, so there is $\mathbf{w}\in W^+$ such that $k^{-1}.C_x=\mathbf{w}.C_0^{\pm}$.

\par When $g\in G^+$, $g.C_0^{\pm}$ is in $\SHC_0^{\pm}$ and there are $k\in K_I^{\pm}$, $\mathbf{w}\in W^+$ with $g.C_0^{\pm}=k.\mathbf{w}.C_0^{\pm}$, \ie $g\in K_I^{\pm}.W^+.K_I^{\pm}$.
We have proved the {\it Bruhat decomposition} $G^+=K_I^{\pm}.W^+.K_I^{\pm}$.
{Similarly we have the {\it Birkhoff decomposition} $G^+=K_I^-.W^+.K_I^+$, but perhaps not $G^+=K_I^+.W^+.K_I^-$.}



\begin{prop}\label{1.12} \cite[1.12]{BPGR16} In the situation of Proposition \ref{1.11},

\par a) If $x\stackrel{o}{<} y$ or $F_x$ and $F_y$ are respectively  negative and positive,  any two apartments $A,A'$ containing $F_x$ and $F_y$ are isomorphic by a Weyl-isomorphism in $G$ fixing the convex hull of $F_x$ and $F_y$ (in $A$ or $A'$).

\par b) If $x=y$ and the directions of $F_x,F_y$ have the same sign, any two apartments $A,A'$ containing $F_x$ and $F_y$ are isomorphic by a Weyl-isomorphism in $G$, $\qf:A\to A'$, fixing $F_x$ and $F_y$.
 If moreover $F_x$ is a local chamber, any minimal gallery from $F_x$ to $F_y$ is fixed by $\qf$ (and in $A\cap A'$).

 \par c) If $F_x$ and $F_y$ {are positive (resp. negative)} and $F_y$ {(resp. $F_x$)} is spherical, any two apartments $A,A'$ containing $F_x$ and $F_y$ are isomorphic by a Weyl-isomorphism  in $G$ fixing $F_x$ and $F_y$.
 
 \par {d) Any isomorphism $\qf:A\to A'$ fixing a local facet $F\subset A\cap A'$ fixes $\overline F$.}

\end{prop}


\subsection{$W-$distance}\label{1.13} 

Let $(C_x,C_y)\in\mathscr C_0^{\pm}\times_\leq \mathscr C_0^{\pm}$, there is an apartment $A$ containing $C_x$ and $C_y$.
We identify $(\A,C_0^{\pm})$ with $(A,C_x)$ \ie we consider the unique $f\in Isom^W_\R(\A,A)$ such that $f(C_0^{\pm})=C_x$.
 Then $f^{-1}(y)\geq 0$ and there is $\mathbf{w}\in W^+$ such that  $f^{-1}(C_y)=\mathbf{w}.C_0^{\pm}$.
 By \ref{1.12}.c, $\mathbf{w}$ does not depend on the choice of $A$.

\par We define the {\it $W-$distance} between the two local chambers $C_x$ and $C_y$ to be this unique element: $d^W(C_x,C_y) = \mathbf{w}\in W^+ = Y^+\rtimes W^v$.
 If $\mathbf{w}=t_\ql.w$, with $\ql\in Y^+$ and $w\in W^v$, we write also {$d^v(C_x,y)=\ql$}.
 As $\leq$ is $G-$invariant, {$d^W$ or $d^v$} is also $G-$invariant.
 When $x=y$, this definition {of $d^W$} coincides with the one in \ref{1.3}.2.

 \par If $C_x,C_y,C_z\in\SHC_0^{\pm}$, with $x\leq y\leq z$, are in a same apartment, we have the Chasles relation: $d^W(C_x,C_z)=d^W(C_x,C_y).d^W(C_y,C_z)$.

\par When $C_x=C_0^{\pm}$ and $C_y=g.C_0^{\pm}$ (with $g\in G^+$), $d^W(C_x,C_y)$ is the only $\mathbf{w}\in W^+$ such that $g\in K_I^{\pm}.\mathbf{w}.K_I^{\pm}$.
We have thus proved the uniqueness in Bruhat decomposition: $G^+=\coprod_{\mathbf{w}\in W^+}\,K_I^{\pm}.\mathbf{w}.K_I^{\pm}$.

 \par The $W-$distance classifies  the orbits of $K_I^{\pm}$ on $\{C_y\in\SHC_0^{\pm}\mid y\geq 0\}$, hence also the orbits of $G$ on $\mathscr C_0^{\pm}\times_\leq \mathscr C_0^{\pm}$.

\par {Similarly we may define a $W-$codistance $d^{*W}$ on $\SHC_0^-\times_≤\SHC_0^+$ with values in $W^+$. It classifies the orbits of $G$ on $\SHC_0^-\times_≤\SHC_0^+$ and $d^{*W}(C_x,C_y)$ is the codistance defined in  \ref{1.3}.2 when $x=y$.}


\subsection{The Kac-Moody examples}\label{1.13c} 

\par Let $G$ be an almost split Kac-Moody group over a non archimedean complete field $\shk$.
 We suppose moreover the valuation of $\shk$ discrete and its residue field $\qk$  perfect.
 Then there is a masure $\SHI$ on which $G$ acts by vectorially Weyl automorphisms. 
 If $\shk$ is a local field (\ie $\qk$ is finite), then we are in the situation described above.
 This is the main result  of \cite{Ch10}, \cite{Ch11} and \cite{R13}. 
 This generalizes the Bruhat-Tits result, dealing with $G$ reductive.
 In the group $G$, there are root groups $U_\qa$, indexed by the relative real root system $\QF_{rel}$, with good commutation relations.
 But $\QF_{rel}$ may be different from $\QF$ as chosen above in \ref{1.2}.1. 
 We get only an identification of the root rays $\R_+.\qa$ associated to the $\qa\in \QF$ or the $\qa\in  \QF_{rel}$.
 
 \par Suppose now $G$ split  and over a field $\shk$ endowed with a discrete  valuation with finite residue field.
 Then the same results as above hold and we have $\QF= \QF_{rel}$.
  
\begin{rema}\label{1.14} A general abstract situation of pairs $(\SHI,G)$ involving root groups $U_\qa$ (as above in \ref{1.13c}) is axiomatized in \cite{R13} (under the name of ``parahoric masure'').
 We do not use here this kind of hypothesis.
It may seem to be used in \cite{GR14}: some of the results or of the proofs in \lc are described using some groups  $U_{\qa,k}$ (essentially from just before 4.4 until 4.8).
But these results are actually correct as stated (\ie with the same hypotheses as above up to \ref{1.13}):  the group $\overline U_{\qa,k}=U_{\qa,k}/U_{\qa,k+}$ is simply a good parametrization of the set of minimal half apartments containing strictly $D(\qa,k)$; {it may be replaced by a set of representatives of $G_{D(\qa,k)}/G_{D(\qa,k+)}$}.

\end{rema}


\section{Convolutions: algebras, modules and Satake isomorphism}\label{s2}

\subsection{Vertices and chambers of type $0$}\label{2.1} 

\par We introduced above the sets $\SHI_0$ and $\shc_0^\pm$  of vertices of type $0$ and of positive or negative chambers of type $0$.
The group $G$ acts transitively on these sets.
The fundamental element of $\SHI_0$ (resp. $\shc_0^\pm$) is $0\in\A$ (resp. $C_0^\pm\subset\A$) and its stabilizer in $G$ is $K$ (resp. $K^\pm_I$).
 So $\SHI_0=G/K$ and $\shc_0^\pm=G/K_I^\pm$.
 
 \par If $E_1,E_2\in \SHI_0$ (\resp $\in\shc_0^\pm$), we write $E_1≤E_2$ when $origin(E_1)≤origin(E_2)$, where $origin(E_i)\in\SHI_0$ is the origin of $E_i$ (the vertex $E_i$ itself if $E_i\in\SHI_0$).
 We get thus a preorder.
 
 \par Let $\SHE=G/K_1$ be a family of objects as above.
   We define $\SHE\times_≤\SHE=\{(E_1,E_2)\in \SHE\times\SHE \mid E_1≤E_2\}$. 
   There is an action of $G$ on this set, and the set $(\SHE\times_≤\SHE)/G$ of $G-$orbits is in one to one correspondence with $K_1\backslash G^+/K_1$ via the map sending $K_1gK_1\in K_1\backslash G^+/K_1$ to the $G-$orbit of the pair $(E^f,g.E^f)$, where $E^f$ is the fundamental element in $\SHE$ (hence $K_1=G_{E^f}$).
   
 \par If $R$ is any commutative ring with unit, we write $\widehat\M^R(\SHE)=\widehat\M(\SHE)$ the $R-$module $Map_G(\SHE\times_≤\SHE,R)$ of all $G-$invariant functions on $\SHE\times_≤\SHE$ (\ie of functions on $K_1\backslash G^+/K_1$).
 We write $\M^R(\SHE)=\M(\SHE)$ the submodule of functions with support a finite union  of orbits of $G$ (\ie with finite support in $K_1\backslash G^+/K_1$).
 
 \par We may also define $\SHE\times_≥\SHE$ and $(\SHE\times_≥\SHE)/G=K_1\backslash G^-/K_1$ with $G^-=(G^+)^{-1}=\{g\in G\mid g.0≤0\}$.
 They are in bijections with $\SHE\times_≤\SHE$ and $(\SHE\times_≤\SHE)/G=K_1\backslash G^+/K_1$, via the maps $(E_1,E_2)\mapsto(E_2,E_1)$ and $K_1gK_1\mapsto K_1g^{-1}K_1$.
 We write $\widehat\M_≥(\SHE)$ or $\M_≥(\SHE)$ the corresponding sets of $G-$invariant functions; they are isomorphic to $\widehat\M(\SHE)$ or $\M(\SHE)$ via the map $\qi=\qi_{\SHE}$, $\qi\qf(E_2,E_1)=\qf(E_1,E_2)$, $\qi\qf(K_1gK_1)=\qf(K_1g^{-1}K_1)$.
 
  \par One considers also the $R-$module  $\widehat\SHF(\SHE)$ of the functions $\chi:\SHE\to R$ invariant by the action of the pointwise stabilizer $G_{-\infty}=G_{{\g S}_{-\infty}}$ of ${\g S}_{-\infty}$, \ie constant on the fibers of $\qr_{-\infty}$.
 We may also think of $\widehat\SHF(\SHE)$  as the $R-$module of functions on $\SHE(\A)=\{E\in\SHE \mid E\subset\A\}$.
 
 \par We write $\SHF_{fin}(\SHE)$ the $R-$submodule of functions with finite support (in $\SHE(\A)$).

 \subsection{Convolution}\label{2.3} 

 One considers a family $\SHE$  equal to $\SHI_0$ or $\shc_0^\pm$.
 For $\qf,\psi\in\widehat\M(\SHE)$, we define their convolution product (if it exists) as the function $\qf*\psi\in\widehat\M(\SHE)$ defined by:
  $$\qf*\psi(E_1,E_3)=\sum_{E_2\in\SHE}\,\qf(E_1,E_2)\psi(E_2,E_3)$$
 in this formula, we set $\qf(E_1,E_2)=0$ (resp. $\psi(E_2,E_3)=0$) when $E_1\not≤E_2$ (resp. $E_2\not≤E_3$).

  For $\qf\in \widehat\SHF(\SHE)$ and $\psi\in\widehat\M(\SHE)$, we define their convolution product (if it exists) as the function $\qf*\psi\in\widehat\SHF(\SHE)$ defined by:

  $$\qf*\psi(E_2)=\sum_{E_1\in\SHE}\,\qf(E_1)\psi(E_1,E_2)$$
 in this formula, we set also $\psi(E_1,E_2)=0$ when $E_1\not≤E_2$.
 
\par Actually we are led to look only at the case where $\qf$ and $\psi$ are the characteristic functions of some $G-$orbits in $\SHE\times_≤\SHE$ or some $G_{-\infty}-$orbit in $\SHE$  and we may also suppose $R=\Z$.
Then $\qf*\psi$ is well defined in $\widehat\M^\Z(\SHE)$ or $\widehat\SHF^\Z(\SHE)$, if we are willing to allow this function to take infinite values.
These values are called the {\it structure constants} of the convolution product on 
$\widehat\M(\SHE)$ or on $\widehat\SHF(\SHE) \times \widehat\M(\SHE)$.
 
 So the first problem for the existence of the convolution product is to avoid these infinite values.
 The second problem is about supports: we are interested in functions in $\widehat\M(\SHE)$ or $\widehat\SHF(\SHE)$ with some conditions of support.
 When $\qf$ and $\psi$ satisfy the corresponding conditions, does $\qf*\psi$ satisfy also the appropriate condition?
 The answer (in some cases) is given in \cite{GR14} or \cite{BPGR16} and recalled below.
 
 \par The convolution product has interesting properties, when it exists: it is $R-$bilinear and associative.
 More precisely the following equalities hold whenever each convolution product appearing in the equality is defined.

\par $(a_1\qf_1+a_2\qf_2)*(b_1\psi_1+b_2\psi_2)=\sum_{i,j}\,a_ib_j\qf_i*\psi_j$ \quad and \quad $(\qf*\psi)*\theta=\qf*(\psi*\theta)$

\parni for $a_i,b_j\in R$ and $\qf,\qf_i,\psi,\psi_j,\qth$ are in $\widehat\M(\SHE)$ or $\widehat\SHF(\SHE)$.

\begin{rema*} We may also consider the modules $\widehat\M_≥(\SHE)$; we saw in \ref{2.1} that they are isomorphic, via $\qi$, to $\widehat\M(\SHE)$.
 Moreover, for $\qf,\psi\in\widehat\M(\SHE)$, then  $\qi(\qf),\qi(\psi)\in\widehat\M_≥(\SHE)$ and $\qi(\psi)*\qi(\qf)=\qi(\qf*\psi)$ in $\widehat\M_≥(\SHE)$, if at least one of these two convolution products is defined.
 
\end{rema*}

 \subsection{Algebras and modules}\label{2.4} 
 
 \parni{\bf a)} For $(E_1,E_2)\in\SHE\times_≤\SHE$, there is  an apartment $A$ containing $E_1$ and $E_2$ (\cf \ref{1.11}.b).
 Moreover by \ref{1.3}.1.c and \ref{1.12},  two apartments containing $E_1$ and $E_2$ are isomorphic by a Weyl isomorphism (induced by a $g\in G$) fixing $E_1\cup E_2$.
 So the classification of the $G-$orbits in $\SHE\times_≤\SHE$ is equivalent to the classification of the $W-$orbits in $(\SHE\times_≤\SHE)(\A):=\{(E_1,E_2)\in \SHE\times_≤\SHE \mid E_1,E_2\subset\A \}$.
 
 \par For the most important examples, we give now a more precise parametrization of these orbits and a name for the corresponding characteristic functions: they constitute the canonical basis of $\M(\SHE)$ or $\SHF_{fin}(\SHE)$.
 
 \medskip
 \parni{\bf b)} $\M(\SHI_0)$ and the spherical Hecke algebra $\shh\subset \widehat \M(\SHI_0)$ : The vectorial distance $d^v:\SHI_0\times_≤\SHI_0\to Y^{++}$ classifies the orbits of $G$ (\ie $K\backslash G^+/K\simeq Y^{++}$).
 So the canonical basis  of $\M(\SHI_0)$ is $(c_\ql)_{\ql\in Y^{++}}$, where, for $x,y\in\SHI_0$, $c_\ql(x,y)=1$ if $x≤y$, $d^v(x,y)=\ql$ and $c_\ql(x,y)=0$ otherwise.
 
  \par The convolution on $\widehat\M^R(\SHI_0)$ is studied in \cite{GR14}.
 To get an algebra inside $\widehat\M^R(\SHI_0)$, it is natural to consider the $R-$submodule $\M^R_{af}(\SHI_0)$ of functions $\qf=\sum_{\ql\in Y^{++}}\,a_\ql.c_\ql:\SHI_0\times_≤\SHI_0\to R$ such that the support $supp(\qf)=\{\ql\in Y^{++} \mid a_\ql≠0\}$ is almost finite, \ie:
 \medskip
 \par \qquad there exist $\ql_1,\cdots,\ql_n\in Y^{++}$ such that $supp(\qf)\subset \bigcup_{i=1}^n\,(\ql_i- Q_+^\vee)$.
\medskip 
 \par Then $\shh_R=\shh:=\M^R_{af}(\SHI_0)$ is an algebra, the {\it spherical Hecke algebra}.

 \medskip
 \parni{\bf c)} $\M(\SHC^±_0)$ and the Iwahori-Hecke algebras ${^I\shh}^±$: The W-distance $d^W:\SHC_0^±\times_≤\SHC^±_0\to W^+=W^v\ltimes Y^{+}$ classifies the orbits of $G$ (\ie $K_I^±\backslash G^+/K_I^±\simeq W^+$).
 So the canonical basis  of $\M(\SHC^±_0)$ is $(T^±_{\mathbf w})_{\mathbf w\in W^+}$, where, for $C_x,C_y\in\SHC^±_0$, $T^±_{\mathbf w}(C_x,C_y)=1$ if $C_x≤C_y$, $d^w(C_x,C_y)=\mathbf w$ and $T^±_{\mathbf w}(C_x,C_y)=0$ otherwise.
 
  \par In \cite{BPGR16} we studied the convolution product on ${^I\shh}={^I\shh^+}:=\M(\SHC_0^+)$.
  It is always well defined and in ${^I\shh}$.
  So ${^I\shh}$ is an $R-$algebra: the (positive) {\it Iwahori-Hecke algebra}.
  We elucidated its structure: if $\Z\subset R$ and the parameters $q_i,q'_i$ are invertible in $R$, we may define a basis of ${^I\shh}$ and explicit relations (including some ``Bernstein-Lusztig'' relations) defining the multiplication.
  Actually for good $R$, the algebra ${^I\shh}$ is a subalgebra of a specialization of the Bernstein-Lusztig algebra ${^{BL}\shh}$, we shall speak of in §\ref{s5}, see \ref{5.2}.6.

   \par As explained in Remark \ref{2.3}, ${^I\shh}^-:=\M(\SHC_0^-)$ (with its convolution) is anti-isomorphic (via $\qi$) to $\M_≥(\SHC_0^-)$.
   But now $\M_≥(\SHC_0^-)$ has the same structure as ${^I\shh}=\M(\SHC_0^+)$: we just have to exchange all the signs $±$, and, as the choice of what is $+$ or $-$ in $\SHI$ was arbitrary, we get the same explicit structure.
   
 \par In particular the convolution is always well defined in ${^I\shh^-}$ and we may describe precisely the structure of this algebra (called the negative Iwahori-Hecke algebra).
   Moreover ${^I\shh^-}$ is anti-isomorphic to $^I\shh$ by a map that we still call $\qi$.
   We shall explain in \ref{5.2}.6 a more precise link with ${^{BL}\shh}$ .
   
 \medskip
 \parni{\bf d)} $\SHF_{fin}=\SHF_{fin}(\SHI_0)$ and the $\shh-$module $\SHF$: the functions $\qf$ in $\SHF_{fin}$ may be considered as functions on $\SHI_0(\A)=\A_0=Y$.
 So the canonical basis is $(\chi_\qm)_{\qm\in Y}$, where for $x\in\SHI_0$, $\chi_\qm(x)=1$ if $\qr_{-\infty}(x)=\qm$ and $\chi_\qm(x)=0$ otherwise.
 
 \par One defines also the $R-$submodule $\SHF$ of $\widehat\SHF(\SHI_0)$ by a condition of almost finite support:
\medskip
\par for $\chi=\sum_{\qm\in Y}\,a_\qm.\chi_\qm\in\widehat\SHF(\SHI_0)$, the support $supp(\chi)=\{\qm\in Y \mid a_\qm≠0 \}$ is almost finite if there exist $\ql_1,\cdots,\ql_n\in Y^{}$ such that $supp(\chi)\subset \bigcup_{i=1}^n\,(\ql_i- Q_+^\vee)$.
\medskip
\par Then, by \cite[5.2]{GR14}, the convolution $\SHF\times\shh \to \SHF$ is well defined.
 So $\SHF$ is a right $\shh-$module.
    
     \subsection{Left actions of $R[[Y]]$ on $\SHF$}\label{2.10} 

\par We define the following completion $R[[Y]]$ of the group algebra $R[Y]$:
it is the set of all $f=\sum_{\ql\in Y} a_\ql e^\ql\in R^Y$ such that $ supp(f)=\{\ql\in Y\mid a_\ql\not=0\}$ is in  $\cup_{j=1}^n(\mu_j- Q^\vee_+)$  for some $\qm_j\in Y$.
Actually we may suppose $\qm_j\in supp(f)$ by the Lemma \ref{2.10a} below, indicated to us by Auguste H\'ebert.
Clearly $R[[Y]]$ is a commutative $R-$algebra (with $e^\ql.e^\qm=e^{\ql+\qm}$). 

\par The formula $(f.\chi)(\qm)=\sum_{\ql\in Y}\,a_\ql\chi(\qm-\ql)$, for $f=\sum a_\ql e^\ql\in R[[Y]]$, $\chi\in\SHF$ and $\qm\in Y$, defines an element $f.\chi\in\SHF$; in particular $e^\ql.\chi_\qm=\chi_{\qm+\ql}$.
 By \cite[5.1, 5.2]{GR14}, the map $R[[Y]]\times\SHF\to\SHF$, $(f,\chi)\mapsto f.\chi$ makes $\SHF$ into a free $R[[Y]]-$module of rank $1$, with any $\chi_\qm$ as basis element.
 Moreover the actions on $\SHF$ of $R[[Y]]$ (on the left) and of $\shh$ (on the right) commute.

\par 
We may slightly modify this action of $R[[Y]]$ when $R$ is large enough.
As in \cite[5.3.2]{GR14}, we choose an homomorphism $\qd^{1/2}:Y\to \R_+^*$ such that $\qd=(\qd^{1/2})^2$, restricted to $ Q^\vee$, sends $\sum_{i\in I}\, a_i.\qa_i^\vee$ to $\prod_{i\in I}\,(q_iq'_i)^{a_i}$.
We suppose now that $R$ contains the image of $\qd^{1/2}$.
Then the formula :

\par $(f\Box\chi)(\qm)=\sum_{\ql\in Y}a_\ql.\qd^{-1/2}(\ql).\chi(\qm-\ql)$, for $f=\sum a_\ql.e^\ql\in R[[Y]]$, $\chi\in\SHF$ and $\qm\in Y$, 
\medskip
\parni defines an element $f\Box\chi\in\SHF$; in particular $e^\ql\Box\chi_\qm=\qd^{-1/2}(\ql).\chi_{\qm+\ql}$.
We get thus a new structure of free $R[[Y]]-$module of rank $1$ on $\SHF$, commuting with the action of $\shh$.
 
 \par{Note that, when all $q_i,q'_i$ are equal to some $q$, a good choice for $\qd^{1/2}$ is given by $\qd^{1/2}(\qm)=q^{\qr(\qm)}$, where $\qr\in X$ takes value $1$ on each simple coroot.}
 

{\begin{lemm}\label{2.10a} \cite{AbH17} Let $A$ be a subset of $Y$ and $y$ an element in $Y$. Then the intersection $A\cap(y-Q_+^\vee)$ is included in a finite union of sets $a_i-Q_+^\vee$ for some $a_i\in A$
\end{lemm}

\begin{proof} Let $N$ be the set of the $\qn\in Q_+^\vee$ such that $y-\qn\in A$, that are minimal (with respect to $\leq_{Q^\vee}$) for this property.
Clearly $A\cap(y-Q_+^\vee)\subset \bigcup_{\qn\in N} (y-\qn- Q_+^\vee)$.
The elements of $N$ are not comparable for $\leq_{Q^\vee}$, so $N$ is finite by \cite[Lemma 2.2]{He17}.
\end{proof}}
 
 \subsection{The Satake isomorphism $\shs$}\label{2.11} 

\parni{\bf 1)}  As $\SHF$ is a free $R[[Y]]-$module of rank one (for $.$ or $\Box$), we have $End_{R[[Y]]}(\SHF)=R[[Y]]$. 
So the right action of $\shh$ on the $R[[Y]]-$module $\SHF$ gives an algebra homomorphism $\shh\to R[[Y]]$ for each structure $.$ or $\Box$. More precisely:

\par\qquad $\shs_*:\shh\to R[[Y]]$ is such that $\chi*\qf=\shs_*(\qf).\chi$ for any $\qf\in\shh$ and any $\chi\in\SHF$.

\par\qquad $\shs:\shh\to R[[Y]]$ is such that $\chi*\qf=\shs(\qf)\Box\chi$ for any $\qf\in\shh$ and any $\chi\in\SHF$.

\par Clearly, if $\shs_*(\qf)=\sum a_\qm.e^\qm$, then $\shs(\qf)=\sum a_\qm.\qd^{1/2}(\qm).e^\qm$.

\medskip
\parni{\bf 2)} There is a linear action of $W^v$ on $Y$, hence an algebra action on $R[Y]$, by setting $w.e^\ql=e^{w\ql}$, for $w\in W^v$ and $\ql\in Y$.
This action does not extend to $R[[Y]]$.
 But we define $R[[Y]]^{W^v}$ as the $R-$algebra of all $f=\sum \, a_\ql e^\ql\in R[[Y]]$ such that $a_\ql=a_{w\ql}$ (for all $w\in W^v$ and $\ql\in Y$).
{By the following Remark \ref{2.12}, Lemma \ref{2.10a} and Lemma \ref{1.10}.a, we may suppose $supp(f)\subset  \bigcup_{j=1}^n\,(\qm_j- Q_+^\vee)$, for some $\qm_j\in Y^{++}$.}

\begin{theo}\label{2.12}
The spherical Hecke algebra $\shh$ is isomorphic, via $\shs$, to the commutative algebra   $R[[Y]]^{W^v}$ of Weyl invariant elements in $R[[Y]]$.
\end{theo}

\begin{rema*} As stated here, following \cite[5.4]{GR14}, this theorem may be misleading, as it does not tell that any $f$ in the image of $\shs$ has its support in $Y^+$.
 Actually the last paragraph of the proof of \cite[5.4]{GR14} gives that any $f\in R[[Y]]^{W^v}$ has its support in $Y^+$, see also \cite[4.6]{AbH17}.
\end{rema*}

 \subsection{Explicitation and decomposition of $\shs(c_\ql)$}\label{2.13} 

\parni{\bf 1)} From \cite[5.2]{GR14} we get that, for $\ql\in Y^{++}$, \qquad$\shs(c_\ql)=\sum_{\qm≤_{ Q^\vee}\ql} \, \qd^{1/2}(\qm).n_\ql(\qm).e^\qm$

\parni Here $n_\ql(\qm)=\sum_\qp \, \#S_{-\infty}(\qp,\qm)$ is the sum, over all Hecke paths $\qp$ of type $\ql$ from $0$ to $\qm\in Y$, of the numbers $\#S_{-\infty}(\qp,\qm)$ and $\#S_{-\infty}(\qp,\qm)$ is  the number of line segments $[x,\qm]$ in $\SHI$ with $d^v(x,\qm)=\ql$ and $\qr_{-\infty}([x,\qm])=\qp$.

\par Actually the condition $\qm≤_{ Q^\vee}\ql$ is a consequence of the assumptions on the path $\qp$, so we may also write:
 \qquad\qquad $\shs(c_\ql)=\sum_\qp \, \qd^{1/2}(\qm). \#S_{-\infty}(\qp,\qm).e^\qm$.

\par As any Hecke path is increasing for $≤$ (\cf \ref{1.9}), any $\qm\in supp(\shs(c_\ql))$ is in $Y^+$.
So it is clear that the support of any $f\in \shs(\shh)$ is in $Y^+$ and satisfies the  condition of support in \ref{2.11}.2.

\medskip
\parni{\bf 2)} The expected Macdonald's formula is a (still more) explicit formula for $\shs(c_\ql)$.
To prove it we 
decompose $\shs(c_\ql)$ as a sum $\shs(c_\ql)=\sum_{w} \, J_w(\ql)$.
 We give here a 
 definition using paths:
 \medskip
 \par\qquad $J_w(\ql)=\sum_{\qm≤_{ Q^\vee}\ql}(\sum_\qp \, \qd^{1/2}(\qm). \#S_{-\infty}(\qp,\qm).e^\qm)=\sum_\qp \, \qd^{1/2}(\qm). \#S_{-\infty}(\qp,\qm).e^\qm$
\medskip 
 \parni where $\qp$ runs over all Hecke paths of type $\ql$ from $0$ to $\qm\in Y\subset \A$, satisfying moreover $\qp'_+(0)=w.\ql$.
 
 \par Actually $J_w(\ql)$ depends only on $w.\ql$.
 So the sum in $\shs(c_\ql)=\sum_{w} \, J_w(\ql)$ is over 
 a set $(W^v)^\ql$ of representatives of $W^v/W^v_\ql$.
{ Moreover this sum is clearly convergent as, for any $\qm$, $\#S_{-\infty}(\qp,\qm)$ and the sum $n_\ql(\qm)=\sum_{w\in W^v}(\sum_{\qp'_+(0)=w.\ql}\, \#S_{-\infty}(\qp,\qm))$ (of integers) are clearly finite.
}
 
 \par We shall prove in  \S \ref{s4}, a recursive formula for the $J_w(\ql)$ which will enable us to prove in \S \ref{s6} the expected Macdonald's formula.
 
\medskip
\parni{\bf 3)} We compute now $J_e(\ql)$.
By \ref{1.10}.b, a Hecke path $\qp$ of type $\ql$ in $\A$ with initial derived vector $e.\ql=\ql$ has to be the line segment $\qp(t)=t\ql$ from $0$ to $\ql=\qm$.
And the number $\# S_{-\infty}(\qp,\ql)$ of line segments $[x,\ql]$ in $\SHI$ with $d^v(x,\ql)=\ql$ and $\qr_{-\infty}([x,\ql])=\qp$ is $1$: any such line segment is equal to $\qp=[0,\ql]$ (see a proof in \cite{BPGR16} part e) of the proof of Lemma 2.1).
 Hence the above formula for $J_w(\ql)$ gives $J_e(\ql)=\qd^{1/2}(\ql).e^\ql$.

 \subsection{Comparison with the conventions of \cite{BrKP15}}\label{2.14} 

\par This article deals with split affine Kac-Moody groups as in \ref{1.13c}.
The prominent part there is taken by the groups $K$, $K_I=K_I^+$ and $U^+=\langle U_\qa \mid \qa\in\QF^+\rangle$.
Moreover the semigroup in \cite{BrKP15} is $G^\oplus=K.\varpi^{Y^+}.K{=K.\varpi^{Y^{++}}.K}$ where, for $\ql\in Y^+$, $\varpi^\ql$ is the element $\ql(\varpi)$ of the maximal torus $Z$ (actually $Y=Hom(\G_m,Z)$).
 By \ref{1.13c} the action of $\varpi^\ql$ on $V$ is the translation $t_{-\ql}$.
 So $G^\oplus$ is our $G^-=(G^+)^{-1}$.
 
 \par To translate the results in \lc, we have to exchange $+$ and $-$.
 We chose to work with $K$, $K_I^-$, $U^-=\langle U_\qa \mid \qa\in\QF^-\rangle$ and $G^+$.
 This corresponds moreover to the choices in \cite{GR14}.
 Contrary to \cite{BPGR16}, the convenient Iwahori-Hecke algebra is now $^I\shh^-$; but its role will be minor: only the Bernstein-Lusztig-Hecke algebra seems to appear.


\section{Recursion formula for $J_w(\ql)$ using paths}\label{s4}




In this section, using the masure and an extended tree inside it, we prove a recursion formula for the $J_w(\ql)$'s, introduced in Section \ref{2.13} (with $\ql\in Y^{++}$). 
From this subsection, recall also the notation $\#S_{-\infty}(\qp,\qm)$, then the $J_w(\ql)$'s are defined as
$$
J_w(\ql)=\sum_{\qm≤_{ Q^\vee}\ql}(\sum_\qp \, \qd^{1/2}(\qm). \#S_{-\infty}(\qp,\qm).e^\qm)=\sum_\qp \, \qd^{1/2}(\qm). \#S_{-\infty}(\qp,\qm).e^\qm
$$ where $\qp$ runs over all Hecke paths of type $\ql$ from $0$ to $\qm\in Y\subset \A$, satisfying moreover $\qp'(0)=w.\ql$. 

Consider, in the same way as Macdonald \cite{Ma03},  the functions $b$ and $c$: 
 \begin{equation}\label{5.5.2} b(t,u; z) ={{(t-t^{-1}) + (u-u^{-1})z}\over {1-z^2}}, \quad c(t, u; z)= t-b(t, u ; z)\end{equation} which satisfy the relations 
  \begin{align} 
&c(t, u; z)=c(t^{-1}, u^{-1}; z^{-1})\, ,\label{5.5.3}\\
&c(t, u; z)+c(t, u; z^{-1})=t+t^{-1}\, ,\label{5.5.3b}\\
&c(t, u; z)=t^{-1}+b(t,u;z^{-1})\, .\label{5.5.3c}
\end{align}
 
\parni In the case where $t=u$, then $b(t,t;z)={{t-t^{-1}}\over {1-z}}$ and $c(t,t;z)={{t^{-1}-tz}\over {1-z}}$.

{\par From now on the set $(W^v)^\ql$ of representatives of $W^v/W^v_\ql$ is chosen to be $(W^v)^\ql=\{ w\in W^v \mid \ell(wv)\geq \ell(w), \forall v\in W^v_\ql\}$.}

\begin{theo}\label{th3.1}

Let us fix $w\in (W^v)^\lambda$ and $i\in I$ such that $w = r_iw'$ with $\ell(w) = 1+\ell(w')$. 
Then 
$$
J_w(\ql)= q_i^{-1/2}c\big (q_i^{-1/2}, (q_i')^{-1/2}; e^{\alpha_i^\vee}\big ) J_{w'}(\lambda)^{r_i} + q_i^{-1/2}b\big (q_i^{-1/2}, (q_i')^{-1/2}; e^{\alpha_i^\vee}\big ) J_{w'}(\lambda)
$$ with $J_{w'}(\lambda)^{r_i}$ meaning that the reflection $r_i$ is applied only to the $\qm$'s in $e^\qm$.
\end{theo}

The whole section is devoted to the proof of this theorem. For any Hecke path $\pi$ of type $\ql$ from $0$ to $\qm$,  with $\qp'(0)=w.\ql$ or $\qp'(0)=w'.\ql$, let us first set 
$$
\Gamma_i(\pi) =
\left\{
\begin{array}{ll}
c_q(\alpha_i^\vee)e^{r_i\pi(1)} + b_q(\alpha_i^\vee)e^{\pi(1)} & \hbox{ if } \qp'_+(0)=w'.\ql\\
- e^{\pi(1)} & \hbox{ if } \qp'_+(0)=w.\ql
\end{array}
\right.
$$ 
where $c_q(\alpha_i^\vee)=q_i^{-1/2}c\big (q_i^{-1/2}, (q_i')^{-1/2}; e^{\alpha_i^\vee}\big )$ and $b_q(\alpha_i^\vee)=q_i^{-1/2}b\big (q_i^{-1/2}, (q_i')^{-1/2}; e^{\alpha_i^\vee}\big ).$

Therefore, the formula we want to show reads now:
\begin{equation}\label{eq3.1}
0 = \sum_\qp \, \qd^{1/2}(\pi(1)). \#S_{-\infty}(\qp,\pi(1)).\Gamma_i(\pi)
\end{equation} where the sum runs over all Hecke paths of type $\ql$ starting from $0$ such that $\qp'_+(0)=w.\ql$ or $\qp'_+(0)=w'.\ql$.

\par 
Let  $w\in (W^v)^\lambda$ and $i\in I$ be such that $w = r_iw'$ with $\ell(w) = 1+\ell(w')$ as above.
Then $\qa_i(w'C^v_f)>0$, hence $\alpha_i(w'\lambda)\geq0$.
But $\alpha_i(w'\lambda)=0\implies r_iw'\ql=w'\ql$
, a contradiction as $w\in (W^v)^\lambda$ and $\ell(w) = 1+\ell(w')$.
So $\alpha_i(w'\lambda)>0$.
Moreover $w'\in (W^v)^\lambda$: otherwise there exists $v\in W^v_\ql$ with $\ell(w'v)<\ell(w')$ and $\ell(wv)=\ell(r_iw'v) \leq 1+\ell(w'v) \leq \ell(w') < \ell(w)$, a contradiction.

\subsection{Through the extended tree}\label{sse3.1}

The masure $\SHI$ contains the extended tree $\SHT$ associated to $(\mathbb A, \alpha_i)$ that was defined in \cite{GR14} under the name $\SHI(M_\infty)$. 
Its apartment associated to $\A$ is $\A$ as affine space, but with only walls the walls directed by $\ker \qa_i$.

There, it is also proven that the retraction $\rho_{-\infty}$ factorizes through $\SHT$ and equals the composition $\rho_{-\infty}:\SHI\stackrel{\rho_1}{\to}\SHT\stackrel{\rho_2}{\to}\mathbb A$, where $\rho_1$ is the parabolic retraction defined in 5.6 of \cite{GR14} and $\rho_2$ is the retraction with center the end $-\infty_\SHT$ (denoted $\g S'$ in loc. cit.), \ie the class of half-apartments in $\SHT$ containing $\g S_{-\infty}$.

Recall that there is a group $G$ acting strongly transitively on $\SHI$, let $G_\SHT$ be the stabilizer of $\SHT$ under this action. Because of the assumptions on $G$ and the action, one has, for any $x\in\SHI$ and any $g\in G_\SHT$, $\rho_1(g\cdot x) = g\cdot \rho_1(x)$; the action of $G_\SHT$ commutes with the retraction $\rho_1$.

\par We shall actually consider a subgroup of $G_\SHT$, namely $G_\SHT^w$ is the subgroup of the $g\in G_\SHT$ that induce on $\SHT$ Weyl automorphisms.

\begin{lemm}\label{3.2b} a) The fixer in $G_\SHT^w$ of an half-apartment $D\subset\SHT$ acts transitively on the apartments of $\SHT$ containing $D$.

\par b) The stabilizer $G_{\SHT,A}^w$ in $G_\SHT^w$ of an apartment $A$ in $\SHT$, induces on $A$ a group $W_G(A)$ containing the reflections with respect to the walls of $A$.
Moreover $W_G(A)$ is the infinite diedral group $W_\SHT(A)$ generated by these reflections.
For $x\in A$, $(G_\SHT^w.x)\cap A=W_G(A).x$.

In particular for $A=\A$, the group $W_G(\A)$ is equal to $\Z\qa_i^\vee\rtimes\{e,r_i\}$.
\end{lemm}

\begin{rema*} The extended tree $\SHT$ may be identified $G_\SHT^w-$equivariently with the product of a genuine (discrete) tree $\SHT^e\subset\SHT$ and the vector space $\ker \qa_i$: if $x$ is a point in a wall of an apartment of $\SHT$ (\eg $x=0\in \ker\qa_i\subset \A$), one may choose $\SHT^e$ to be the convex hull of the orbit $G_\SHT^w.x$.

\end{rema*}

\begin{proof} a) The fixer $G_D$ of $D$ in $G$ fixes {a wall of direction $\ker \qa_i$, hence the corresponding wall direction $M_\infty$. }
So $G_D$ stabilizes $\SHT$ and fixes a chamber in $D\subset\SHT$: one has $G_D \subset G_\SHT^w$ \cite [5.5]{GR14}.
By \cite[2.2.5]{R11}, we know that $G_D$ acts transitively on the apartments of $\SHI$ containing $D$.

\par b) From a) one deduces classically that $W_G(A)$ contains the reflections with respect to the walls of $A$ and that $(G_\SHT^w.x)\cap A=W_G(A).x$.
As the elements in $W_G(A)$ are Weyl-automorphisms, we have $W_G(A)=W_\SHT(A)$.
\end{proof}

\begin{defi}\label{3.3}
A path $\tilde\pi : [0,1]\to \SHT$ that is the image by  $\rho_1$ of a segment $[x,y]$ in $\SHI$ with $(x,y)\in \SHI\times_\leq\SHI$ and $d^v(x,y) = \lambda$ is called a $\SHT$-Hecke path of shape $\lambda$. 
The images by $\qr_2$ of the $\SHT$-Hecke paths of shape $\lambda$ are exactly the Hecke paths of shape $\lambda$ in $\A$ (\cf \ref{1.10} and \cite[Th. 6.3]{GR08}).

Given a $\SHT$-Hecke path $\tilde\pi$ of shape $\lambda$ such that $\tilde\pi(1)\in\mathbb A$, we let $S_{1}(\tilde\pi,\tilde\pi(1))$ be the set of segments $[x,\tilde\pi(1)]$ that retract onto $\tilde\pi$ by $\rho_1$.
\end{defi}

Now if $\pi$ is a Hecke path of shape $\lambda$ in $\mathbb A$, then
$$
\#S_{-\infty}(\qp,\pi(1)) = \sum_{\tilde\pi} \#S_{1}(\tilde\pi,\tilde\pi(1)),
$$ 
where the sum runs over all $\SHT$-Hecke paths $\tilde\pi$ such that $\rho_1\tilde\pi = \pi$ and $\tilde\pi(1) = \pi(1)$. 

\par In particular $S_{1}(\tilde\pi,\tilde\pi(1))$ is finite for any $\SHT$-Hecke path $\tilde\pi$. Formula (\ref{eq3.1}) becomes
\begin{equation}\label{eq3.3}
0 = \sum_{\tilde\pi} \, \qd^{1/2}(\rho_2\tilde\pi(1)). \#S_{1}(\tilde\qp,\tilde\pi(1)).\Gamma_i(\rho_2\tilde\pi),
\end{equation} where the sum runs over all the $\SHT$-Hecke paths $\tilde\pi$ such that $\rho_2\tilde\pi(0) = 0$, $(\rho_2\tilde\qp)'(0)=w.\ql$ or $(\rho_2\tilde\qp)'(0)=w'.\ql$ and $\tilde\pi(1) \in\mathbb A$.

\subsection{Orbits of $\SHT$-Hecke paths}\label{sse3.4}

The set of $\SHT$-Hecke paths is acted upon by $G_\SHT^w$ and decomposes as a disjoint union of orbits. Now, for any fixed orbit $\mathcal O$, we consider the formula
\begin{equation}\label{eq3.4}
0 = \sum
 \, \qd^{1/2}(\rho_2\tilde\pi(1)). \#S_{1}(\tilde\qp,\tilde\pi(1)).\Gamma_i(\rho_2\tilde\pi),
\end{equation} where the sum runs over all the $\SHT$-Hecke paths $\tilde\pi$ in $\mathcal O$ such that $\rho_2\tilde\pi(0) = 0$, $(\rho_2\tilde\qp)'(0)=w.\ql$ or $(\rho_2\tilde\qp)'(0)=w'.\ql$ and $\tilde\pi(1)=\qr_2\tilde\pi(1) \in\mathbb A$. 

This set of $\SHT$-Hecke paths is finite by the finiteness of $n_\ql(\qm)$ (\ref{2.13}) and the fact that $\qm=\tilde\qp(1)$ has to satisfy $\ql\geq \qm \geq w\ql$ or $\geq w'\ql$ (\ref{1.10}.b).

Of course, if, for any orbit $\mathcal O$, equation (\ref{eq3.4}) holds true, then so does equation (\ref{eq3.3}). 

If the condition $(\rho_2\tilde\qp)'(0)=w.\ql$ or $(\rho_2\tilde\qp)'(0)=w'.\ql$ is true for a path $\tilde\pi$ in an orbit $\mathcal O$, then it is true for any path $\tilde\eta = g\cdot \tilde\pi$ in the orbit of $\tilde\pi$, for $G_\SHT^w$ stabilizes $\SHT$, which is defined by $\alpha_i$, and $w = r_iw'$. So we want to prove Formula (\ref{eq3.4}) for all orbits $\mathcal O$ such that the set of $\SHT$-Hecke paths in $\mathcal O$ satisfying the conditions $\rho_2\tilde\pi(0) = 0$, $(\rho_2\tilde\qp)'(0)=w.\ql$ or $(\rho_2\tilde\qp)'(0)=w'.\ql$ and $\tilde\pi(1) \in\mathbb A$ is not empty. And in this case, the condition $(\rho_2\tilde\qp)'(0)=w.\ql$ or $(\rho_2\tilde\qp)'(0)=w'.\ql$ can be removed.

Further, some terms in the RHS of Equation (\ref{eq3.4}) can be factorized. First, since $\rho_1$ and the action of $G_\SHT^w$ commute, $\#S_{1}(\tilde\qp,\tilde\pi(1))$ is constant on an orbit. 
Second, for a given path $\tilde\pi\in\sho$, $\qp=\qr_2\tilde\qp$ is a Hecke-path in $\A$ of shape $\lambda$.
As we require that $\rho_2\tilde\pi(0) = 0$ and $\tilde\pi(1)=\qr_2\tilde\pi(1) \in\mathbb A$, we have $\qr_2\tilde\pi(1)=\tilde\pi(1)\leqslant_{Q^\vee} \lambda$, so $\tilde\pi(1) \in \lambda + \sum_j  \Z\alpha_j ^\vee$.

But now, the intersection of $\mathbb A$ and an orbit of $G^w_\SHT$ through $\tilde\pi$ is an orbit of the Weyl group of $\SHT$
(Lemma \ref{3.2b}.b).
 Since $\tilde\pi(1)\in Y$, $\big (G^w_\SHT\cdot \tilde\pi(1)\big )\cap \mathbb A = \mu_0 + \mathbb Z\alpha_i^\vee$, for some $\mu_0\in Y$. 
We choose such a $\qm_0$ ($\in \lambda + \sum_j  \Z\alpha_j ^\vee$), so $\tilde\qp(1)=\qm_0-a_i\qa_i^\vee$ for some $a_i\in\Z$.

Now, set $\qd^{1/2}_i(\rho_2\tilde\pi(1)) = (q_iq'_i)^{-a_i/2}$ and $\qd^{1/2}_*(\rho_2\tilde\pi(1)) = \qd^{1/2}(\qm_0)$.
Then 
$$
\qd^{1/2}(\rho_2\tilde\pi(1)) = \qd^{1/2}_*(\rho_2\tilde\pi(1))\qd^{1/2}_i(\rho_2\tilde\pi(1)), 
$$ and $\qd^{1/2}_*(\rho_2\tilde\pi(1))$ is constant on the $\SHT$-Hecke paths in an orbit fulfilling the conditions.
 At the end of the day, Equation (\ref{eq3.4}) is equivalent to 
\begin{equation}\label{eq3.5}
0 = \sum_{\tilde\pi\in\mathscr C} \, \qd_i^{1/2}(\rho_2\tilde\pi(1)).\Gamma_i(\rho_2\tilde\pi),
\end{equation} 
where $\mathscr C$ is the finite set of all $\SHT$-Hecke paths $\tilde\qp$ in $\mathcal O$ such that $\tilde\pi(1) \in Y\subset \mathbb A$ and $\rho_2\tilde\pi(0) = 0$.

\begin{rema*} In some simple cases, \eg when the $\tilde\qp\in \sho$ are line segments, the set $\qr_2\SHC=\{\qp=\qr_2\tilde\qp \mid \tilde\qp\in\SHC\}$ has a simple description:
it is the ``orbit'' of some $\qp\in\qr_2\SHC$ under the action of the operators $e_i,f_i$ of Littelmann \cite{Li94}, see also \cite[5.3.2]{GR08}.
So one may directly argue then with Formula \ref{eq3.1}, just considering the $\qp$ in such a ``Littelmann orbit''.
But this does not work in general.
One may also notice that, for two different orbits $\sho_1,\sho_2$ as above, one may sometimes find $\tilde\qp_1,\tilde\qp_2$ in the corresponding sets $\SHC_1,\SHC_2$ with $\qr_2\tilde\qp_1=\qr_2\tilde\qp_2$.
\end{rema*} 

\subsection{Simplify the path}\label{sse3.5}


To any $\SHT$-Hecke path $\tilde\pi$, we associate first $\g s_0(\tilde\pi)$ the chamber of $\SHT$ containing $\tilde\pi([0,1))$, the germ in $0$ of $\tilde\pi$ 
(we saw in \ref{th3.1} that $\qa_i(\qr_2\tilde\qp'(0))\not=0$).
Second, let $(\g s_1(\tilde\pi), \ldots, \g s_n(\tilde\pi))$ be the minimal gallery from $\tilde\pi(0)$ to $\mu = \tilde\pi(1)$, by this we mean the minimal gallery between the chamber containing $\tilde\pi(0)$ closest to $\mu$ and the chamber containing $\mu$, closest to $\tilde\pi(0)$. Recall that the chambers of the extended tree are stripes isomorphic to $[0,1]\times \mathbb R^{\dim V - 1}$ separated by walls of direction $\ker \alpha_i$. Let $M_0(\tilde\pi)$ be the wall containing $\tilde\pi(0)$, $M_{-1}(\tilde\pi)$ be the other wall of $\g s_0(\tilde\pi)$, and recursively, for $j\geqslant 1$, let $M_j(\tilde\pi)$ be the wall of $\g s_j(\tilde\pi)$ distinct from $M_{j-1}(\tilde\pi)$. 

The simplification of $\tilde\pi$ is then the sequence $\mathbf s =\qs(\tilde\pi) = (\g s_0(\tilde\pi), \g s_1(\tilde\pi), \ldots, \g s_n(\tilde\pi),\tilde\qp(1)=\qm)$ of $n+1$ stripes and a point.
Up to conjugation by $G^w_\SHT$, this simplification only depends on $n,\qm$ and if $n\geqslant 1$, on whether or not $\g s_0(\tilde\pi) = \g s_1(\tilde\pi)$.

Let us note that in Formula (\ref{eq3.5}), the summands only depend on $\qs(\tilde\pi)$.

Now let $\tilde\pi$ and $\tilde\eta$ two $\SHT$-Hecke paths in $\mathscr C$ such that $\qs(\tilde\pi) = \qs(\tilde\eta)$. Then $\tilde\pi\in G^w_{\SHT, \g s_0(\tilde\eta), \g s_n(\tilde\eta)}\cdot\tilde\eta$, where $G^w_{\SHT, \g s_0(\tilde\eta), \g s_n(\tilde\eta)}$ is the fixer in $G^w_\SHT$ of $\SHT, \g s_0(\tilde\eta)$ and $\g s_n(\tilde\eta)$ and 
$$
\# \{\tilde\pi\in\mathscr C\mid \qs(\tilde\pi) = \qs(\tilde\eta)\} = \# \big (G^w_{\SHT, \g s_0(\tilde\eta), \g s_n(\tilde\eta)}/G^w_{\SHT,\tilde\eta}\big ).
$$ 
Furthermore, if we let $\tilde\pi = \gamma\cdot\tilde\eta$ {with $\qg\in G^w_\SHT$}, then the groups $G^w_{\SHT,\tilde\pi}$ and $G^w_{\SHT, \g s_0(\tilde\pi), \g s_n(\tilde\pi)}$ are the conjugates by $\gamma$, respectively, of $G^w_{\SHT,\tilde\eta}$ and $G^w_{\SHT, \g s_0(\tilde\eta), \g s_n(\tilde\eta)}$. 
So, the cardinality of the set $\{\tilde\pi\in\mathscr C\mid \qs(\tilde\pi) = \qs(\tilde\eta)\} $ is the same for all $\tilde\eta$ in $\mathscr C$ and we can factorize it in Formula (\ref{eq3.5}). The latter is therefore equivalent to the equality

\begin{equation}\label{eq3.6a}
0 = \sum_{\mathbf s\in\qs(\SHC)} \, \qd_i^{1/2}(\mu).\Gamma_i(\rho_2\mathbf s),
\end{equation} 
where $\mu$ is the last element of $\mathbf{s}$ and 
$$
\Gamma_i(\rho_2\mathbf s) = \Gamma_i(\qm,\rho_2\g s_0) =
\left\{
\begin{array}{ll}
c_q(\alpha_i^\vee)e^{r_i\mu} + b_q(\alpha_i^\vee)e^{\mu} & \hbox{ if } \rho_2\g s_0\in \rho_2(M_0)^+\\
- e^{\mu} & \hbox{ if } \rho_2\g s_0\in \rho_2(M_0)^-.
\end{array}
\right.
$$ 
where $\qr_2(M_0)^-$ (\resp $\qr_2(M_0)^+$) is the half-apartment of $\A$ limited by $\qr_2(M_0)$ and containing (\resp non containing) $-\infty_\SHT$.

Note that if $\mathbf s = \qs(\tilde\pi)$, then $\pi'_+(0) = w'.\lambda \Leftrightarrow \alpha_i(\pi'(0))>0 \Leftrightarrow \rho_2\g s_0\in \rho_2(M_0)^+$.

\bigskip
\parni{\bf Abstraction:} A simplified path in $\SHT$ is a sequence $\mathbf s = (\g s_0, \g s_1, \ldots, \g s_n,\qm)$ of $n+1$ stripes and an element $\qm\in Y$, such that, for $n\geq1$, $\g s_0$ and $\g s_1$ are stripes sharing a same wall $M_0$ of $\SHT$, $(\g s_1, \ldots, \g s_n)$ is a minimal gallery of stripes {(from $M_0$ to $\g s_n$)}  and the wall $M_n$ of $\g s_n$ not shared with $\g s_{n-1}$ contains the element $\mu$.
Up to conjugation by $G^w_\SHT$, such  a simplified path $\mathbf s$ depends only on $n,\qm$ and on whether or not $\g s_0 = \g s_1$. 

Denote by $\mathscr C_s$ the set of all simplified paths in an orbit of $G^w_\SHT$ such that moreover, $\qr_2(M_0)=\ker \qa_i$ and $\qm\in\A$.
By Lemma \ref{3.2b}.b $\qm$ is in $\mu_0 + \mathbb Z\alpha_i^\vee$, for some fixed $\qm_0\in Y\subset\A$.
Clearly an example of such a $\mathscr C_s$ is ${\qs(\SHC)}$.
So the expected Formula \ref {eq3.6a} is a consequence of the following Formula:

\begin{equation}\label{eq3.6}
0 = \sum_{\mathbf s\in\mathscr C_{s}} \, \qd_i^{1/2}(\mu).\Gamma_i(\rho_2\mathbf s),
\end{equation} 

Note also that this requirement is a bit stronger since we do not ask that $\mathscr C_s$ may be written  ${\qs(\SHC)}$.


\subsection{Last computations}\label{sse3.6}

\par{\bf\quad\ 1)} Let $\mathbf s = (\g s_0, \g s_1, \ldots, \g s_n,\qm)$ be a sequence in $\mathscr C_{s}$. Denote the sequence of walls that bound the stripes by $M_j$, for $-1\leqslant j\leqslant n$, where $M_{j}$ and $M_{j+1}$ are the walls of $\g s_{j+1}$. 
Since $M_n$ is in $\mathbb A$, let $k(\mathbf s)$ in $\{-1,0,\ldots, n\}$ be the smallest integer $k$ such that $M_k$ belongs to the half-apartment $D_{-\infty}$ in $\A$ bounded by $M_n$ and containing $-\infty_\SHT$. 
When $\g s_0=\g s_1$ and $M_0\subset D_{-\infty}$, then $M_{-1}=M_1\subset D_{-\infty}$, but we set $k(\mathbf s)=0$.
This integer $k(\mathbf s)$ determines completely $\rho_2 (\mathbf s)$ and allows to split the set $\mathscr C_s$ into disjoints subsets $\{\mathbf s\in\mathscr C_s\mid k(\mathbf s) = k\}$. So Formula (\ref{eq3.6}) is equivalent to
\begin{equation}\label{eq3.7}
0 = \sum_{k=-1 (0)}^{n} \#\{\mathbf s\in\mathscr C_{s}\mid k(\mathbf s) = k\} \, \qd_i^{1/2}(\mu).\Gamma_i(\qm,\rho_2\g s_0),
\end{equation} 
where $\mu$ is the same for all $\mathbf s$ in $\{\mathbf s\in\mathscr C_s\mid k(\mathbf s) = k\}$, and where $k= -1(0)$ means that the sum starts at $-1$ if $\g s_0\ne\g s_1$ and at $0$, otherwise.

\par We know already that, for each $\mathbf s\in\SHC_s$, $\mu \in \mu_0 + \mathbb Z\alpha_i^\vee$, for some fixed $\qm_0\in Y\subset\A$.
But clearly there is some $\mathbf s^0\in\SHC_s$, such that the corresponding $ \g s_1^0, \ldots, \g s_n^0$ are in $\A$, more precisely $k(\mathbf s^0)\leq0$.
For such a $\mathbf s^0$, $\qa_i(\qm^0)=n$; this is this $\qm^0$ that we choose to be $\qm_0$.

If $k(\mathbf s) =-1$, then the whole sequence $\mathbf s = \rho_2(\mathbf s)$ is in $\mathbb A$ and $\alpha_i(\mu) = n$. 
So $\qm=\qm_0$ and $\qd_i^{1/2}(\mu)=1$.
If $k\ne -1$, then $\alpha_i(\mu) = n - 2k$ and $\mu = \mu_0 - k\alpha_i^\vee$. So, in this case, $\qd_i^{1/2}(\mu) = \qd_i^{1/2}(\mu_0 - k\alpha_i^\vee) =  (\sqrt{q_iq'_i})^{-k}$. 

\medskip
\par {\bf 2)} Let us now present the computations of the components of the sum above. First in the case $n=0$.
$$
\begin{array}{|c|c|c|c|}
k & \#\{\mathbf s\mid k(\mathbf s) = k\} &\qd_i^{1/2}(\mu) & \Gamma_i(\qm,\rho_2\g s_0)/e^{\mu_0}\\
\hline
-1 & 1 & 1 & -1 \\
\hline
0 & q_i & 1 & 1/q_i\\
\hline
\end{array}
$$ So Formula (\ref{eq3.7}) holds true in this case! Second, the case $n>0$ and $\g s_0\ne \g s_1$ reads as follows, taking into account that $r_i(\mu) = \mu_0 - (n-k)\alpha_i^\vee$.
$$
\begin{array}{|c|c|c|c|}
k & \#\{\mathbf s\mid k(\mathbf s) = k\} & \qd_i^{1/2}(\mu) & \Gamma_i(\qm,\rho_2\g s_0)/e^{\mu_0}\\
\hline
-1 & 1 & 1 & -1 \\
\hline
0 & q_i-1 & 1 & c_q(\alpha_i^\vee)e^{-n\alpha_i^\vee} + b_q(\alpha_i^\vee) \\
\hline
0<k<n & (q_i '^{*k} - q_i'^{*(k-1)})q_i & (\sqrt{q_iq'_i})^{-k} & c_q(\alpha_i^\vee)e^{-(n-k)\alpha_i^\vee} + b_q(\alpha_i^\vee)e^{-k\alpha_i^\vee}\\
\hline
n & q_iq_i'^{*n} & (\sqrt{q_iq'_i})^{-n} & c_q(\alpha_i^\vee) + b_q(\alpha_i^\vee)e^{-n\alpha_i^\vee}\\
\hline
\end{array}
$$ Third, the table in the case $n>0$ and $\g s_0 = \g s_1$ looks as follows.
$$
\begin{array}{|c|c|c|c|}
k & \#\{\mathbf s\mid k(\mathbf s) = k\} & \qd_i^{1/2}(\mu) & \Gamma_i(\qm,\rho_2\g s_0)/e^{\mu_0}\\
\hline
0 & 1 & 1 & c_q(\alpha_i^\vee)e^{-n\alpha_i^\vee} + b_q(\alpha_i^\vee) \\
\hline
0<k<n  & q_i'^{*k} - q_i '^{*(k-1)} & (\sqrt{q_iq'_i})^{-k} & -e^{-k\alpha_i^\vee} \\
\hline
n & q_i '^{*n} & (\sqrt{q_iq'_i})^{-n} &  -e^{-n\alpha_i^\vee}  \\
\hline
\end{array}
$$ 

To check that all these components sum up to zero, let us remark that if $n=\alpha_i(\mu_0)$ is odd, then $q_i = q'_i$. So we are reduced to two cases: first, $n$ even, then $q_i = q'_i$. We give a detailed account of the first case, leaving the second one to the reader. Denote the right hand side of Formula (\ref{eq3.7}) by $RHS$.

\medskip
\par {\bf 3)} The case $n>0$,  $\g s_0 = \g s_1$ (and $n=2\ell$).
$$
\begin{array}{rcl}

\frac{RHS}{e^{\mu_0}} & = &  \frac{1}{e^{\mu_0}}   \sum_{k=0}^{n} \#\{\mathbf s\mid k(\mathbf s) = k\} \, \qd_i^{1/2}(\mu).\Gamma_i(\qm,\rho_2\g s_0)\\ 

 & = & \sum_{k=1}^{n} \#\{\mathbf s\mid k(\mathbf s) = k\} (\sqrt{q_iq'_i})^{-k} (-e^{-k\alpha_i^\vee}) + c_q(\alpha_i^\vee)e^{-n\alpha_i^\vee} + b_q(\alpha_i^\vee)\\
 
 & = & \sum_{k=1}^{n-1}(q_i '^{*k} - q_i '^{*(k-1)}) (\sqrt{q_iq'_i})^{-k} (-e^{-k\alpha_i^\vee}) + c_q(\alpha_i^\vee)e^{-n\alpha_i^\vee} + b_q(\alpha_i^\vee) \\
 
 & & \qquad + q_i '^{*n}(\sqrt{q_iq'_i})^{-n} (-e^{-n\alpha_i^\vee}) \\
 & = & \sum_{k=1}^{n-1}(q_i '^{*k} - q_i '^{*(k-1)}) (\sqrt{q_iq'_i})^{-k} (-e^{-k\alpha_i^\vee}) \\
& & + q_i^{-1/2}\big (q_i^{-1/2} - b(q_i^{-1/2}, (q'_i)^{-1/2}; e^{\alpha_i^\vee})\big )e^{-n\alpha_i^\vee} + q_i^{-1/2}b(q_i^{-1/2}, (q'_i)^{-1/2}; e^{\alpha_i^\vee}) -e^{-n\alpha_i^\vee}. \\
\end{array}
$$ 

Next, note that if $k = 2s$, then $(q_i '^{*k} - q_i '^{*(k-1)}) (\sqrt{q_iq'_i})^{-k} = (1-q_i^{-1})$ and if $k = 2t + 1$, then $(q_i '^{*k} - q_i '^{*(k-1)}) (\sqrt{q_iq'_i})^{-k} = q_i^{-1/2}(q_i'^{1/2}-q_i'^{-1/2})$. So, we split the last sum above into two sums:

$$
\begin{array}{rcl}
\sum_{k=1}^{n-1}\frac{q_i '^{*k} - q_i '^{*(k-1)}}{\sqrt{q_iq'_i}^{k}}  (-e^{-k\alpha_i^\vee}) & = & \sum_{s=1}^{\ell -1} (1-(q_i)^{-1})(-e^{-2s\alpha_i^\vee})  \\

 & &  
 + \sum_{t=0}^{\ell -1}q_i^{-1/2}(q_i'^{1/2}-q_i'^{-1/2})(-e^{-(2t+1)\alpha_i^\vee}).\\
 
    & = &  (1-(q_i)^{-1})\frac{1-e^{-2(\ell-1)\alpha_i^\vee}}{1- e^{2\alpha_i^\vee}}  \\

 & &  
 + q_i^{-1/2}(q_i'^{1/2}-q_i'^{-1/2})\frac{e^{\qa_i^\vee}-e^{-(2\ell-1)\alpha_i^\vee}}{1- e^{2\alpha_i^\vee}} .\\
 \end{array}
$$ 

Now, remember that
$$
b(q_i^{-1/2}, (q'_i)^{-1/2}; e^{\alpha_i^\vee}) = \frac{q_i^{-1/2}-q_i^{1/2} + ((q'_i)^{-1/2}-(q'_i)^{1/2})e^{\alpha_i^\vee}}{1- e^{2\alpha_i^\vee}}
$$ 

So we get
$$
\begin{array}{rcl}
\frac{RHS(1- e^{2\alpha_i^\vee})}{e^{\mu_0}} & = & 
(1-(q_i)^{-1})(1-e^{-2(\ell-1)\alpha_i^\vee})
+ q_i^{-1/2}(q_i'^{1/2}-q_i'^{-1/2})(e^{\qa_i^\vee}-e^{-(2\ell-1)\alpha_i^\vee})  \\
 
& &  + q_i^{-1/2}\big ( q_i^{-1/2}-q_i^{1/2} + ((q'_i)^{-1/2}-(q'_i)^{1/2})e^{\alpha_i^\vee}\big ) (1-e^{-2\ell\alpha_i^\vee})  \\

  & & + q_i^{-1}(1- e^{2\alpha_i^\vee})e^{-2\ell\alpha_i^\vee} 
  - (1- e^{2\alpha_i^\vee})e^{-2\ell\alpha_i^\vee} \\
  
     & = & (1-(q_i)^{-1})(1-e^{-2(\ell-1)\alpha_i^\vee})
+  q_i^{-1/2}\big ( q_i^{-1/2}-q_i^{1/2} ) (1-e^{-2\ell\alpha_i^\vee}) \\

  & & + q_i^{-1}(1- e^{2\alpha_i^\vee})e^{-2\ell\alpha_i^\vee} 
  - (1- e^{2\alpha_i^\vee})e^{-2\ell\alpha_i^\vee} \\
  
      & = & (1-(q_i)^{-1})\Big(1-e^{-2(\ell-1)\alpha_i^\vee} -1 + e^{-2\ell\alpha_i^\vee} - (1- e^{2\alpha_i^\vee})e^{-2\ell\alpha_i^\vee}\Big )  \\

     & = & 0 \\
  
\end{array}
$$ 

The computation in the case $n$ odd (hence $q_i=q_i'$) is similar and easier: one has not to cut the sum in an even and an odd sum.

\medskip
\par {\bf 4)} The case $n>0$,  $\g s_0 \not= \g s_1$ (and $n=2\ell$).

$$
\begin{array}{rcl}

\frac{RHS}{e^{\mu_0}} & = &  \frac{1}{e^{\mu_0}}   \sum_{k=-1}^{n} \#\{\mathbf s\mid k(\mathbf s) = k\} \, \qd_i^{1/2}(\mu).\Gamma_i(\qm,\rho_2\g s_0)\\ 

  & = & -1 +(q_i-1)(c_q(\alpha_i^\vee)e^{-n\alpha_i^\vee} + b_q(\alpha_i^\vee)) + q_i(c_q(\alpha_i^\vee) + b_q(\alpha_i^\vee)e^{-n\alpha_i^\vee} ) \\
 
 &  &+ \sum_{k=1}^{n-1} (q_i '^{*k} - q_i '^{*(k-1)})q_i (\sqrt{q_iq'_i})^{-k} (  c_q(\alpha_i^\vee)e^{-(n-k)\alpha_i^\vee} + b_q(\alpha_i^\vee)e^{-k\alpha_i^\vee}  ) \\

\end{array}
$$ 

\par Considering the above formulas for $(q_i '^{*k} - q_i '^{*(k-1)}) (\sqrt{q_iq'_i})^{-k} $, we get:

$$
\begin{array}{rcl}

\frac{RHS}{e^{\mu_0}} & = &   -1 +(q_i-1)(c_q(\alpha_i^\vee)e^{-n\alpha_i^\vee} + b_q(\alpha_i^\vee)) + q_i(c_q(\alpha_i^\vee) + b_q(\alpha_i^\vee)e^{-n\alpha_i^\vee} ) \\

  &  &+ \sum_{s=1}^{\ell-1} (q_i-1)(  c_q(\alpha_i^\vee)e^{-(n-2s)\alpha_i^\vee} + b_q(\alpha_i^\vee)e^{-2s\alpha_i^\vee}  ) \\

 &  &+ \sum_{t=0}^{\ell-1} q_i^{1/2}(q_i'^{1/2}-q_i'^{-1/2})(  c_q(\alpha_i^\vee)e^{-(n-2t-1)\alpha_i^\vee} + b_q(\alpha_i^\vee)e^{-(2t+1)\alpha_i^\vee}  ) \\

\end{array}
$$ 

\par We use now $c_q(\qa_i^\vee)=q_i^{-1}-b_q(\qa_i^\vee)$, so:

$$
\begin{array}{rcl}

\frac{RHS}{e^{\mu_0}} & = &   (q_i-1)b_q(\alpha_i^\vee)(1-e^{-2\ell\alpha_i^\vee}) + q_i b_q(\alpha_i^\vee)(e^{-2\ell\alpha_i^\vee}-1 ) \\

  &  &+ \sum_{s=0}^{\ell-1} (q_i-1)q_i^{-1} e^{-(n-2s)\alpha_i^\vee}  \\

 &  &+ \sum_{t=0}^{\ell-1} q_i^{-1/2}(q_i'^{1/2}-q_i'^{-1/2})  e^{-(n-2t-1)\alpha_i^\vee}  \\

  &  &+ \sum_{s=1}^{\ell-1} (q_i-1)b_q(\alpha_i^\vee)(e^{-2s\alpha_i^\vee} - e^{-(n-2s)\alpha_i^\vee}  ) \\

 &  &+ \sum_{t=0}^{\ell-1} q_i^{1/2}(q_i'^{1/2}-q_i'^{-1/2}) b_q(\alpha_i^\vee)(e^{-(2t+1)\alpha_i^\vee}  - e^{-(n-2t-1)\alpha_i^\vee} ) \\
 
\end{array}
$$ 

\par The last two sums are equal to zero, so, using the formula for $b_q(\alpha_i^\vee)$, we get:

$$
\begin{array}{rcl}
\frac{RHS(1- e^{2\alpha_i^\vee})}{e^{\mu_0}} & = &    q_i^{-1/2}\big ( q_i^{-1/2}-q_i^{1/2} + ((q'_i)^{-1/2}-(q'_i)^{1/2})e^{\alpha_i^\vee}\big ) (e^{-2\ell\alpha_i^\vee}-1)  \\

  & & + (1-q_i^{-1}) (e^{-2\ell\alpha_i^\vee}-1) + q_i^{-1/2}(q_i'^{1/2}-q_i'^{-1/2})  (e^{-(2\ell-1)\alpha_i^\vee} -e^{\qa_i^\vee}) \\
  
       & = & 0 \\
  
\end{array}
$$

\par The computation in the case $n$ odd (hence $q_i=q_i'$) is similar and easier.


\section{Representation of the Bernstein-Lusztig-Hecke algebra}\label{s5}

 The aim of this section is to define an action of the  Bernstein-Lusztig-Hecke algebra introduced in  \cite{BPGR16} via Demazure-Lusztig operators.
 In this section and the following {two}, we do not use the masure and all the results  in the preceding sections, except subsections \ref{1.1}, \ref{1.2} and \ref{1.3}.5 (for the definition of $Y$).
 
 \subsection{Notation}\label{5.0}
 
We shall work with the ring $R=\Z_\qs:=\Z[{(\qs_{i}}^{\pm 1}, {\qs'_{i}}^{\pm 1})_{ i\in I} ]$ where the indeterminates ${\qs_{i}}, {\qs'_{i}}$ (for $i\in I$) satisfy the following relations:  if  $\qa_i(Y)=\Z$, then ${\qs_{i}}={\qs'_{i}}$; 
if $r_i$ and $r_j$ are conjugated (i.e. if $\qa_i(\qa_j^\vee)=\qa_j(\qa_i^\vee)=-1$), then ${\qs_{i}}={\qs_{j}}={\qs'_{i}}={\qs'_{j}}$. 
But we shall also consider the subring $\Z'_\qs:=\Z[{(\qs_{i}.\qs'_i}, \qs_i.(\qs'_{i})^{-1})_{ i\in I} ]\supset \Z[(\qs_{i}^2)_{ i\in I} ]$. 
 We may also introduce further equalities between the indeterminates $\qs_i,\qs'_i$, to be closer to the parameters $q_i,q_i'$ of a specific masure.
 For example, the case of a split Kac-Moody group suggests to consider the equality of all $\qs_i,\qs'_i$ to some indeterminate $\qs$; then $\Z_\qs$ (\resp $\Z'_\qs$) becomes $\Z[\qs^{\pm1}]$ (\resp $\Z[\qs^2]$).

The above conditions on the $\qs_i,\qs'_i$ enable us to define $\qs_\qa$ and $\qs'_\qa$ for all $\qa  \in \Phi$ by $\qs_\qa=\qs_{i}$ and $\qs'_\qa=\qs'_{i}$ as soon as $\qa\in W^v\qa_i$. For any reduced decomposition $w=r_{i_1}\ldots r_{i_n}$ of $w\in W^v$, we get also that $\qs_w:=\prod_{k=1}^n\,\qs_{i_k}$ is independent of the decomposition.


\subsection{The Bernstein-Lusztig-Hecke algebra $^{BL}\mathcal H_{\Z_\qs}$}\label{5.2}

In \cite{BPGR16} the Bernstein-Lusztig-Hecke algebra,  $^{BL}\mathcal H_{\Z_\qs}$  is defined as an  associative algebra in the following way.   
 
\par We consider $\A$ as in \ref{1.2} and $Aut(\A)\supset W=W^v\ltimes Y\supset W^a$, with $Y$ a discrete group of translations. We denote by  $^{BL}\mathcal H_{\Z_\qs}$ the free $\Z_\qs-$module with basis $(Z^\ql H_w)_{\ql\in Y, w\in W^v}$.
 For short, we write $H_i=H_{r_i}, H_w=Z^0H_w$ and 
  $Z^\ql=Z^\ql H_e$, where $e$ is the unit element in $W^v$ (and $H_e=Z^0$ will be the multiplicative  unit element in $^{BL}\mathcal H_{\Z_\qs}$).
\begin{prop*} \cite[6.2]{BPGR16}
 There exists a unique multiplication $*$ on $^{BL}\mathcal H_{\Z_\qs}$ which makes it an associative unitary $\Z_\qs$-algebra with unity $H_e$ and satisfies the following conditions:

 (1) $ \forall \ql\in Y \quad  \forall w\in W^v \qquad Z^\ql*H_w= Z^\ql H_w$, 

 \smallskip

  (2) $ \forall i\in I \quad \forall w\in W^v \qquad H_i*H_w= H_{r_iw}\, $ if $\,\ell (r_iw)>\ell (w)$

 \qquad\qquad\qquad\qquad\qquad\quad $H_i*H_w = ({\qs_{i}}-{\qs_{i}}^{-1}) H_w+H_{r_iw}\,$ if $\, \ell (r_iw)<\ell (w)$,

   \smallskip
  (3) $ \forall \ql\in Y \quad  \forall \qm\in Y \qquad Z^\ql*Z^\qm= Z^{\ql+\qm}$,

   \smallskip

  (4) $\forall \ql\in Y \quad  \forall i\in I \qquad H_i*Z^\ql-Z^{r_i(\ql)}*H_i=b(\qs_i, \qs'_i; Z^{-\qa_i^\vee}) (Z^\ql-Z^{r_i(\ql)})$.
  

\end{prop*}
 
\parni{\bf  1)} This $\Z_\qs-$algebra depends only on $\A$ and $Y$ (\ie on $\A$ and $W$) and is called the Bernstein-Lusztig-Hecke algebra over $\Z_\qs$ (associated to $\A$ and $W$). It follows from this proposition (and the remark 2) below) that we obtain a presentation of  $^{BL}\mathcal H_{\Z_\qs}$ by generators $\{H_i, Z^\ql\}_{i\in I, \ql\in Y} $ and relations ((2), (3) and (4) above).

\medskip
\parni{\bf  2)} The sub-algebra with basis $(H_w)_{w\in W^v}$ is the so called  ``Hecke algebra of the group $W^v$ over $\Z_\qs$'' and is denoted by $\mathcal H_{\Z_\qs} (W^v)$. This $\Z_\qs-$algebra  is generated by the $\bigl(H_i\bigr)_{i\in I} $; with relations the braid relations  and $H_i^2=(\qs_i-\qs_i^{-1})H_i+H_e$. 
  We have also, 
   \begin{align*}\forall w,w'\in W^v \qquad H_w*H_{w'}&= H_{ww'}\,\qquad \text { if }\,\ell (ww')= \ell (w)+ \ell(w') \\
  \forall i\in I \,\,\, \forall w\in W^v \qquad H_w*H_i&= H_{wr_i}*H_i*H_i= (\qs_i-\qs_i^{-1})H_w+H_{wr_i}\, \text { if }\, \,\ell (wr_i)<\ell (w).
 \end{align*}

\parni{\bf  3)} The submodule $\mathcal H_{\Z_\qs} (Y)$ with basis $(Z^\ql)_{\ql\in Y}$ is a commutative subalgebra which may be identified to $\Z_\qs[Y]=\{\sum_{\ql\in Y}c_\ql e^{\ql}\, |\,  c_\ql=0  \, \text {, except for a finite number of terms} \}$. In the following, we favor the isomorphism defined by $e^\ql \mapsto Z^{-\ql}$.

\medskip
\parni{\bf  4)} Let us define $T_i:=\qs_iH_i$ and $T_w=\qs_wH_w$. So $(Z^\ql*T_w)_{\ql\in Y,w\in W^v}$ is a new $\Z_\qs-$basis of $^{BL}\mathcal H_{\Z_\qs}$. These new elements satisfy the following relations (for $i\in I, w\in W^v, \ql\in Y$):

\smallskip
 $ T_i*T_w= T_{r_iw}\, $ if $\,\ell (r_iw)>\ell (w)$
and $T_i*T_w = ({\qs_{i}}^2-1) T_w+{\qs_{i}}^2T_{r_iw}\,$ if $\, \ell (r_iw)<\ell (w)$,

 $T_i*Z^\ql-Z^{r_i(\ql)}*T_i=b'(\qs_i, \qs'_i; Z^{-\qa_i^\vee}) (Z^\ql-Z^{r_i(\ql)})$;
  
\parni where $b'(t,u; z)=t.b(t,u;z)={{(t^2-1) + (t.u-t.u^{-1})z}\over {1-z^2}}$.
  
  So these elements generate a $\Z'_\qs-$algebra $^{BL}\mathcal H_{\Z'_\qs}$ which is also given by generators and relations.
  Clearly $^{BL}\mathcal H_{\Z_\qs}=^{BL}\mathcal H_{\Z'_\qs}\otimes_{\Z'_\qs}\Z_\qs$.

\medskip
\parni{\bf  5)} There are several involutive automorphisms or anti-automorphisms of $^{BL}\mathcal H_{\Z_\qs}$: 

\smallskip
\par\quad a) We already remarked in \cite{BPGR16} that 
$$
Z^\ql*H_i-H_i*Z^{r_i(\ql)} =H_i*Z^\ql-Z^{r_i(\ql)}*H_i=b(\qs_i, \qs'_i; Z^{-\qa_i^\vee}) (Z^\ql-Z^{r_i(\ql)})
$$ (just replace $\ql$ by $r_i(\ql)$ in \ref{5.2} (4)).
So we deduce from 1) above that the algebra $^{BL}\mathcal H_{\Z_\qs}$ is $\Z_\qs-$isomorphic to its opposite: the $\Z_\qs-$linear map $\mathfrak{i}: {^{BL}}\mathcal H_{\Z_\qs} \to {^{BL}}\mathcal H_{\Z_\qs}, Z^\ql H_w \mapsto H_{w^{-1}}*Z^\ql
$ is an involutive $\Z_\qs-$anti-automorphism.

\smallskip
\par\quad b) Clearly $\Z_\qs$ has a commutative group of involutive automorphisms $\qf_{\underline\qe,\underline\eta,\underline\eta'}$ (for $\underline\qe=(\qe_i)_{i\in I},\underline\eta=(\eta_i)_{i\in I},\underline\eta'=(\eta'_i)_{i\in I}\in\{\pm1\}^I$, with $\qe_i=\qe_j,\eta_i=\eta_j=\eta'_i=\eta'_j$ if $\qs_i=\qs_j$ and $\eta_i=\eta'_i$ if $\qs_i=\qs'_i$), sending $\qs_i$ to $\qe_i\eta_i\qs_i^{\eta_i}$ and $\qs'_i$ to $\qe_i\eta'_i(\qs'_i)^{\eta'_i}$ (hence also $b(\qs_i,\qs'_i;Z^{-\qa_i^\vee})$ to $\qe_ib(\qs_i,\qs'_i;Z^{-\qa_i^\vee}$).
This group acts equivariantly on the $\Z_\qs-$algebra $^{BL}\mathcal H_{\Z_\qs}$: $\qf_{\underline\qe,\underline\eta,\underline\eta'}$ sends $Z^\ql$ to $Z^\ql$ and $H_i$ to $\qe_iH_i$.
Two of these automorphisms are particularly interesting $\mathfrak{b}=\qf_{\underline\qe,\underline\eta,\underline\eta'}$ for $\qe_i=1,\eta_i=\eta'_i=-1$ for any $i$ and $\mathfrak{j}=\qf_{\underline\qe,\underline\eta,\underline\eta'}$ for $\qe_i=\eta_i=\eta'_i=-1$ for any $i$.

\smallskip
\par\quad c) There is an involutive $\Z_\qs-$anti-automorphism $\mathfrak{a}$ (\resp $\Z_\qs-$automorphism $\mathfrak{a}'=\mathfrak{ai}$) of $^{BL}\mathcal H_{\Z_\qs}$ sending $Z^\ql$ to $Z^{-\ql}
$ (for $\ql\in Y$) and $H_i$ to $
(\qs_i-\qs_i^{-1})-H_i=-H_i^{-1}$ (for $i\in I$).
This is a consequence of the definition of $^{BL}\mathcal H_{\Z_\qs}$ by generators and relations.
We leave to the reader the verification that (4) is satisfied by $\mathfrak{a}(Z^\ql)$ and $\mathfrak{a}(H_i)$ (4 cases to consider, according to the parity and sign of $\qa_i(\ql)$).

\smallskip
\par \quad d) Then $\mathfrak{d}=\mathfrak{ja}'=\mathfrak{jai}$ is an  involutive automorphism: $\mathfrak{d}(h)=\overline{h}$, with $\overline{\qs_i}=\qs_i^{-1}$, $\overline{\qs'_i}=(\qs_i')^{-1}$, $\overline{Z^\ql}=Z^{-\ql}$, $\overline{H_i}=H_i^{-1}$ and $\overline{H_w}=H_{w^{-1}}^{-1}$.
{ Its restriction to $\mathcal H_{\Z_\qs} (W^v)$ is the Kazhdan-Lusztig involution.}

\medskip
\parni{\bf  6)} With notation of Section \ref{s1}, consider the homomorphism $\Z_\qs\to\R,\qs_i\mapsto \sqrt{q_i},\qs'_i\mapsto \sqrt{q'_i}$.
We saw in \cite{BPGR16} that $^{BL}\mathcal H_{\R}=\R\otimes_{\Z_\qs}{^{BL}\mathcal H}_{\Z_\qs}$ contains a subalgebra $^{BL}\mathcal H_{\R}^+$ (with $\R-$basis $Z^\ql H_w$ for $\ql\in Y^+$ and $w\in W^v$), isomorphic to the Iwahori-Hecke algebra $^{I}\mathcal H^+_{\R}$.

We saw in \ref{2.4}.c) that the negative Iwahori-Hecke algebra $^{I}\mathcal H^-_{\R}$ is anti-isomorphic to  $^{I}\mathcal H^+_{\R}$ via a map $\qi$.
So, using the anti-isomorphism $\mathfrak{a}$ above, we see that $^{I}\mathcal H^-_{\R}$ is isomorphic to the $\R-$subalgebra $^{BL}\mathcal H_{\R}^-$ of $^{BL}\mathcal H_{\R}$, generated by the $Z^\ql $ and $H_i$ for $\ql\in -Y^+$ and $i\in I$.

\subsection{The algebra $ \Z_\qs(Y)[W^v]$}\label{5.1}

\par We may replace $\Z_\qs$ by $\Z'_\qs$ everywhere in this subsection.

\subsubsection{The algebra $ \Z_\qs(Y)$}\label{4.3.1}

 We consider, for any $w\in W^v$, the completion $  \Z_\qs[[Y,-wQ^\vee_+]]$ of $  \Z_\qs[Y]$ (in the direction of $-wQ_+^\vee$) as the algebra consisting of all infinite formal sums $f=\sum_{y\in Y} a_ye^y$ (with $a_y\in  \Z_\qs$)  such that $supp(f):=\{ y\in Y\, |\, a_y\not= 0\} \subset \bigcup_{i=1}^n\,(\ql_i- wQ_+^\vee)$ for some $\ql_1, \ldots, \ql_n\in Y$. We may also consider the subalgebra  $  \Z_\qs[[-wQ^\vee_+]]$  of $  \Z_\qs[[Y,-wQ^\vee_+]]$ consisting of the infinite formal sums $f$ such that $supp(f) \subset - wQ_+^\vee$.
We denote the field of fractions of $\Z_\qs[Y]$ by $Fr (\Z_\qs[Y])$
and  consider 
\begin{equation} \label{5.1.2} \Z_\qs(Y) = \bigg(\bigcap_{w\in W^v}  \Z_\qs[[Y,-wQ^\vee_+]]\bigg ) \cap Fr (\Z_\qs[Y])
\end{equation} which clearly is an algebra. Typical elements of this algebra are given using the functions $b$ and $c$ defined at the beginning of \ref{s4}: 
 \begin{equation}\label{5.5.2} b(t,u; z) ={{(t-t^{-1}) + (u-u^{-1})z}\over {1-z^2}}, \quad c(t, u; z)= t-b(t, u ; z).\end{equation} 
Indeed, for $\qa\in \Phi $ (with, as before $\qs_\qa=\qs_i $ and $\qs'_\qa=\qs'_i $ when $\qa\in W^v\qa_i$),  we set 
 \begin{equation}\label{eq:5.5.2}\begin{cases}
 b(\qa^\vee)= b(\qs_\qa, \qs'_\qa ; e^{\qa^\vee}),\quad\qquad\qquad\qquad\quad b'(\qa^\vee)=\qs_\qa b(\qa^\vee)\\
 c(\qa^\vee)= c(\qs_\qa, \qs'_\qa ; e^{\qa^\vee})=\qs_\qa-b(\qa^\vee),\quad\quad c'(\qa^\vee)=\qs_\qa c(\qa^\vee).
 \end{cases}
 \end{equation} 
Let us check that, for any $\qa\in \Phi$, $ b(\qa^\vee)$ and  $ c(\qa^\vee)$ are in $\Z_\qs(Y)$. Indeed (these details will be useful later), 
if  $\qa^\vee\in -Q_+^\vee\cap \Phi^\vee$, then
\begin{equation}\label{5.5.4}\displaystyle 
{b(\qa^\vee)= \biggl((\qs_\qa-{\qs_\qa}^{-1})+(\qs'_\qa-{\qs'_\qa}^{-1})e^{\qa^\vee}\biggr )\sum_{k=0}^{\infty}e^{2k\qa^\vee}.}
\in \qs_\qa^{-1}\Z'_\qs[[e^{2\qa^\vee}]]\subset\Z_{\qs}[[-Q_+^\vee]]
\end{equation}
If $\qa^\vee\in Q_+^\vee\cap \Phi^\vee$, then
\begin{equation}\label{5.5.5} \displaystyle
\begin{cases}
{ b({\qa^\vee})=-e^{-\qa^\vee} \biggl ( (\qs_\qa-{\qs_\qa}^{-1})e^{-\qa^\vee}+(\qs'_\qa-{\qs'_\qa}^{-1})\biggr )(\sum_{k=0}^{\infty}e^{-2k\qa^\vee}})
\\
\qquad\qquad\qquad\qquad\quad
\in\qs_\qa^{-1}e^{-\qa^\vee}\Z'_\qs[[e^{-2\qa^\vee}]]\subset
e^{-\qa^\vee}\Z_{\qs}[[-Q_+^\vee]].
\end{cases}
\end{equation}

By  (\ref{5.5.2}) and (\ref{5.5.5}), $\text {  for } \qa^\vee\in Q_+^\vee\cap \Phi^\vee$: 
\begin{equation}\label{5.5.6} \quad c({\qa^\vee})-\qs_\qa= -b({\qa^\vee})\in \qs_\qa^{-1}e^{-\qa^\vee}\Z'_\qs[[e^{-2\qa^\vee}]]\subset
e^{-\qa^\vee}\Z_{\qs}
[[-Q_+^\vee]]\end{equation} 

by  (\ref{5.5.3c}), and (\ref{5.5.5}), $\text {  for } \qa^\vee\in -Q_+^\vee\cap \Phi^\vee$:
\begin{equation}\label{5.5.7}c({\qa^\vee})-{\qs_\qa}^{-1}=b({\qs_\qa}, {\qs'_\qa}; e^{-\qa^\vee})\in \qs_\qa^{-1}e^{\qa^\vee}\Z'_\qs[[e^{2\qa^\vee}]]\subset
e^{\qa^\vee}\Z_{\qs}.
[[-Q_+^\vee]]\end{equation}

\smallskip
Therefore, for $\qa^\vee\in \Phi^\vee$ and any $w\in W^v$, $b(w^{-1}\qa^\vee)=b(\qs_\qa, \qs'_\qa; e^{w^{-1}\qa^\vee})$ and  $c(w^{-1}\qa^\vee)=c(\qs_\qa, \qs'_\qa; e^{w^{-1}\qa^\vee})$ are in $ \Z_{\qs}
[[-Q_+^\vee]]$, so  \begin{equation}  \label{5.5.7a} b({\qa^\vee}) \text{ and } c({\qa^\vee})\text { are in } \Z_{\qs}[[-w(Q_+^\vee)]]\subset \Z_{\qs}
[[Y, -w(Q_+^\vee)]]. \end{equation} 
 
 And we can say that 

\begin{equation}  \label{5.5.8} b(\qa^\vee) \text{ and } c({\qa^\vee})\text { are in }\Z_{\qs}(Y),\  b'(\qa^\vee) \text{ and } c'({\qa^\vee})\text { are in }\Z'_{\qs}(Y).
 \end{equation}

\subsubsection{The algebra $ \Z_\qs(Y)[W^v]$} \label{4.3.2}

Now, we set 
\begin{equation} \label{5.1.2a}   \Z_\qs(Y)[W^v]:=\{\sum_{w\in F} f_w[w]\, |\, F\text{ is a finite subset of } W^v, f_w\in  \Z_\qs(Y)\}
\end{equation} and we can see that it is an algebra for the multiplication rule 
$$ 
(\sum_{w\in F} c_w[w]).(\sum_{v\in F'} d_v[v])=\sum_{(w,v)\in F\times F'}c_w {(^{w}d_v)}[wv],
$$ where $^w d_v = {^{w}(\sum_{y\in Y} a_ye^y)} = \sum_{y\in Y} a_ye^{wy}$.

\medskip
This algebra $\Z_\qs(Y)[W^v]$ acts on $\Z_\qs(Y) $; explicitly, for $F$ a finite subset of $W^v$  and $\sum_{y\in Y} a_ye^y\in \Z_\qs(Y)$, we have:
 \begin{equation}\label{5.1.2b}
 \ (\sum_{w\in F} c_w[w])*(\sum_{y\in Y} a_ye^y)=\sum_{w\in F}\sum_{y\in Y} c_w a_ye^{wy}.\end{equation}
 
\begin{lemm*}
Consider the morphism of algebras defined by this action: 
\begin{equation}\label{5.1.3}\phi : \Z_\qs(Y)[W^v]\to  End_{\Z_\qs} (\Z_\qs (Y)). 
\end{equation} Then $\phi$ is injective. Actually, the morphism 
\begin{equation}\label{5.1.4}\phi_1: \Z_\qs(Y)[W^v]\to  {Hom}_{\Z_\qs}  (\Z_\qs [Y], Z_\qs(Y))
\end{equation} is already an injection, where $\phi_1(C)$ denotes the restriction of $\phi(C)$ to $\Z_\qs [Y]$.
\end{lemm*}
\begin{proof}
Indeed, let us consider $C=\displaystyle{\sum_{w\in F } \Bigl (\displaystyle{\sum_{\ql\in \cup_{\qm\in F_w} \qm-Q_+^\vee }}} c_{w,\ql}e^\ql\Bigr ) [w]\in Z_\qs(Y)[W^v]$ where $F$ is a finite subset of $W^v$ and $F_w$ finite subsets of $Y$ with some  $(w_1, 
\ql_1)\in F\times F_{w_1}$ such that $c_{w_1, \ql_1}\not=0$. Let us choose $\nu \in Y$ such that  for all $i\in I$,  $\qa_i(\qn)\geq1$ and $N=1+\displaystyle{\max_{\ql\in\cup F_w}}(\HT(\ql))+|\HT (\ql_1)| $, then $\phi_1(C)(e^{Nw_1^{-1}(\qn)})\not=0$. In fact, with such choices, in the calculus of $C*e^{Nw_1^{-1}(\qn)}$, we can see that for $w\not=w_1$, we  have 
$$ 
\HT(\ql+w(Nw_1^{-1}(\qn)))<\HT(\ql_1+w_1(Nw_1^{-1}(\qn)))
$$ and so no term in the sum can cancel the coefficient   $c_{w_1, \ql_1}$ of $e^{\ql_1+N\qn}$. 
\end{proof}

 \subsection{Actions of the  Bernstein-Lusztig-Hecke algebra $^{BL}{\mathcal H_{\Z_\qs}}$}\label{5.5} 
 
\subsubsection{An homomorphism of algebras} \label{5.5.00} 

We are  first going to define an action of $^{BL}\mathcal H_{\Z_\qs}$ on $\Z_\qs[Y]$. To do that,  we follow the idea in \cite{Ma03}. 
By the definition of the Bernstein-Lusztig-Hecke algebra, we have an isomorphism of $\Z_\qs-$modules  $^{BL}\mathcal H_{\Z_\qs}\cong \Z_\qs[Y] \otimes_{\Z_\qs}\mathcal H_{\Z_\qs} (W^v)$  given by $Z^\ql H_w\mapsto e^{-\ql}\otimes H_w$. 
If $M$ is a left $\mathcal H_{\Z_\qs} (W^v)-$module, we may consider the induced $^{BL}\mathcal H_{\Z_\qs}-$module structure on $^{BL}\mathcal H_{\Z_\qs} \otimes_{\mathcal H_{\Z_\qs}(W^v)} M \cong \Z_\qs[Y] \otimes_{\mathcal {\Z_\qs}}M, $ where the $\Z_\qs-$linear isomorphism is given by: 
$$
\sum_{\ql, w} {a_{\ql,w}} Z^\ql H_w\otimes m\mapsto \sum_{\ql, w}a_{\ql , w} e^{-\ql}\otimes H_w m,
$$ where $\ql \in Y$, $w\in W^v$ and $m\in M$  with only finitely many non-zero terms $a_{\ql,w}$. 
 
 By (4) in Proposition  \ref{5.2}, 
\begin{align*}
H_i.(Z^\ql H_w\otimes m) & =  (H_i*Z^\ql *H_w)\otimes m\\
& =  (Z^{r_i(\ql)}* H_i*H_w)\otimes m+b(\qs_i, \qs'_i; Z^{-\qa_i^\vee}) (Z^\ql-Z^{r_i(\ql)}) *H_w\otimes m\\
& =  Z^{r_i(\ql)}\otimes (H_i*H_w).m+b(\qs_i, \qs'_i; Z^{-\qa_i^\vee}) (Z^\ql-Z^{r_i(\ql)})\otimes H_w. m\, . 
\end{align*}
So in terms of $ \Z_\qs[Y] \otimes_{\Z_\qs}M$, changing $\lambda$ in $-\lambda$ above,
  $$(*) \, \, H_i.(e^\ql \otimes H_w. m)= e^{r_i(\ql)}\otimes (H_i*H_w).m+b(\qs_i, \qs'_i; e^{\qa_i^\vee}) (e^\ql-e^{r_i(\ql)})\otimes H_w. m\, .$$
 
The remark \ref{5.2} 2) and the conditions on the $\qs_i $ ensure that we define a structure of $\mathcal H_{\Z_\qs} (W^v)-$module on $M=\Z_\qs$ by setting    $H_i.1=\qs_i$ for all $i\in I$.  The construction  above gives us an action of $^{BL}\mathcal H_{\Z_\qs}$ on  $^{BL}\mathcal H_{\Z_\qs} \otimes_{\mathcal H_{\Z_\qs}(W^v)}\Z_\qs \cong \Z_\qs[Y]  $ seen as an $\mathcal H_{\Z_\qs} (W^v)-$module. 
The action of $H_i$ can then be written as:
 
  $$(*) \, \, H_i.(e^\ql \otimes 1)=\qs_i e^{r_i(\ql)}\otimes 1+b(\qs_i, \qs'_i; e^{\qa_i^\vee}) (e^\ql-e^{r_i(\ql)})\otimes  1\, .
  $$
And so, we can define, an algebra morphism
  
\begin{equation}\label{5.5.1}\varphi : {^{BL}\mathcal H_{\Z_\qs} }\to  {End}_{\Z_\qs}  (\Z_\qs [Y]) \end{equation}

\parni by
 \begin{align*}&\forall i\in I,\quad \varphi(H_i) (f)=\qs_i\, {^{r_i}\!f}+b(\qs_i, \qs'_i, e^{\qa_i^\vee}) (f-{^{r_i}\!f}) \\
&\forall \ql\in Y,\quad  \varphi(Z^\ql) (f)= e^{-\ql} f  & \qquad &\end{align*}
for $ f=\sum a_y e^{y}\in \Z_\qs[Y]$.

\smallskip
Using function $c$, we can rewrite $(*)$ in: 
\begin{align*}
(*) \, \, H_i.(e^\ql \otimes 1)&=c(\qs_i, \qs'_i; e^{\qa_i^\vee}) e^{r_i(\ql)}\otimes 1+b(\qs_i, \qs'_i; e^{\qa_i^\vee}) e^\ql \otimes  1\\
 & = c(\qa_i^\vee) e^{r_i(\ql)}\otimes 1+b(\qa_i^\vee) e^\ql \otimes  1\, .
\end{align*}

Now, thanks to the discussion in \ref{4.3.1}, for each $i\in I$ (resp. each $\ql\in Y$),  the element $h'_i: =c({\qa_i^\vee}) [r_i]+b({\qa_i^\vee}) [e]$ (resp. $e^{-\ql}$) lies in $  \Z_\qs(Y)[W^v]$, stabilizes $\Z_\qs[Y]$ for the action $\phi$, cf.(\ref{5.1.2b}),  and induces on it $\varphi(H_i)$ (resp. $\varphi(Z^\ql)$) cf.(\ref{5.5.1}). As $\phi_1$ is injective, cf.(\ref{5.1.4}),  we get that $(h'_i)_{i\in I} $ and $(e^{-\ql})_{\ql \in Y}$ satisfy the same relations as $ (H_i )_{i\in I}$ and $ (Z^\ql )_{\ql\in Y}$ of Proposition \ref{5.2}. So, by the Remark  \ref{5.2} 1),  we can define an algebra homomorphism by 

\begin{equation} \label{eq:5.5.1} \begin{matrix} \Psi :  { ^{BL}\mathcal H_{\Z_\qs}}& \to &\Z_\qs(Y)[W^v]\\
H_i&\mapsto& c({\qa_i^\vee}) [r_i]+b({\qa_i^\vee}) [e]\\
Z^\ql&\mapsto &e^{-\ql}\, .\end{matrix}\end{equation}

We are going to prove now that $\Psi$ is an injective map and so we will be able to identify $   { ^{BL}\mathcal H_{\Z_\qs}}$ and $ \Psi ( { ^{BL}\mathcal H_{\Z_\qs}})$ and consider the action of $ \Psi ( { ^{BL}\mathcal H_{\Z_\qs}})$ given by $\phi$ (\cf (\ref{5.1.3})) on $\Z_\qs(Y)$  as an action of the Bernstein-Lusztig-Hecke algebra. 

\subsubsection{An explicit formula for $\Psi  (H_w)$} \label{5.5.02}  

First we need to give an explicit formula for $\Psi  (H_w)$ with $w\in W^v$, we introduce new notation. 
For $w\in W^v$ given with a reduced decomposition $w=r_{i_1}r_{i_2}\ldots r_{i_n}$, we will consider for any subset $J$ of $\{1, \ldots, n\}$ and any $1\leq k\leq n$,
 \begin{equation}\label{eq:5.5.3}\begin{aligned} v_k^J:=&\displaystyle \prod_{\substack{1\leq p\leq k\\ p\notin J}} r_{i_p}\\
w_k^J:=&\displaystyle \prod_{\substack{1\leq p\leq k\\ p\notin J}}r_{i_p}\prod_{k+1\leq p\leq n} r_{i_p}\\
w^J:=&w_n^J=v_n^J\, .\end{aligned}\end{equation}

With {this} notation, we have

 $
 \displaystyle {\Psi  (H_w)={\prod_{k=1}^{k=n}\Bigl ( c({\qa_{i_k}^\vee}) [r_{i_k}]+b( {\qa_{i_k}^\vee}) [e]\Bigr )}=\sum_{J\subset [1,\ldots, n]} \Bigl (\prod_{k=1}^{k=n}\widetilde{d^J_{k}} ({\qa_{i_k}^\vee})\Bigr)},$

\parni where $\widetilde{d^J_{k}} ({\qa_{i_k}^\vee})=
\begin{cases}c({\qa_{i_k}^\vee})[r_{i_k}]& \quad \text { if }\,\,  k\notin J \\
b({\qa_{i_k}^\vee})[e] &\quad \text { if } \,\,k\in J \end{cases}\, $, and so 

\begin{equation} \label{eq:5.5.5}\Psi  (H_w)=\sum_{J\subset [1,\ldots, n]} \Biggl (\prod_{k=1}^{k=n}{d^J_{k}} \Bigl ( v^J_{k-1}({\qa_{i_k}^\vee})\Bigr )\Biggr) [w^J]\end{equation}
\begin{equation} \label{eq:5.5.6}\text  { with }\,\,{d^J_{k}} ({\qb^\vee})=
\begin{cases}c({\qb^\vee})& \, \text { if }\,\,  k\notin J \\
b({\qb^\vee}) &\, \text { if } \,\,k\in J\, . \end{cases}\end{equation}
 
\begin{rema*}
In Expression (\ref{eq:5.5.5}), the term corresponding to  $J=\emptyset $ is $\prod_{k=1}^{k=n}{c} ( v^\emptyset _{k-1}({\qa_{i_k}^\vee}) )[w]$. By (\ref{5.5.6}) and (\ref{5.5.7}), $\prod_{k=1}^{k=n}{c} ( v^\emptyset _{k-1}({\qa_{i_k}^\vee}) )\in \Z_\qs[[-Q_+^\vee]]$. 
As the  decomposition $w=r_{i_1}r_{i_2}\ldots r_{i_n}$ is reduced, so is the decomposition $v^\emptyset _{k-1}r_{i_k}=r_{i_1}r_{i_2}\ldots r_{i_k}$ and we know that $v^\emptyset _{k-1}({\qa_{i_k}^\vee})\in Q_+^\vee.$ 
Thus,  the coefficient of the term of height $0$ in the above product is $\displaystyle{\prod_{k=1}^{k=n}{\qs_{i_k}}}$.

\end{rema*}
\smallskip

\subsubsection{The action of ${ ^{BL}\mathcal H_{\Z_\qs}}$}\label{5.5.03} 

We can now prove that  $\Psi$ is an injective map. Consider an element $C=\sum_{w\in  F}\sum_{\ql\in F_w} k_{\ql,w} Z^\ql H_w$ of  ${^{BL}\mathcal H_{\Z_\qs} }\setminus \{0\}$ with $F$ is a finite subset of $W^v$ and for $w\in F$, $F_w$ a finite subset of $Y$ such that $k_{\ql,w}\in \Z_\qs\setminus\{0\}$ if $w\in F$ and $\ql\in F_w$. Let us choose $w_h$ an element of highest length of $F$ and $\ql_0$ of minimal height in $F_{w_h}$. Then we have, 
\begin{align*} \Psi (C)&=\Psi (\sum_{w\in  F}\sum_{\ql\in F_w} k_{\ql,w} Z^\ql H_w)\\
&= \sum_{w\in  F}\sum_{\ql\in F_w} k_{\ql,w} e^{-\ql}\Psi (H_w).\\
\end{align*}
By maximality of $\ell(w_h)$ and (\ref{eq:5.5.5}), the coefficient of $[w_h]$ in  $\Psi (C)$ (seen in $\Z_\qs[[Y,-Q_+^\vee]]$) only comes from the terms in $\Psi (H_{w_h})$ and more precisely those corresponding to $J=\emptyset $ in that decomposition. 
By the remark \ref{5.5.02}, this coefficient is the sum of  {$\biggl (\displaystyle{\prod_{k=1}^{k=n}{\qs_{i_k}}\biggr ) k_{\ql_0,w_h} e^{-\ql_0}}\not= 0$} and others terms in $e^{-\qm}$ with $\qm \in (F_{w_h}+Q_+^\vee)\setminus\{\ql_0\}$ all of height greater than that of $\ql_0$. So, $\Psi(C)\not= 0$.

 
\subsubsection*{Conclusion}\label{5.5.04}

Using (\ref{eq:5.5.1}) and the injectivity of $\Psi$, we now identify an element of ${ ^{BL}\mathcal H_{\Z_\qs}}$ and its image by $\Psi$ in $ \Z_\qs(Y)[W^v]$ and so consider, by (\ref{5.1.2b}),  the action of the Bernstein-Lusztig-Hecke algebra on $\Z_\qs (Y)$, called the Cherednik representation of the Bernstein-Lusztig-Hecke algebra, see \cite[Th. 2.3]{Che95}. 
We get also clearly a representation of ${ ^{BL}\mathcal H_{\Z'_\qs}}$ on $\Z'_\qs (Y)$.
In particular $T_i=\qs_i.H_i=\Psi(T_i)=c'(\qa_i^\vee)[r_i]+b'(\qa_i^\vee)[e]$ acts on $\Z'_\qs (Y)$ and $\Z_\qs (Y)$.

\par The element $H_i=\Psi(H_i)$, viewed as an operator on $\Z_\qs (Y)$, is often called a Demazure-Lusztig operator, see \eg \cite{Che95} or \cite{PaP17}. 

\subsection{An evaluation of $H_w=\Psi(H_w)$}\label{5.6} 

\parni {\bf1)} Using the previous identification, we have 
\begin{equation} \label{5.6.00} H_w=\sum_{J\subset [1,\ldots, n]} \biggl (\prod_{k=1}^{k=n}{ d^J_{k}} \Big ( v^J_{k-1}({\qa_{i_k}^\vee})\Big )\biggr) [w^J].\end{equation}
Since for any $w_0\in W^v$, each $d^J_{k} ( v^J_{k-1}({\qa_{i_k}^\vee}))\in \Z_\qs[[-w_0Q_+^\vee]]$ cf.(\ref{5.5.7a}), $ H_w$ is an element of $  \Z_\qs[[-w_0Q_+^\vee]][W^v]$. In fact, by (\ref{5.5.5}) and (\ref{eq:5.5.6}),  we can say that
\begin{equation} \label{5.6.01}d^J_{k} ( v^J_{k-1}({\qa_{i_k}^\vee}))\in e^{-v^J_{k-1}({\qa_{i_k}^\vee})}\Z_\qs[[-w_0Q_+^\vee]]\text{ whenever }k\in J \text{ and  }w_0^{-1}(v^J_{k-1}({\qa_{i_k}^\vee}))\in\! Q_+^\vee.\end{equation} 
So, 
\begin{equation}\label{5.6.0}  \displaystyle \prod_{k=1}^{k=n}{d^J_{k}} \Bigl ( v^J_{k-1}({\qa_{i_k}^\vee})\Bigr ) \in e^{-\sum_k {v^J_{k-1}({\qa_{i_k}^\vee})}}\Z_\qs[[-w_0Q_+^\vee]] =
 e^{w_0(w_0^{-1}(-\sum_k {v^J_{k-1}({\qa_{i_k}^\vee})))}}\Z_\qs[[-w_0Q_+^\vee]] \end{equation}
where the sum runs over the $k\in J$ such that  $w_0^{-1}(v^J_{k-1}({\qa_{i_k}^\vee}))\in Q_+^\vee$.

\smallskip
\parni {\bf2)}
To be able to write more precisely  this expression in  $\Z_\qs[[-w_0Q_+^\vee]]$, we are going to  prove the following proposition with arguments using galleries in $\A^v$.
The aim of the proposition is to give us some information on the "height" of $w_0^{-1}(-\sum_k v^J_{k-1}({\qa_{i_k}^\vee}))$ in relation with the difference in length of $w$ and  $w^J$.

\begin{prop*}  For $ r_{i_1}\ldots r_{i_n}$  a reduced decomposition of $w\in W^v$ and $J$ a subset of $[1,\ldots, n]$, with the notation of (\ref{eq:5.5.3}), 

$$2\HT \Biggl ( \displaystyle  \sum_{\substack{k\in J \\ w_0^{-1}v_{k-1}^J(\qa_{i_k})\in Q_+}} {w_0^{-1}\big(v^J_{k-1}({\qa_{i_k}^\vee})\big)}\Biggr )\geq\ell(w_0^{-1}w)- \ell(w_0^{-1}w^J).$$
\end{prop*}
\begin{rema*}
This technical result has a nice geometric interpretation for $w_0=1$.\\
If a gallery  $\Bigl (C_0=C^v_f,C_1, C_2, \ldots , C_n\Bigr )$ of type $(i_1,\ldots, i_n)$ where the decomposition $r_{i_1}\ldots r_{i_n}$ is  reduced,  is not minimal, then $n-d(C_0,C_n)$  is at most twice the number of walls crossed by the gallery in a direction opposite to $C_0$, each wall $\ker{\qa}$ being counted with a multiplicity equal to the height of the coroot $\qa^\vee$.
\end{rema*} 

\begin{proof}
In the standard vectorial apartment $\A^v$, we consider the minimal gallery $\Gamma^0$ from $C^v_f$ to $wC^v_f$ of length  $n$ and of  type $( {i_1},\ldots ,{i_n})$. More explicitely,   $$\Gamma^0=\Bigl (C_0^0=C^v_f,C_1^0=r_{i_1}(C_0^0), C_2^0=r_{i_1}r_{i_2}(C_0^0), \ldots , C_n^0=r_{i_1}r_{i_2}\ldots r_{i_n} (C_0^0)\Bigr ).$$

We construct a finite sequence of galleries $(\Gamma^k)_{k\in [0,\ldots , n]}$ such that  $\Gamma^k=\Bigl (C_0^k,C_1^k, C_2^k, \ldots , C_n^k\Bigr )$ for $k\geq 1$ is a gallery from  $C_0^k:=C_0^0=C^v_f$ to $C_n^k:=w_k^J(C^v_f)$ of  type $( {i_1},\ldots ,{i_n})$ , as follows : 

\parni if $ k\notin J $, 
$\Gamma^{k}=\Gamma^{k-1}$, 

\parni  if $k\in J$,  $\Gamma^{k}$ is obtained from $\Gamma^{k-1}$ by a folding : 

for $p\leq k-1$, $C_p^{k}=C_p^{k-1},$  

for $p\geq k$, $C_p^{k}=r_{v_{k-1}^J(\qa_{i_k})}C_p^{k-1}$ $\bigr($in particular  $C_k^{k}= C_{k-1}^{k-1}=v_k^J(C_0^0)$ and  $C_n^{k}=w_k^J(C_0^0)\bigr)$.\\

Using the distance between two chambers, we have, for all ${k\in [0,\ldots , n]}$, 
$$\ell (w_0^{-1}w_k^J)=d(C_f^v, w_0^{-1}w_k^J(C_f^v))=d (w_0(C_f^v), w_k^J(C_f^v)) =d (w_0(C_f^v), C_n^k).$$

And so, 

\begin{align*} \ell (w_0^{-1}w)-\ell (w_0^{-1}w^J)&=d (w_0(C_f^v), C_n^0) -d (w_0(C_f^v), C_n^n)\\
& =\displaystyle {\sum_{k=1}^{k=n} d (w_0(C_f^v), C_n^{k-1}) -d (w_0(C_f^v), C_n^k)}\\
&=\displaystyle {\sum_{k\in J} d (w_0(C_f^v), C_n^{k-1}) -d (w_0(C_f^v), C_n^k)}\\
& =\displaystyle {\sum_{k\in J} d (C_f^v,  w_0^{-1}w_{k-1}^J(C_f^v)) -d (C_f^v,w_0^{-1} w_k^J(C_f^v))}\\
&=\displaystyle {\sum_{k\in J}  \ell (w_0^{-1}w_{k-1}^J)-\ell (w_0^{-1}w_k^J).}\end{align*}

If $k\in J$ and  $w_0^{-1}v_{k-1}^J(\qa_{i_k})\in Q_-$, then $v_{k-1}^J(\qa_{i_k})(w_0(C_f^v))<0$, so the chambers $w_0(C_f^v)$ and $C_n^{k-1}$ are on the same side of the wall $\ker({v_{k-1}^J(\qa_{i_k})})$,  as $v_{k-1}^J(\qa_{i_k})$ is positive on $C_{k-1}^{k-1} $ and negative on $C_{p}^{k-1} $ for $p\geq k$ (because the gallery $\Gamma^{k-1}$ is stretched between $C_{k-1}^{k-1} $ and  $C_{n}^{k-1} $). 
So, for $k\in J$ and  $w_0^{-1}v_{k-1}^J(\qa_{i_k})\in Q_-$,  \begin{align*}  \ell (w_0^{-1}w_{k-1}^J)-\ell (w_0^{-1}w_k^J))&= d (w_0(C_f^v), C_n^{k-1}) -d (w_0(C_f^v), C_n^k)\\
&=d (w_0(C_f^v), C_n^{k-1}) -d (w_0(C_f^v), r_{v_{k-1}^J(\qa_{i_k})}C_n^{k-1}) \leq0 \end{align*}

With this result, we obtain the upper bound 
$$ \ell (w_0^{-1}w)-\ell (w_0^{-1}w^J)\leq \displaystyle  \sum_{\substack{k\in J \\ w_0^{-1}v_{k-1}^J(\qa_{i_k})\in Q_+}} \ell (w_0^{-1}w_{k-1}^J)-\ell (w_0^{-1}w_k^J).$$

If $k\in J$ and   $w_0^{-1}v_{k-1}^J(\qa_{i_k})\in Q_+$, we get by the triangle inequality, 

\begin{align*}d (w_0(C_f^v), C_n^{k-1})-d (w_0(C_f^v),r_{v_{k-1}^J(\qa_{i_k})}C_n^{k-1})\!& =d (w_0(C_f^v), C_n^{k-1})-d ( r_{v_{k-1}^J(\qa_{i_k})}w_0(C_f^v),C_n^{k-1})\\
&\leq d (w_0(C_f^v),  r_{v_{k-1}^J(\qa_{i_k})}w_0(C_f^v))\, .\end{align*}
 
 And we can write finally \begin{align*} \ell (w_0^{-1}w_{k-1}^J)-\ell (w_0^{-1}w_k^J)=d (w_0C_f^v, C_n^{k-1}) -d (w_0C_f^v, r_{v_{k-1}^J(\qa_{i_k})}C_n^{k-1})& \leq d(w_0C_f^v,  r_{v_{k-1}^J(\qa_{i_k})}w_0C_f^v)\\
 &\leq d(C_f^v,  w_0^{-1}r_{v_{k-1}^J(\qa_{i_k})}w_0C_f^v)\\
 &\leq d(C_f^v, r_{w_0^{-1}v_{k-1}^J(\qa_{i_k})}C_f^v)\\
 &\leq \ell\big(r_{w_0^{-1}v_{k-1}^J(\qa_{i_k})}\big)\, .
 \end{align*}
 
We  use a little lemma (proven just later)   to conclude that 
$$\ell (w_0^{-1}w)-\ell (w_0^{-1}w^J)\leq 2 \HT \Biggr( \displaystyle  \sum_{\substack{k\in J \\ w_0^{-1}v_{k-1}^J(\qa_{i_k})\in Q_+}} {w_0^{-1}(v^J_{k-1}({\qa_{i_k}^\vee})})\Biggr)\, .$$
\end{proof}

\begin{lemm*}  If $\qa$ is a real positive root, then $\ell (r_\qa) \leq 2  \HT (\qa^\vee)$. 
\end{lemm*}
\begin{proof}
It  is possible to find $(j_1,j_2, \ldots j_q)$ such that $r_{j_{q-1}}\ldots r_{j_1}(\qa^\vee)=\qa_{j_q}^\vee $ with for all $2\leq  h\leq q-1$, $\qa_{j_h}(r_{j_{h-1}}\ldots r_{j_1}(\qa^\vee) ) > 0$ {and $\qa_{j_1}(\qa^\vee)>0$}. 
We have $r_{j_{h-1}}\ldots r_{j_1}(\qa^\vee)-r_{j_{h}} r_{j_{h-1}}\ldots r_{j_1}(\qa^\vee)\in Q^\vee_+\setminus\{0\}$ so obviously $q\leq \text{\rm  ht} (\qa^\vee) $. 
Then $r_\qa=r_{j_1} \ldots r_{j_{q-1}}r_{j_q}r_{j_{q-1}}\ldots r_{j_1}$ so $\ell(r_\qa ) \leq 2q-1\leq 2\text{\rm  ht} (\qa^\vee) $. 
\end{proof}

\smallskip
\parni {\bf3)}
Using Proposition \ref{5.6}.2) and (\ref{5.6.0}), we have in $\Z_\qs[[Y,-w_0Q_+^\vee]]$, 
\begin{equation}\label{5.6.01}\prod_{k=1}^{k=n}{d^J_{k}} \Bigl ( v^J_{k-1}({\qa_{i_k}^\vee})\Bigr )=:\sum_{\substack{\ql\in Q_+^\vee\\2\HT(\ql)\geq \ell(w_0^{-1}w)-\ell(w_0^{-1}w^J )}}d^w_{J,\ql,w_0} e^{-w_0\ql}\quad \text {with  some } d^w_{J,\ql,w_0}\in \Z_\qs.
\end{equation}
We can now, rewrite the expression (\ref{5.6.00}) in $\Z_\qs[[Y, -w_0Q_+^\vee ]]$, 

$$ H_w=\sum_{J\subset [1,\ldots, n]}\sum_{\substack{\ql\in Q_+^\vee\\2\HT(\ql)\geq \ell(w_0^{-1}w)-\ell(w_0^{-1}w^J)}}d^w_{J,\ql,w_0} e^{-w_0\ql}[w^J  ] 
.$$

 With the notation $\leq_{_{B}}$ for the Bruhat order,  if $v\leq_{_B} w$ there exists $J$ such that $w^J=v$, so we have the equality
 \begin{equation}\label{5.6.1} H_w=:\sum_{v\leq_{_B} w}\sum_{\ql\in Q_+^\vee}k^w_{v,\ql,w_0} e^{-w_0\ql}[v]\end{equation} for some $k^w_{v,\ql,w_0}\in \Z_\qs$ and $k^w_{v,\ql,w_0}=0 \text { when } 2\HT(\ql)< \ell(w_0^{-1}w)-\ell(w_0^{-1}v)$.
In the last expression, we have  written $ H_w\in \Z_\qs[[-w_0Q_+^\vee]][W^v]$ by grouping the terms $J$  that correspond to the same element $v=w^J$ with $v\leq_{_B} w$. 

In particular, if we consider the case $w_0=e$ and  $v=e$, the coefficient of $[e]$ in the sum (\ref{5.6.1}) is $\sum_{\substack{J\subset [1,\ldots, n]\\w^J=e}}\sum_{\substack{\ql\in Q_+^\vee\\2\HT(\ql)\geq \ell(w)}}d^w_{J,\ql,e} e^{-\ql}.$
And more specifically, 
\begin{equation} \label{5.6.1a} \text{the coefficient (in }\Z_\qs \text {) of }e^{0}[e] \text{ is  } k^w_{e,0,e}=0 \text { if } w\not=e . \text{ When } w=e,  \text{ one has } H_e=[e].\end{equation}


\section{Symmetrizers and Cherednik's identity}\label{s5b} 


The first aim of this  section is to give a sense to the $T-$symmetrizer $P_\qs=\sum_{w\in W^v} T_w=\sum_{w\in W^v} \qs_w H_w$ as an element of $ \Z_\qs ((Y))[[W^v]]$, where  $\Z_\qs ((Y))$ is a completion of the algebra $\Z_\qs (Y)$ defined below and $ \Z_\qs ((Y))[[W^v]]$ is the $\Z_\qs ((Y))-$module of all formal expressions $\sum_{w\in W^v}a_w[w] $ with $a_w\in \Z_\qs((Y))$. 
We also introduce in $ \Z_\qs ((Y))[[W^v]]$ a $\QD-$symmetrizer $\SHJ_\qs$ and prove its proportionality with $P_\qs$.

\subsection{Definition of the $T-$symmetrizer $P_\qs$} \label{5.7.01} 
For any $N\in \N$, we set $W_N^v=\{w\in W^v\,|\, \ell(w)\leq N\}$ and we consider the element of ${ ^{BL}\mathcal H_{\Z_\qs}}$,  \begin{equation} \label{5.6.1c}P_\qs^N=\sum_{w\in W^v_N} \qs_wH_w . \end{equation}
And so by (\ref{5.6.1}), we have, in $\Z_\qs[[-w_0Q_+^\vee]][W^v]$,

 \begin{align*}P_\qs^N=\sum_{w\in W^v_N} \qs_wH_w &=\sum_{w\in W^v_N} \qs_w\sum_{v\leq_{_{B}} w}\sum_{\ql\in Q_+^\vee}k^w_{v,\ql,w_0} e^{-w_0\ql}[v]\\
&=\sum_{v\in W^v_N} \sum_{\substack{w\in W^v_N\\ v\leq_{_{B}} w}}\qs_w\sum_{\ql\in Q_+^\vee}k^w_{v,\ql,w_0} e^{-w_0\ql}[v]
\end{align*}

And so, we have 

\begin{equation}\label{5.6.2} P_\qs^N=\sum_{v\in W^v_N} \sum_{\ql \in Q_+^\vee} C^N_{v,\ql; w_0}  e^{-w_0\ql}[v]\end{equation} with 
\begin{equation}\label{5.6.2a}C^N_{v,\ql; w_0}= \displaystyle\sum_{w\in W^v_N\,|\,  v\leq_{_{B}} w}\qs_w k^w_{v, \ql,w_0} .\end{equation}

Now, by (\ref{5.6.1}),  we can say that, for given $w_0\in W^v$, $v\in W^v$ and $\ql \in Q_+^\vee $,   the elements $w\in W^v$ which have a contribution to the coefficient  $C^N_{v,\ql; w_0}$  (for $N$ arbitrary large) are such that $\ell (w_0^{-1}w)-\ell (w_0^{-1}v)\leq 2 \HT(\ql) $. Such a  $w$ is in $w_0W^v_{\ 2\HT(\ql) +\ell(w_0^{-1} v)}$ which is a  finite set. Therefore, the coefficient $C^N_{v,\ql; w_0}$ is independent of $N$ for $N$ large enough.
Based on the formula (\ref{5.6.2}),  we can now consider, for any $w_0\in W^v$,
\begin{equation}\label{5.6.3}P_\qs=\sum_{w\in W^v} \qs_wH_w=\sum_{v\in W^v} \sum_{\ql \in Q_+^\vee} C_{v,\ql; w_0}  e^{-w_0\ql}[v]\text{ with }C_{v,\ql; w_0}=\lim\limits_{N\to+\infty}C^N_{v,\ql; w_0}. \end{equation}

\begin{remas*}1) Let us notice that, using (\ref{5.6.1a}) and equality $P_\qs=\sum_{w\in W^v} \qs_wH_w$, we find that  $C_{e,0;e}$ is $1$, as $\qs_eH_e$ is the only term which has a contribution to the coefficient of $e^0[e]$. 

2) For each $N$, and each $w_0\in W^v$,  $P_\qs^N$ is  an element of {$\Z_\qs(Y)[W^v]$} but the coefficient of $[v]$ in $P_\qs$,  $ \sum_{\ql \in Q_+^\vee} C_{v,\ql; w_0}  e^{-w_0\ql}$, is perhaps not any more in $\Z_\qs(Y)\subset Fr(\Z_\qs [Y])$. 

\medskip
We need to consider the completion $\Z_\qs ((Y))$ of $\Z_\qs(Y)=Fr (\Z_\qs[Y])) \cap (\cap_{w\in W^v}  \Z_\qs[[Y,-wQ^\vee_+]])$ where a sequence $\phi_n$ is said to converge if it converges in each $\Z_\qs[[Y,-wQ^\vee_+]]$. We still can consider the action of $W^v$ on  $\Z_\qs ((Y))$. 
\end{remas*} 

With this definition, \begin{equation}\label{5.6.4} P_\qs\in \Z_\qs ((Y))[[W^v]] \text{ and can be written } P_\qs=:\sum_{v\in W^v}C_v[v]\end{equation}
 \begin{equation}\label{5.6.4a} \text{ with } C_v:= \sum_{\ql \in Q_+^\vee} C_{v,\ql; w_0}  e^{-w_0\ql}\in  \Z_\qs ((Y)).\end{equation}
This result is the generalization of \cite[Lemma 2.19]{CheM13}, see also \cite[7.3.13]{BrKP15}. In particular the fact that we can use $\Z_\qs$, and do not have to consider an appropriate completion $\hat{\Z}_\qs$, is a translation of $P_\qs $ being "$v-$finite" (in the language of \cite{BrKP15}). Proposition \ref{5.6}.2) above is the key step for the generalization of this result of Cherednik and Ma \cite{CheM13}.

In the case of a split symmetrizable Kac-Moody group, the same result was obtained independently by Patnaik and Pusk\'as in \cite[Th. 3.1.8]{PaP17}. They get this result, except $v-$finiteness by the same methods as in \cite{CheM13}, and then deduce the $v-$finiteness from \ref{5.6.9} below and a result of Viswanath \cite{Vi08}.

\subsection{$W^v-$invariance of $P_\qs$}\label{5.7.02}

We must pay attention that $ \Z_\qs ((Y))[[W^v]]$ is not an algebra. But if $f \in \Z_\qs ((Y))[[W^v]]$ and  if  $g,h$ are in $ \Z_\qs ((Y))[W^v]$  the products $gf$, $fh$ are  well defined in $ \Z_\qs ((Y))[[W^v]]$. One just uses the action of $W^v$ on $\Z_\qs((Y))$ and the formulas $[w] k=(^wk)[w]$, $[w][w']=[ww']$ for $w,w'\in W^v$ and $k\in \Z_\qs((Y))$.

\begin{lemm*}  

(i) For all $i\in I$,  $ H_iP_\qs=\qs_iP_\qs\quad \text {and }\quad  P_\qs H_i=\qs_iP_\qs$.

(ii) With the notation introduced in (\ref{5.6.4}) and (\ref{5.6.4a})  with $w_0=e$, let us consider $\Gamma := C_e$ (which is invertible in $\Z_\qs[[-Q_+^\vee]]$ and in $\Z_\qs((Y))$). Then $C_v={^v\QG}$, for any $v\in W^v$.

 In $ \Z_\qs ((Y))[[W^v]]$, $P_\qs$ can be written $$ P_\qs=\sum_{v\in W^v}(^{v}\Gamma )[v] =\sum_{v\in W^v}[v]\Gamma\quad  \text{  with }\Gamma \in  \Z_\qs ((Y)). $$

\end{lemm*} 
\begin {proof}
(i) Using the definition of $P_\qs$ and Proposition \ref{5.2}, we can write  
 \begin{align*} H_iP_\qs=\sum_{w\in W^v} \qs_wH_iH_w &=\sum_{\substack{w\in W^v\\\ell(r_iw)>\ell(w)}} \qs_w H_{r_iw}+\sum_{\substack{w\in W^v\\\ell(r_iw)<\ell(w)}} (\qs_w H_{r_iw}+(\qs_i-\qs_i^{-1})\qs_wH_w)\\
&=\sum_{\substack{w\in W^v\\\ell(r_iw)<\ell(w)}}\qs_i^{-1} \qs_w H_{w}+\sum_{\substack{w\in W^v\\ \ell(r_iw)<\ell(w)}} (\qs_w H_{r_iw}+(\qs_i-\qs_i^{-1})\qs_wH_w)\\
&=\sum_{\substack{w\in W^v\\ \ell(r_iw)<\ell(w)}} \qs_w H_{r_iw}+\sum_{\substack{w\in W^v\\ \ell(r_iw)<\ell(w)}} \qs_i\qs_wH_w\\
&=\sum_{\substack{w\in W^v\\ \ell(r_iw)>\ell(w)}} \qs_i\qs_w H_{w}+\sum_{\substack{w\in W^v\\ \ell(r_iw)<\ell(w)}} \qs_i\qs_wH_w\\
&=\qs_i P_\qs\end{align*}
 (To be more precise, the reader can consider in the equalities above $P_\qs^N$ instead of $P_\qs$ in order to stay in $^{BL}\mathcal H_{\Z_\qs}$ for using formula of Proposition \ref{5.2}  and obtain (i) when N tends to infinity.)

The second equality is obtained in the same way, using the remark 2) of  \ref{5.2}. 

\medskip
(ii) In $ \Z_\qs((Y))[W^v]$, we know that $H_i= c(\qa_i^\vee) [r_i]+b(\qa_i^\vee) [e]$ . Moreover,  $c(\qa_i^\vee) +b(\qa_i^\vee)=\qs_i$ by (\ref{5.5.2}). So, the equality  $H_iP_\qs =\qs_i P_\qs$ shows that $[r_i]P_\qs=P_\qs$ and we can write in $ \Z((Y))[[W^v]]$: $$ 0= ( [r_i]- [e])P_\qs 
= \sum_{v\in W^v} ( (^{r_i}C_v)[r_iv]-C_v[v]) .$$
Hence $^{r_i}C_v=C_{r_iv}$ for any $v\in W^v$ and any $i\in I$. We deduce from that $C_v=^v\!C_e$, and so  $$P_\qs=\sum_{v\in W^v}C_v[v]=\sum_{v\in W^v}[v]C_e=\sum_{v\in W^v}[v]\Gamma .$$

By the first remark of \ref{5.7.01}.1), $\Gamma= C_e= \sum_{\ql \in Q_+^\vee} C_{e,\ql; e}  e^{-w_0\ql}=1+ \sum_{\ql \in Q_+^\vee\setminus\{0\}} C_{e,\ql; e}  e^{-w_0\ql}$ is an invertible element in {$\Z_\qs[[-w_0Q_+^\vee]]$ for any $w_0$, hence in $\Z_\qs((Y))$}.  
\end{proof} 

\subsection{Definition of the $\QD-$symmetrizer $\SHJ_\qs$} \label{5.7.03} 

For $\qa\in \Phi_+$, we know by (\ref{5.5.7}) that $\QD_\qa:=\qs_\qa c(-\qa^\vee)=c'(-\qa^\vee)$ lies in $1+ e^{-\qa^\vee}\Z'_\qs [[-Q_+^\vee]]$. 
So we can consider the infinite product 
\begin{equation}\label{5.6.6} \Delta = \prod_{\qa\in \Phi_+} c'(-\qa^\vee) \in \Z'_\qs[[-Q_+^\vee]] \text { which is  invertible in } \Z'_\qs[[-Q_+^\vee]] . \end{equation}
Moreover, for any $w_0\in W^v$,

-  if $w_0^{-1}(\qa )\in \Phi_+$,  $  c'(-w_0^{-1}(\qa^\vee))-1\in e^{-w_0^{-1}(\qa^\vee)}\Z_\qs [[-Q_+^\vee]]$; 

- else  $w_0^{-1}(\qa )\in \Phi_-$,  $ c'(-w_0^{-1}(\qa^\vee))-\qs_\qa^2\in e^{w_0^{-1}(\qa^\vee)}\Z_\qs [[-Q_+^\vee]]$ by (\ref{5.5.6}). 

\smallskip
So, we can also consider $ \Delta$ as an invertible element of $ \Z_\qs[[-w_0(Q_+^\vee)]]$, but not of $ \Z'_\qs[[-w_0Q_+^\vee]]$. Finally, 
\begin{equation}\label{5.6.7} \Delta = \prod_{\qa\in \Phi_+} c'(-\qa^\vee) \in \Z_\qs((Y)) \text { and  is  invertible. } \end{equation}
It may be interesting to notice that 
\begin{equation}\label{5.6.8} c'(-\qa_i^\vee).(^{r_i}\Delta)=c'(\qa_i^\vee).\Delta.\end{equation}

\par Now the $\QD-$symmetrizer is the following formal sum $\SHJ_\qs:=\sum _{w\in W^v}{^{w}\!\Delta} \, [w]\in \Z_\qs ((Y))[[W^v]]$.

\subsection{The Cherednik's identity}\label {5.6.9}

\begin{prop*}  In $\Z_\qs ((Y))[[W^v]]$, we have the equality
 $$ P_\qs=\sum_{w\in W^v} \qs_wH_w=\mathfrak{m}_\qs\!\SHJ_\qs=\ (\Gamma \Delta^{-1})\!\sum _{w\in W^v}{^{w}\!\Delta} \,\, [w]$$ 
and  $\mathfrak{m}_\qs=\Gamma \Delta^{-1}\in \Z_\qs((Y))$ is invertible and $ W^v-$invariant.
\end{prop*}

\begin{NB} This generalizes Theorem 2.18 in \cite{CheM13}, except that we do not give an explicit formula for the multiplier $\mathfrak{m}_\qs=\Gamma \Delta^{-1}$.
\end{NB}

\begin{proof} 
We first want to prove that  $\mathfrak{m}_\qs:=\Gamma \Delta^{-1}$ is $W^v-$ invariant in $\Z_\qs((Y))$, \ie that $$\forall i\in I,\,\,  \Gamma(^{r_i}\!\Delta)= ({ ^{r_i}\!\Gamma})\Delta,$$ 
or more simply,  by (\ref{5.6.8}), that: $$\forall i\in I,\,\, {^{r_i}\!\Gamma} c'(-\qa_i^\vee)=\Gamma  c'(\qa_i^\vee)\text{ in }  \Z_\qs((Y)).$$ 

It's easy to see  in $\Z_\qs ((Y))[[W^v]]$, that ${^{r_i}\!\Gamma} c'(-\qa_i^\vee)$ is the coefficient of $[e]$ in $P_\qs([r_i]c'(-\qa_i^\vee))$   and $\Gamma c'(\qa_i^\vee)$, the coefficient of $[e]$ in $P_\qs(c'(\qa_i^\vee))$.

Moreover, \begin{align*} [r_i]c'(-\qa_i^\vee)- c'(\qa_i^\vee)[e]&=c'(\qa_i^\vee)[r_i]- c'(\qa_i^\vee)[e]\\
&=\qs_iH_i-(b'(\qa_i^\vee)+c'(\qa_i^\vee))[e]\\
&=\qs_iH_i-\qs_i^2[e] \quad  (\text{by }(\ref{5.5.2}) \text { and } (\ref{eq:5.5.2})).\end{align*}
So, by Lemma \ref{5.7.02}, as $P_\qs(\qs_iH_i-\qs_i^2[e])=0$, we get ${^{r_i}\!\Gamma} c'(-\qa_i^\vee)=\Gamma  c'(\qa_i^\vee)$, and we have: 
$$P_\qs \Delta^{-1}= \sum_{w\in W^v}[w]\Gamma \Delta^{-1}=\sum_{w\in W^v}  {^w\!(\Gamma} \Delta^{-1}) [w]=\sum_{w\in W^v}(\Gamma \Delta^{-1}) [w]=(\Gamma \Delta^{-1}) \sum_{w\in W^v}[w].$$
\end{proof}

\section{Well-definedness of the symmetrizers as operators}\label{se6}

In this section, we introduce the images of $e^\ql$ (for $\ql\in Y^{++}$) by the symmetrizers $P_\qs$ and $\SHJ_\qs$.

\subsection{Definition and  expression of $P_\qs (e^{\ql})$ for  $\ql\in Y^{++}$}\label{5.8} 

We would like to consider the symmetrizer $P_\qs$ as an operator.
At first, we
 define and give a simple expression of $P_\qs (e^{\ql})$ for  $\ql\in Y^{++}$.
 As already mentioned, we must take care that $ \Z_\qs ((Y))[[W^v]]$ is not an algebra, we have to justify that the expression $P_\qs (e^{\ql})$ is well defined,
actually in some completion $\widehat\Z_\qs[[Y,-Q_+^\vee]]$ of $Z_\qs[[Y,-Q_+^\vee]]$.
 
 \par 
 Recall that $(W^v)^\ql \subset W^v$ is the set of minimal length representatives of $W^v/W^v_{\ql}$.
 Contrary to the affine case treated in \cite{BrKP15}, both  sets $W^v_{\ql}$ and $(W^v)^\ql$ can be infinite for the same $\ql$. Hence, there are some well-definedness issues associated to both sets.
 
\medskip
\parni{\bf 1)}
To be more precise, we first consider in $ \Z_\qs ((Y))[[W^v]]$, the two elements 
\begin{align} \label {5.8.1}
P_{\qs,\ql}=\sum_{w\in W^v_\ql} \qs_w H_w \\
 \label {5.8.2}P_{\qs}^{\ql}=\sum_{w\in (W^v)^\ql} \qs_w H_w .
\end{align}
The same proof as in the case of  $P_\qs\in \Z_\qs ((Y))[[W^v]]$, with $P^N_{\qs,\ql}=\sum_{w\in W^v_\ql\cap W^v_N} \qs_w H_w $ and 
$P^{\ql\, N}_{\qs}=\sum_{w\in (W^v)^\ql\cap W^v_N} \qs_w H_w $ instead of $P_\qs ^N$, enables us to see that these elements are well-defined, as the main argument is based on (\ref{5.6.1}). 

\smallskip
For all $(w_1,w_2)\in (W^v)^\ql\times  W^v_{\ql} $, we have $\ell (w_1w_2)=\ell(w_1)+\ell(w_2)$, and moreover every $w\in W^v$ has a unique factorization $w=w^\ql w_\ql$ with $(w^\ql, w_\ql)\in  (W^v)^\ql\times  W^v_{\ql} $ so
 \begin{align*}P^{\ql\, N}_{\qs}P^N_{\qs,\ql}&=\Big(\sum_{w_1\in (W^v)^\ql\cap W^v_N} \qs_{w_1}H_{w_1}\Big)\Big(\sum_{w_2\in W^v_\ql\cap W^v_N} \qs_{w_2} H_{w_2} \Big)\\
 &=\sum_{w_1\in (W^v)^\ql\cap W^v_N} \Bigg(\sum_{w_2\in W^v_\ql\cap W^v_N} \qs_{w_1w_2}H_{w_1w_2}\Bigg)\\
 &=P_\qs ^N+\sum_{w_1\in (W^v)^\ql\cap W^v_N} \Bigg(\sum_{\substack{w_2\in W^v_\ql \cap W^v_N\\ \ell(w_1w_2)>N} } \qs_{w_1w_2}H_{w_1w_2}\Bigg).
 \end{align*}
 
 Let us consider $v\in W^v$ and  $\ql \in Q_+^\vee$, we know by (\ref{5.6.3}) that $C_{v, \ql, e}=C^N_{v, \ql, e}$ for $N$ large enough.
 Moreover, using (\ref{5.6.1}) with $w_0=e$,     in the equality above, we can see that, if $N$ is such that $2 \HT(\ql)< N-\ell (v)$, then the coefficient of $e^{-\ql}[v]$ in $\sum_{w_1} (\sum_{w_2} \qs_{w_1w_2}H_{w_1w_2})$  (with always $w_1\in (W^v)^\ql\cap W^v_N,w_2\in W^{v}_\ql\cap W^v_N, \,\ell(w_1w_2)>N$) seen in $\Z_\qs[[Y, -Q_+^\vee]]$ is 0. It follows that, in $\Z_\qs ((Y))[[W^v]]$, we have the equality
 \begin{equation}\label{5.8.3} P_\qs=P^\ql_\qs P_{\qs, \ql}\, .\end{equation}

In the same way as in the proof of the fact that $P_\qs$ is well-defined in $ \Z_\qs ((Y))[[W^v]]$, we want to see  that in the computation of 

\begin{align*}P_\qs^{\ql N}(e^\ql)&=\Big(\sum_{w\in (W^v)^\ql\cap W^v_N} \qs_w H_w \Big)(e^\ql)\\
 &=\sum_{w\in (W^v)^\ql\cap W^v_N} \qs_w \sum_{v\leq_{_{B}} w}\sum_{\qm\in Q_+^\vee}k^w_{v,\qm,w_0} e^{-\qm+v\ql}
\quad \text { by }\,\,  (\ref{5.6.1}) \end{align*}
the number of terms which give the same $e^\Lambda$ with $\Lambda\in \ql- Q_+^\vee$ is finite and does not depend on $N$ for $N$ large enough. This will give a sense to the expression $P_\qs^{\ql N}(e^\ql)$ and its limit $P_\qs^{\ql}(e^\ql)$ when $N\to \infty$.

\subsection{Computation of $H_w(e^\ql)$}

We consider, for $w\in W^v$, $\ell_\ql (w):=\ell(w^\ql)$, where $w=w^\ql w_\ql$ is the factorization above. 

\smallskip
In the following lemma we give a formula analogous to (\ref{5.6.1}) for the calculus of $H_w(e^\ql)$. 
We will use the same kind of proof as in Proposition \ref{5.6}.2.  If we consider the case of $w\in (W^v)^\ql$ in this proposition, then the gallery $\Gamma^0$ is stretched between $C_f$ and  $w(C_f)$ which is the nearest (for $d$) chamber from $C^v_f$ among the chambers of $\A$ containing the segment  germ $w([0,\ql))$; so $\ell_\ql(w)=\ell(w)= d(C^v_f,w(C^v_f))$.

More generally, for $v\in W^v$  with  reduced decomposition $v=r_{i_1}r_{i_2}\ldots r_{i_n}$ such that $v^\ql= r_{i_1}r_{i_2}\ldots r_{i_k}$ and $v_\ql= r_{i_{k+1}}\ldots r_{i_n}$, then for the gallery   $\Gamma$ stretched  from $C^v_f$ to $vC^v_f$ of length  $n$ and of  type $( {i_1},\ldots ,{i_n})$, we have $\ell_\ql (v)=k=d(C^v_f, v^\ql (C^v_f))=d(C^v_f, v([0,\ql)))$ where $d(C^v_f, v([0,\ql)))=\displaystyle{\min_{\substack{v([0,\ql))\subset C \\C\subset \A} } d(C^v_f, C)}$, the minimum over all chambers $C$ containing $v([0,\ql))$.

We keep below the same notation as in \ref{5.5.02} and Proposition \ref{5.6}.2, for $w\in W^{v\ql}$.

\begin{lemm*}

 $2\HT \Bigl ( \displaystyle  \sum_{\substack{k\in J \\ v_{k-1}^J(\qa_{i_k})\in Q_+}} {v^J_{k-1}({\qa_{i_k}^\vee})}\Bigr )\geq\ell_\ql(w)- \ell_\ql(w^J).$ 

\end{lemm*}

\begin{proof} 

 By considering $\ell_\ql$ instead of $\ell$ in the proof  of Proposition  \ref{5.6}.2, we get,
 for all ${k\in [0,\ldots , n]}$, 
$$\ell_\ql (w_k^J)=d(C_f^v, w_k^J([0,\ql))) \quad \text {and }\quad 
\ell_\ql (w)-\ell_\ql (w^J)=\displaystyle {\sum_{k\in J}  \ell_\ql (w_{k-1}^J)-\ell_\ql (w_k^J)}$$

If $k\in J$ and  $v_{k-1}^J(\qa_{i_k})\in Q_-$, then $\ell_\ql (w_{k-1}^J)\leq \ell_\ql (w_k^J)$. Indeed, if we consider a minimal gallery from $C_f^v$ to $w_k^J([0,\ql))$ (\ie to the nearest chamber from $C_f^v$, among the chambers containing $w_k^J([0,\ql)$), then this gallery meets or crosses the wall $\ker v_{k-1}^J(\qa_{i_k})$ because 
$v_{k-1}^J(\qa_{i_k})(C_f^v)<0$ and $v_{k-1}^J(\qa_{i_k})(w_k^J(\ql))=(r_{i_n}\cdots r_{i_{k+1}}(\qa_{i_k}))(\ql)\geq 0$. Moreover  $w_{k-1}^J([0,\ql))=r_{v_{k-1}^J(\qa_{i_k})}w_{k}^J([0,\ql))$. So if the gallery crosses the wall then $\ell_\ql (w_{k-1}^J)< \ell_\ql (w_k^J)$. If it does not cross the wall, then  $v_{k-1}^J(\qa_{i_k})(w_k^J(\ql))=0$ and $w_{k-1}^J([0,\ql))=w_{k-1}^J([0,\ql))$  and  $\ell_\ql (w_{k-1}^J)= \ell_\ql (w_k^J)$.
So, in both cases  $\ell_\ql (w_{k-1}^J)\leq\ell_\ql (w_k^J)$. 

With this result, we obtain the majoration 
$$ \ell_\ql (w)-\ell_\ql (w^J)\leq \displaystyle  \sum_{\substack{k\in J \\ v_{k-1}^J(\qa_{i_k})\in Q_+}} \ell _\ql(w_{k-1}^J)-\ell_\ql (w_k^J).$$

If $k\in J$ and   $v_{k-1}^J(\qa_{i_k})\in Q_+$, we consider $C_1$ (\resp $C_2$) the nearest chamber from $C^v_f$ (\resp $r_{v_{k-1}^J(\qa_{i_k})}(C_f^v)$) among the chambers containing the segment germ $w_{k-1}^J([0,\ql))$.
Then we get,

\begin{align*}\ell _\ql(w_{k-1}^J)-\ell_\ql (w_k^J)&=d (C_f^v,  w_{k-1}^J([0,\ql))) -d (C_f^v, w_k^J([0,\ql))) \\
                                                                      &=d (C_f^v, w_{k-1}^J([0,\ql)))-d ( C_f^v,r_{v_{k-1}^J(\qa_{i_k})}w_{k-1}^J([0,\ql)))\\
                                                                      &=d (C_f^v, w_{k-1}^J([0,\ql)))-d ( r_{v_{k-1}^J(\qa_{i_k})}(C_f^v),w_{k-1}^J([0,\ql)))\\
                                                                     &=d (C_f^v, C_1)-d ( r_{v_{k-1}^J(\qa_{i_k})}(C_f^v),C_2)\\
                                                                        &\leq d (C_f^v,  r_{v_{k-1}^J(\qa_{i_k})}(C_f^v))                                   \end{align*}
as $d(C^v_f,C_1)\leq d(C^v_f,C_2) \leq d(C^v_f, r_{v_{k-1}^J(\qa_{i_k})}(C_f^v)) + d ( r_{v_{k-1}^J(\qa_{i_k})}(C_f^v),C_2)$, by the triangle inequality for $d$.

  Finally we conclude as in the end of the proof of Proposition of \ref{5.6}.2.
\end{proof}  

\subsection{Well-definedness of $P_{\qs}(e^\ql)$}

First, we get now the  expected result on $P_{\qs}^{\ql}(e^\ql)$.

\begin{prop*}
 $P_{\qs}^{\ql}(e^\ql) =\displaystyle {lim_{N\to +\infty}}P_{\qs}^{\ql N}(e^\ql)$ is well defined in $ \Z_\qs [[Y, -Q_+^\vee]]$, and in $ \Z'_\qs [[Y, -Q_+^\vee]]$.

\end{prop*}

\begin{NB} The fact that $P_\qs^\ql(e^\ql)$ is in $\Z_\qs[[Y,-Q_+^\vee]]$ is a key result for section \ref{s6} below.
The fact that $P_\qs^\ql(e^\ql)$ is well defined in $\widehat\Z_\qs[[Y,-Q_+^\vee]]$ (see 3) below) is easier to get: use (\ref{5.8.3}), Proposition \ref{5.6.9} and the same result for $\sum_{w\in W^v}{^w\QD}e^{w\ql}$ (see below \ref{5.9} and \ref{5.10}).
\end{NB}

\begin{proof}

 Let us remark that $\{w^J(\ql)\, | J\subset \{1,\ldots , n\} \}= \{v(\ql) | v\in (W^v)^\ql\,\,   v\leq_{_{B}} w\}$, indeed $w^J(\ql)=w^{J\, \ql}(\ql)$, with $w^{J\, \ql}\leq_{_{B}}w^J\leq_{_{B}}w$ and  otherwise $ v\leq_{_{B}} w$ implies $ v=w^J$ for some $ J\subset \{1, \ldots , n\}$ .
Using  (\ref{5.6.00}), we have 

\begin{align*} H_w(e^\ql)&=\sum_{J\subset [1,\ldots, n]} \biggl (\prod_{k=1}^{k=n}{ d^J_{k}} \Bigl ( v^J_{k-1}({\qa_{i_k}^\vee})\Bigr )\biggr) [w^J] (e^\ql)\\
&=\sum_{\substack{ {v\in (W^v)^\ql}\\ {v\leq_{_{B}}  w}}  }\Biggl ( \sum_{\substack {{J\subset [1,\ldots, n]} \\ {w^J(\ql )=v(\ql)}}} \Bigl (\prod_{k=1}^{k=n}{ d^J_{k}} \Bigl ( v^J_{k-1}({\qa_{i_k}^\vee})\Bigr )\Bigr ) e^{v(\ql) }\Biggr ).\end{align*}

{For $J$ such that $w^J(\ql)=v(\ql)$, we know by (\ref{5.6.0}) that
$  e^{-\sum_k {v^J_{k-1}({\qa_{i_k}^\vee})}}\Z_\qs[[-Q_+^\vee]]$ contains 
$ \displaystyle \prod_{k=1}^{k=n}{d^J_{k}} \Bigl ( v^J_{k-1}({\qa_{i_k}^\vee})\Bigr )$,} 
where the sum runs over the $k\in J$ such that  $v^J_{k-1}({\qa_{i_k}^\vee}) \in Q_+^\vee$. Moreover, the above lemma tells that the height of this sum is at most $(\ell_\ql(w)-\ell_\ql(v))/2$. 

So, we can rewrite the last equality into:

$$H_w(e^\ql)=:\sum_{\substack{ {v\in (W^v)^\ql}\\ {v\leq_{_{B}} w}}  }\biggl ( \sum_{\substack{\qn\in Q_+^\vee \\ 2 \HT(\qn)\geq \ell_\ql(w)-\ell_\ql(v)}}k^w_{v,\qn} e^{-\qn+v(\ql)}\biggr )\quad \text {for some } k^w_{v,\qn}\in \Z_\qs.
$$
This expression enables us to obtain the expression of  $P_\qs^{\ql\, N}(e^\ql)$, \begin{align*}P_\qs^{\ql\, N}(e^\ql)&=\sum_{w\in (W^v)^\ql\cap W^v_N} \qs_w H_w (e^\ql)\\
&=\sum_{v\in (W^v)^\ql\cap W^v_N} \sum_{\substack{{w\in (W^v)^\ql}\\ {v\leq_{_{B}} w}}  } \qs_w\biggl ( \sum_{\substack{\qn\in Q_+^\vee \\ 2 \HT(\qn)\geq \ell_\ql(w)-\ell_\ql(v)}}k^w_{v,\qn} e^{-\qn+v(\ql)}\biggr ).
\end{align*}

As $\ql\in Y^{++}$, for any $v\in W^v$, $\ql-v(\ql)\in Q_+^\vee$. 
Moreover, if $\Lambda \in \ql-Q_+^\vee$, the set $(\ql-Q_+^\vee)\cap ( \Lambda+Q_+^\vee)$ is finite, so the number of $v\in (W^v)^\ql$ such that $v(\ql) -\Lambda\in Q_+^\vee$ is finite and, for $v$ in this set, we know that the $w$ which have a contribution to the term in $e^\Lambda$ are in $W^{v \, \ql}\cap W^v_{2ht(v(\ql)-\Lambda)+\ell (v)}$. So, for $N$ large enough, the coefficient of $e^\Lambda$ in  $P_\qs^{\ql\, N}(e^\ql)$ is independent of $N$, and we get the well-definedness of $P_{\qs}^{\ql}(e^\ql) $ in $ \Z_\qs [[Y, -Q_+^\vee]]$. 

\par The reader will check by himself that $P_\qs^\ql(e^\ql)\in \Z'_\qs[[Y,-Q_+^\vee]]$.
\end{proof}

\medskip
\parni
Let us introduce now  the completion $\hat\Z_\qs$ of $\Z_\qs$ (\resp $\hat\Z'_\qs$ of $\Z'_\qs$), as the ring of Laurent series in $\qs_i$ ($i\in I$) with coefficients  polynomials in $  {\qs'_{i}}^{\pm 1}$ (\resp the ring of series in $\qs_i\qs'_i$ (for $i\in I$) and $\qs_i(\qs'_i)^{-1}$ {(for $i\in E=\{i\in I \mid \qs_i\neq\qs'_i\}$))} : 

 \begin{equation}\label{5.8.6}  \hat{\Z}_\qs:=\Z[ ({\qs'_{i}}^{\pm 1})_{ i\in E} ](((\qs_{i})_{ i\in I} )),\end{equation}  
 
 \begin{equation}\label{5.8.6a} {\hat{\Z}'_\qs:=\Z[[ ({\qs_i\qs'_{i}})_{ i\in I}, ({\qs_i\qs'_{i}}^{- 1})_{ i\in E} ]]\subset\Z[({\qs'_{i}}^{\pm 1})_{ i\in E} ][[(\qs_{i})_{ i\in I} ]]}.\end{equation}  

For $W'$ any subgroup of $W^v$, the following Poincaré series  is defined and invertible in $\hat{\Z}'_\qs\subset\hat{\Z}_\qs$:
\begin{equation} \label{5.8.7}W'(\qs^2)= \sum_{w\in W'}\qs^2_w.
\end{equation} 
We consider $\hat  \Z_\qs[[Y,-Q^\vee_+]]$ (\resp $\hat  \Z'_\qs[[Y,-Q^\vee_+]]$) as the set consisting of all infinite formal sums $f=\sum_{y\in Y} a_ye^y$, with $a_y\in  \hat\Z_\qs$ (\resp $a_y\in  \hat\Z'_\qs$) such that $supp(f):=\{ y\in Y\, |\, a_y\not= 0\} \subset \bigcup_{i=1}^n\,(\ql_i- Q_+^\vee)$ for some $\ql_1, \ldots, \ql_n\in Y$.

\begin{prop}\label{p6.4}  The element $P_{\qs}(e^\ql)$ is a well-defined element of $\hat  \Z_\qs[[Y,-Q^\vee_+]]$ and of $\hat  \Z'_\qs[[Y,-Q^\vee_+]]$ and we have 

$$P_{\qs}(e^\ql)=W^v_{\ql}(\qs^2)P_{\qs}^{\ql}(e^\ql) \text{with } W^v_{\ql}(\qs^2)\text { the Poincaré series of }   W^v_{\ql}.$$
\end{prop}
\begin{proof}

If $\qa_i(\ql)=0$, then $H_i(e^\ql)= (c(\qa_i^\vee)+b(\qa_i^\vee))e^\ql=\qs_ie^\ql$, so, for  any $w\in W^v_\ql$, we have $H_w(e^\ql)=\qs_w e^\ql$.
Therefore we obtain the equality 
$$P_{\qs,\ql}^N(e^\ql)=\sum_{w\in W^v_{\ql}\cap W^v_N} \qs_w H_w(e^\ql)=\bigg(\sum_{w\in W^v_\ql \cap W^v_N} \qs_w^2\bigg)e^\ql.
$$
We can consider $P^{\ql\, N}P_{\qs,\ql}^N(e^\ql)=(\sum_{w\in W^v_\ql \cap W^v_N} \qs_w^2)P^{\ql\, N}(e^\ql)$ and we know, by the previous lemma and the arguments  used to obtain (\ref{5.8.3}) that, for a given $\Lambda\in Y$, and $M\in \N$, if $N$ is large enough all the terms of total degree less than M in the  $\qs_i \ ({i\in I})$ of the coefficient of  $e^\Lambda$ are independent of $N$. 

So we obtain simultaneously the fact that $P_{\qs}(e^\ql)$ is a well-defined element of  $\hat  \Z_\qs[[Y,-Q^\vee_+]]$ and of $\hat  \Z'_\qs[[Y,-Q^\vee_+]]$, and also the expected equality. 
\end{proof}

\subsection{Definition and expression of $\SHJ_\qs(e^\ql)$, for $\ql\in Y^{++}$} \label{5.9} 

\parni{\bf 1)} As in \ref{5.8}, we want to show that $\SHJ_\qs(e^\ql)=\sum_{w\in W^v}\, {^w\QD}.([w]e^\ql)=\sum_{w\in W^v}\, {^w\QD}.e^{w\ql}$ is well defined in $\hat\Z_\qs[[-Q_+^\vee]]$ (actually in $\hat\Z'_\qs[[-Q_+^\vee]]$).
Of course $\SHJ_\qs(e^\ql)=\sum_{w\in W^{v\ql}}\, (\sum_{v\in W^v_\ql}{^{wv}\QD}).e^{w\ql}$.
So we have to prove that $\sum_{v\in W^v_\ql}{^{wv}\QD}$ converges in $\hat\Z'_\qs[[-Q_+^\vee]]$ for any $w\in W^v$.
This is equivalent to prove that $\sum_{v\in W^v_\ql}\frac{^{wv}\QD}{\QD}$ converges in $\hat\Z'_\qs[[-Q_+^\vee]]$.
The ideas for this proof come from \cite[p. 199]{Ma03b}.

\medskip
\parni{\bf 2)} Let us consider, for $u\in W^v$, the quotient $\frac{^u\QD}{\QD}$.
We define $\QF(u^{-1})=\QF^+\cap u\QF^-$ which is finite of cardinality $\ell(u)$.
So $\QF^+=(\QF^+\cap u\QF^+)\sqcup \QF(u^{-1})$, $u\QF^+=(\QF^+\cap u\QF^+)\sqcup -\QF(u^{-1})$ and $\frac{^u\QD}{\QD}=\prod_{\qa\in\QF(u^{-1})}\,\frac{\QD_{-\qa}}{\QD_{\qa}}$.

\par Now, for $\qa\in\QF^+$, we get from (\ref{5.5.6}) and (\ref{5.5.7}), that $\qs_\qa c(\qa^\vee)\in\qs_\qa^2+e^{-\qa^\vee}\Z'_\qs[[e^{-2\qa^\vee}]]$, $\qs_\qa c(-\qa^\vee)\in1+e^{-\qa^\vee}\Z'_\qs[[e^{-2\qa^\vee}]]$, hence 
$$\frac{\QD_{-\qa}}{\QD_{\qa}}=\frac{\qs_\qa c(\qa^\vee)}{\qs_\qa c(-\qa^\vee)}\in\qs_\qa^2+e^{-\qa^\vee}\Z'_\qs[[e^{-\qa^\vee}]].
$$
So, we may write $\frac{\QD_{-\qa}}{\QD_{\qa}}=\sum_{r\geq0}\,z_{r,\qa}e^{-r\qa^\vee}$, with $z_{r,\qa}\in\Z'_\qs$ and $z_{0,\qa}=\qs_\qa^2$.

\smallskip
\par Hence $\frac{^u\QD}{\QD}=\sum_{\qm\in Q_+^\vee}\,z_{u,\qm}e^{-\qm}$ with $z_{u,\qm}=\sum(\prod_{\qa\in\QF(u^{-1})}\,z_{r_\qa,\qa})\in \Z'_\qs$ and the sum runs over the families $(r_\qa)_{\qa\in\QF(u^{-1})}$ of integers such that $\sum_{\qa\in\QF(u^{-1})} r_\qa\qa^\vee=\qm$.

\begin{enonce*}[plain]{3) Lemma} Let $\mathfrak p$ be the ideal of $\Z'_\qs$ generated by the $\qs_i\qs'_i$ (for $i\in I$) and $\qs_i(\qs'_i)^{-1}$ (for $i\in I$ with $\qs_i\neq\qs'_i$) ).
Suppose $\ell(u)=ht(\qm)+m$ with $m\geq1$.
Then $z_{u,\qm}\in \mathfrak p^{m}$.
\end{enonce*}

\begin{proof} \cite{Ma03b} Actually $\#\QF(u^{-1})=\ell(u)=ht(\qm)+m$.
So, if $\sum_{\qa\in\QF(u^{-1})} r_\qa\qa^\vee=\qm$, at least $m$ of the $r_\qa$ are equal to $0$.
So at least $m$ of the $z_{r_\qa,\qa}$ are equal to $\qs_\qa^2$ and $\prod_{\qa\in\QF(u^{-1})}\,z_{r_\qa,\qa}\in \mathfrak p^{m}$.
\end{proof}

\parni{\bf 4)} Now $\sum_{u\in wW^v_\ql}\frac{^{u}\QD}{\QD}$ is equal, at least formally, to $\sum_{\qm\in Q_+^\vee}(\sum_{u\in wW^v_\ql} z_{u,\qm})e^{-\qm}$.
 The infinite sum $z^w=\sum_{u\in wW^v_\ql} z_{u,\qm}$ may be cut into infinitely many finite sums $z_m^w\in\Z'_\qs$, where for the sum $z_0^w$ (\resp $z_m^w$ with $m\geq1$) we consider the $u\in wW^v_\ql$ such that $\ell(u)\leq ht(\qm)$ (\resp $\ell(u) = ht(\qm)+m$). By the above Lemma $z_m^w\in \mathfrak p^{m}$, so the series $z^w=z_0^w+\sum_{m\geq1}z_m^w$ converges in $\Z'_\qs$.
 
 \par We have proved that $\sum_{u\in wW^v_\ql}\frac{^{u}\QD}{\QD}\in \hat\Z'_\qs[[-Q_+^\vee]]$, hence 
 
 \begin{equation} \label{5.9.1}
\sum_{v\in W^v_\ql} {^{wv}\QD} \in \hat\Z'_\qs[[-Q_+^\vee]].
 \end{equation} 
 
 \begin{NB} The particular case $w=e,\ql=0$ is interesting: $\sum_{v\in W^v} {^{v}\QD} \in \hat\Z'_\qs[[-Q_+^\vee]]$.
 \end{NB}
 
\begin{enonce*}[plain]{5) Proposition} $\SHJ_\qs(e^\ql)=\sum_{w\in W^v}\, {^w\QD}.e^{w\ql}$ is well defined in  $\hat\Z'_\qs[[-Q_+^\vee]]\subset \hat\Z_\qs[[-Q_+^\vee]]$.
\end{enonce*}

\begin{proof} For $\ql\in Y^{++}$ and $w\in W^v$, $w\ql\in\ql-Q_+^\vee$. So the proposition follows from (\ref{5.9.1}) and 1) above.
\end{proof}

\subsection{Proportionality of $P_\qs(e^\ql)$ and $\SHJ_\qs(e^\ql)$}\label{5.10}

\par We saw in \ref{5.6.9}, that $P_\qs=\mathfrak m_\qs.\!\SHJ_\qs$, with $\mathfrak m_\qs\in \Z_\qs((Y))\subset \Z_\qs[[Y,-Q_+^\vee]]\subset\hat \Z_\qs[[Y,-Q_+^\vee]]$ and $P_\qs=\sum_{w\in W^v}\qs_wH_w$, $\SHJ_\qs=\sum_{w\in W^v}\,{^w\QD}[w]$.
We proved in \ref{5.8} and \ref{5.9} that the series $P_\qs(e^\ql)=\sum_{w\in W^v}\qs_wH_w(e^\ql)$ and $\SHJ_\qs(e^\ql)=\sum_{w\in W^v}\,{^w\QD}e^{w\ql}$ converge in $\hat \Z_\qs[[Y,-Q_+^\vee]]$.
So, using moreover Proposition {\ref{p6.4}}, we have:

 \begin{equation} \label{5.10.1}
W^v_\ql(\qs^2).P^\ql_\qs(e^\ql)=P_\qs(e^\ql)=\mathfrak m_\qs.\!\SHJ_\qs(e^\ql) \in \hat\Z_\qs[[Y,-Q_+^\vee]].
 \end{equation} 

 \goodbreak


\section{Macdonald's formula}\label{s6}

\par In this short section we harvest the fruits of the union of the results we got separately in Section \ref{s4} and in Sections \ref{s5}, \ref{s5b} and \ref{se6}.

\subsection{Specialization of the indeterminates $\qs_i,\qs'_i$} \label{6.1} 

\par We consider a ring $R$ as in \ref{2.10}, in order to define  $\qd^{1/2}:Y\to R$.
So, for any $i\in I$, $\sqrt{q_iq'_i}\in R$, we ask more precisely that $\sqrt{q_i}^{\pm1}$ and $\sqrt{q'_i}^{\pm1}$ are in $R$.

\par We now consider the homomorphism $\Z_\qs\to R, \qs_i\mapsto \sqrt{q_i}^{-1}, \qs'_i\mapsto \sqrt{q'_i}^{-1}$.
 We get thus an homomorphism $\Z_\qs[[Y,-Q_+^\vee]]\to R[[Y]]=R[[Y,-Q_+^\vee]]=\Z_\qs[[Y,-Q_+^\vee]]\otimes_{\Z_\qs}R$; the image $f_{(\qs^2=q^{-1})}$ of an element $f$ is its specialization (at ``$\qs^2=q^{-1}$'').
 
 \begin{NB} This specialization may seem awkward, as we used the specialization $\qs_i\mapsto \sqrt{q_i}^{}, \qs'_i\mapsto \sqrt{q'_i}^{}$ to embed the Iwahori-Hecke algebra $^I\shh_\R$ into $^{BL}\shh_\R$.
 But actually the Iwahori-Hecke algebra has no role in the present work (see also the last paragraph of \ref{2.14}).
 Moreover the Bernstein-Lusztig-Hecke algebra has involutions exchanging $\qs_i,\qs'_i$ with $\qs_i^{-1}, (\qs'_i)^{-1}$.
 Hence some different convention in sections \ref{s5}, \ref{s5b}{, \ref{se6} would perhaps} induce the use here of the specialization $\qs_i\mapsto \sqrt{q_i}^{}, \qs'_i\mapsto \sqrt{q'_i}^{}$.
 \end{NB}

\begin{prop}\label{6.2} Let $\ql\in Y^{++}$ and $w\in W^{v\ql}$. Then $J_w(\ql)$ (as defined in \ref{2.13}) is equal to the specialization of the element $\qd^{1/2}(\ql).\qs_w.H_w(e^\ql)$, \ie $$J_w(\ql)=\qd^{1/2}(\ql).(\qs_w.H_w(e^\ql))_{(\qs^2=q^{-1})}\, .$$
\end{prop}

\begin{proof} We argue by induction on $\ell(w)$.
If $w=e$, we have $J_e(\ql)=\qd^{1/2}(\ql).e^\ql$ by \ref{2.13}.3 and $\qs_eH_e(e^\ql)=e^\ql$, hence the formula.

Let us now consider $w,w' \in W^v$ such that $w\in W^{v\ql}$, $w=r_i.w'$ and $\ell(w)=\ell(w')+1$ (hence $w'\in W^{v\ql}$).
We saw in section \ref{s4}, that $J_w(\ql)=c_q(\qa_i^\vee).({^{r_i}J_{w'}(\ql)})+b_q(\qa_i^\vee).J_{w'}(\ql)$, with $b_q(\qa_i^\vee)=\sqrt{q_i}^{-1}b(\sqrt{q_i}^{-1},\sqrt{q'_i}^{-1};e^{\qa_i^\vee})$ and $c_q(\qa_i^\vee)=\sqrt{q_i}^{-1}-b_q(\qa_i^\vee)$ (hence $b_q(\qa_i^\vee)=b'(\qa_i^\vee)_{(\qs^2=q^{-1})}$ and $c_q(\qa_i^\vee)=c'(\qa_i^\vee)_{(\qs^2=q^{-1})}$).

\par On the other side, we want to prove the following recursive formula:
$$\qd^{1/2}(\ql)\qs_wH_w(e^\ql)=c'(\qa_i^\vee)\qd^{1/2}(\ql)\qs_{w'}(^{r_i}H_{w'}(e^\ql))+b'(\qa_i^\vee)\qd^{1/2}(\ql)\qs_{w'}(H_{w'}(e^\ql))$$

\parni \ie $\qs_iH_iH_{w'}(e^\ql)=c'(\qa_i^\vee)(^{r_i}H_{w'}(e^\ql))+b'(\qa_i^\vee)(H_{w'}(e^\ql))$.
But this is a clear consequence of $\qs_iH_i=c'(\qa_i^\vee)[{r_i}]+b'(\qa_i^\vee)[e]$ for the action on $\Z_\qs((Y))$.
The resursive formula for the specialization is then 
$\qd^{1/2}(\ql)(\qs_wH_w(e^\ql))_{(\qs^2=q^{-1})}=c_q(\qa_i^\vee)\qd^{1/2}(\ql)(\qs_{w'}(^{r_i}H_{w'}(e^\ql)))_{(\qs^2=q^{-1})}+b_q(\qa_i^\vee)\qd^{1/2}(\ql)(\qs_{w'}(H_{w'}(e^\ql)))_{(\qs^2=q^{-1})}$.
So the proposition is proved.
\end{proof}

\begin{theo}\label{6.3} Let $\ql\in Y^{++}$. Then the image $\shs(c_\ql)$ of $c_\ql$ by the Satake isomorphism is the specialization of $P_\qs^\ql(e^\ql)=\mathfrak m_\qs.(\SHJ_\qs(e^\ql)/W^v_\ql(\qs^2))$, multiplied by $\qd^{1/2}(\ql)$.
More precisely let $H_\ql:=\frac{\SHJ_\qs(e^\ql)}{W^v_\ql(\qs^2)}=\frac{\sum_{w\in W^v}\,{^w\QD}.e^{w\ql}}{W^v_\ql(\qs^2)}\in \Z_\qs[[Y,-Q_+^\vee]]{=\Z_\qs[[Y]]}$, then $\mathfrak m_\qs=H_0^{-1}=\frac{W^v(\qs^2)}{\sum_{w\in W^v}\,{^w\QD}}$ is in $\Z_\qs[[Y,-Q_+^\vee]]$ and we may write:

\begin{equation}\label{6.3.1}
\shs(c_\ql)=\qd^{1/2}(\ql).\bigg(\frac{W^v(\qs^2)}{\sum_{w\in W^v}\,{^w\QD}}\bigg)_{(\qs^2=q^{-1})}  .  \bigg(\frac{\sum_{w\in W^v}\,{^w\QD}.e^{w\ql}}{W^v_\ql(\qs^2)}\bigg)_{(\qs^2=q^{-1})} . 
\end{equation}
\end{theo}

\begin{NB} The elements $W^v_\ql(\qs^2)$, $W^v(\qs^2)$, $\sum_{w\in W^v}\,{^w\QD}.e^{w\ql}$ and $\sum_{w\in W^v}\,{^w\QD}$ are in general not in $\Z_\qs[[Y,-Q_+^\vee]]$: they are only in $\hat\Z_\qs[[Y,-Q_+^\vee]]$.
So one cannot specialize them. Only the above written quotients can be specialized.
\end{NB}

\begin{proof} We proved in \ref{2.13}.2 (\resp \ref{5.8}.3) the convergence of the series $\shs(c_\ql)=\sum_{w\in W^{v\ql}}\,J_w(\ql)$ in $R[[Y]]$ (resp $P_\qs^\ql(e^\ql)=\sum_{w\in W^{v\ql}}\,\qs_wH_w(e^\ql)$ in $\Z_\qs[[Y,-Q_+^\vee]]$).
So, from Proposition \ref{6.2}, we get $\shs(c_\ql)=\qd^{1/2}(\ql).(P_\qs^\ql(e^\ql))_{(\qs^2=q^{-1})}$.
From \ref{5.10}, we get $P_\qs^\ql(e^\ql)=\mathfrak m_\qs.(\SHJ_\qs(e^\ql)/W^v_\ql(\qs^2))$ in $\hat\Z_\qs[[Y,-Q_+^\vee]]$.
But $\mathfrak m_\qs$ is invertible in $\Z_\qs((Y))\subset \Z_\qs[[Y,-Q_+^\vee]]$, so $\SHJ_\qs(e^\ql)/W^v_\ql(\qs^2)=:H_\ql \in \Z_\qs[[Y,-Q_+^\vee]]$.

\par We consider the expression $P_\qs^\ql(e^\ql)$ when $\ql=0$ (hence $W^v_\ql=W^v$ and $W^{v\ql}=\{e\}$).
We get $1=P_\qs^0(e^0)=\mathfrak m_\qs.H_0$, so $\mathfrak m_\qs=H_0^{-1}=\frac{W^v(\qs^2)}{\sum_{w\in W^v}\,{^w\QD}}$ and (\ref{6.3.1}) is proved.
\end{proof}

\begin{remas}\label{6.4} 1) In the case of a split affine untwisted Kac-Moody group, this theorem is one of the main results of Braverman, Kazhdan and Patnaik in \cite{BrKP15}.

\par 2) Let us consider now the case where all $q_i,q'_i$ are equal to a same $q\in\N$. Accordingly we consider the homomorphism $\Z_\qs\to \Z[\qs^{\pm1}]$ sending each $\qs_i,\qs'_i$ to the indeterminate $\qs$ (and $\Z'_\qs$ to $\Z[\qs^2]$).
We consider now all the above algebras over $\Z[\qs^{\pm1}]$ by extension of scalars.
 In particular now $b'(\qa^\vee)=\frac{\qs^2-1}{1-e^{-\qa^\vee}}$ and $c'(\qa^\vee)=1-b'(\qa^\vee)=\frac{1-\qs^2e^{-\qa^\vee}}{1-e^{-\qa^\vee}}$.
 
 \par We may modify $\QD$ by adding a factor $\QD_{im}$: $\QD_{all}=\QD.\QD_{im}$ with 
 $$\QD_{im}=\prod_{\qa\in\QF_{im}^+}\,\Big(\frac{1-\qs^2e^{-\qa^\vee}}{1-e^{-\qa^\vee}}\Big)^{m(\qa^\vee)},
 $$ 
 where $m(\qa^\vee)$ is some ``multiplicity'' assumed invariant under $W^v$.
 So $^w\QD_{im}=\QD_{im}$, for any $w\in W^v$.
 Now, if we replace in the theorem $\QD$ by $\QD_{all}$ and compute the corresponding $H_\ql^{all}$, then we get $H_\ql^{all}=\QD_{im}.H_\ql$.
 In particular the Macdonald's formula (\ref{6.3.1}) remains true with $\QD_{all}$: the factor $\QD_{im}$ disappears.
 
 \par For the split affine case, one may choose $m(\qa^\vee)$ equal to the multiplicity of $\qa^\vee$ in the corresponding Kac-Moody algebra.
 Then there is a more explicit formula for $H_0^{all}$ (as a product), related to the solution of the Macdonald's constant term conjecture by Cherednik.
 The formula is reasonably simple when $\QF_{all}$ is affine and simply laced (hence untwisted), see \cite{Ma03}, \cite{Ma03b}, \cite{BrKP15} or \cite{PaP17}.
 
 \par We do not know such a product formula for $H_0^{all}$ in the general Kac-Moody case.
 
\end{remas}


\medskip

\noindent Universit\'e de Lyon, Institut Camille Jordan (UMR 5208)\\
Universit\'e Jean Monnet, Saint-Etienne, F-42023, France

E-mail: Stephane.Gaussent@univ-st-etienne.fr
\medskip
\par
\noindent Universit\'e de Lorraine, Institut \'Elie Cartan de Lorraine, UMR 7502, and \\
CNRS, Institut \'Elie Cartan de Lorraine, UMR 7502,
\\Vand\oe uvre l\`es Nancy, F-54506, France

E-mail: Nicole.Panse@univ-lorraine.fr ; Guy.Rousseau@univ-lorraine.fr

\bigskip
\noindent {\bf Acknowledgement:} The second author acknowledges support of the ANR grants ANR-13-BS01-0001-01 and ANR-15-CE40-0012.

\end{document}